\documentclass[11pt,a4paper,reqno]{amsart}

 \usepackage{cases,hyperref} \usepackage{xcolor,appendix,geometry}
\usepackage{mathrsfs}
\usepackage{epsfig}
\usepackage{fp,ifthen}
\usepackage{float}
\usepackage{graphicx}
\usepackage{enumitem}
\usepackage{esint}
\usepackage{accents}
\usepackage{cite}
\usepackage{mathtools}
\usepackage{amsmath,amssymb,amsthm,esint,empheq,etoolbox}

\title[${ \Gamma}$-convergence of convolution-type functionals]{${\bf \Gamma}$-convergence of convolution-type functionals for free discontinuity problems}

\author[G.C.\! Brusca]{Giuseppe C. Brusca}
\author[D.\! Donati]{Davide Donati}
\author[S.\! Scalabrino]{Sergio Scalabrino}
\author[C.\!\! Trifone]{Chiara Trifone}
\author[E.\!\! Voglino]{Edoardo Voglino}
\newcommand{\m}{\mathfrak{m}}
\newcommand{\e}{\varepsilon}
\newcommand{\R}{\mathbb R}

\newcommand{\Rd}{\R^d}
 
\newcommand{\Rmd}{\R^{m\times d}}
\newcommand{\Rm}{\R^m}
\newcommand{\Sd}{\mathbb{S}^{d-1}}
\newcommand{\A}{\mathcal{A}}

\newcommand{\Ld}{\mathcal{L}^d}
\newcommand{\Hd}{\mathcal{H}^{d-1}}
\newcommand{\N}{\mathbb{N}}

\newcommand{\dx}{{\rm d}}
\mathchardef\emptyset="001F

\newcommand{\mres}{\mathbin{\vrule height 1.6ex depth 0pt width
0.13ex\vrule height 0.13ex depth 0pt width 1.3ex}}
\newcommand\restr[2]{{
  \left.\kern-\nulldelimiterspace 
  #1 \vphantom{\big|} 
  \right|_{#2} }}

\newcommand{\eps}{\varepsilon}

\newtheorem{Theorem}{Theorem}[section]
\newtheorem{Corollary}[Theorem]{Corollary}

\newtheorem{Proposition}[Theorem]{Proposition}
\newtheorem{Lemma}[Theorem]{Lemma}

\theoremstyle{definition}
\newtheorem{Remark}[Theorem]{Remark}

\numberwithin{equation}{section}

\let\phi\varphi
\begin{document}
\begin{abstract}
\smallskip
We prove compactness with respect to $\Gamma$-convergence  for a general class of non-local energies modelled after  the ones considered in [Gobbino, CPAM (1998)]. We give an integral representation result for the limits, which are free discontinuity functionals defined on the space of generalised special functions of bounded variation. We then characterise the bulk and surface energy densities of the obtained limits by means of minimisation problems on small cubes for the approximating energies.

\vspace{0.5 cm}

\end{abstract}

\maketitle

\section{Introduction}

The necessity of treating numerically 
minimisation problems related to free discontinuity functionals has stimulated the development of a wide literature of 
approximation results, which offer insights on the structure of minimisers to such problems. Indeed, although general compactness and lower semicontinuity theorems due to Ambrosio \cite{AmbrosioExistence,AmbrosioLower} (see also  \cite{FeldmanStinson,FriedrichCompactness}) allow to prove the existence of solutions to minimisation problems, very little can be inferred on the properties of the obtained minimisers relying on these arguments alone.  
 In this setting, the prototypical example is the {\it Mumford-Shah} functional, which for $\Omega\subset \Rd$ open and bounded and for $u\in {\rm GSBV}(\Omega)$ (see \cite{DeGiorgiAmbrosio} and \cite[Chapter~4]{AFP}) takes the form  
 \begin{equation}\label{intro Mumford Shah}
    \int_\Omega|\nabla u|^2\,{\rm d}x+\Hd({J_u}).
\end{equation}
Here, $\nabla u$ is the {\it approximate gradient} of $u$ and $\Hd(J_u)$ denotes the $(d-1)$-dimensional Hausdorff measure of the {\it jump set} $J_u$. 

 A natural approach is to approximate, in a variational sense, the  functionals under consideration. This can be achieved in a remarkably wide variety of fashions, all sharing a common feature: the replacement of the surface part of the energy, which is in general very hard to treat, by means of a computationally more favourable expression.

A first result in this direction  was obtained in \cite{Tortorelli} by Ambrosio and Tortorelli, who proposed to introduce an additional phase-field variable to approximate \eqref{intro Mumford Shah} by means of elliptic energies.  This result was later employed by Bellettini and Coscia \cite{BellettiniCoscia} to provide a discrete approximation of \eqref{intro Mumford Shah}, obtained by replacing the elliptic functionals considered in \cite{Tortorelli} with discrete sums. 
Other approximations involving discretisation procedures have then been considered, for instance, by Chambolle \cite{ChambolleImage} and Braides and Gelli \cite{BraGeConvex,BraidesGellMMS}. A different strategy was adopted by  Alicandro, Braides, and Gelli \cite{AliBraGelli} (see also \cite{AliGelli}) and more recently by Solci \cite{SolciInterp}, who used singular perturbations of integer order $k\geq 2$ of non-convex energies to approximate 
\begin{equation}\notag 
    \int_\Omega|\nabla u|^2\,{\rm d}x+\int_{J_u}|[u]|^{1/k}\,{\rm d}\Hd,
\end{equation}
where  $[u]:=u^+-u^-$ is the difference of the unilateral traces $u^+$ and $u^-$ of $u$ on $J_u$. 

Several non-local approximations of the Mumford-Shah functional have been proposed.    In \cite{BraidesDalMaso}  Braides and Dal Maso  considered functionals of the form
\begin{equation}\label{braides-dalmaso}
    \frac{1}{\e}\int_\Omega f\Bigl(\e\fint_{\Omega\cap B_\e(x)} |\nabla u(y)|^2\, \dx y\Bigr)\, \dx x, \qquad u\in H^1(\Omega),
\end{equation}
where $f$ is a non-negative increasing continuous function such that
\begin{equation*}
    \lim _{t\to 0^+} \frac{f(t)}{t}=a, \qquad \lim_{t\to+\infty} f(t)=b 
\end{equation*}
for some $a,b>0$; see also \cite{BraidesGarroni,CortesaniARMA,CortesaniToaderSIAM,LussardiVitali1,LussardiVitali2} for generalisations to the ${\rm BV}$ setting and \cite{Negri,MarzianiSolombrino,ScillaSolombrino} for related works in the setting of functions of {\it bounded deformation}.

 A different approach was pursued by Gobbino, who proved that \eqref{intro Mumford Shah} could be approximated by sequences of convolution-type functionals depending on difference quotients, as conjectured by De Giorgi in 1996. More precisely, in \cite{Gobbino} the author proves that, up to multiplicative constants and assuming some regularity on $\Omega$, the sequence of functionals defined by 
\begin{equation}\label{intro de giorgi}
    \frac{1}{\e^{d+1}}\int_{\Omega}\int_{\Omega}\arctan\Big(\frac{|u(x)-u(y)|^2}{\e}\Big)e^{-|\frac{x-y}{\e}|^2}\,{\rm d}x\,{\rm d}y, \qquad u\in L^1_{\rm loc}(\Omega),
\end{equation}
$\Gamma$-converges as $\e\to 0^+$  in the $L^1_{\rm loc}(\Omega)$ topology to the Mumford-Shah functional. This is obtained by a careful analysis of the one-dimensional case, in which the convolution kernel does not appear and \eqref{intro de giorgi} takes a simpler form. Such result is then extended to the $d$-dimensional case using slicing techniques. In higher dimensions, the convolution kernel is introduced in order to avoid the anisotropies that one would obtain if only increments with respect to the coordinate axes were considered. No particular role is played by the specific choice of the integrand and of the convolution kernel appearing in \eqref{intro de giorgi}, and the arguments of \cite{Gobbino} are suitable also to treat functionals given by
\begin{equation}\label{intro de giorgi 2}
    \frac{1}{\e^{d}}\int_{\Omega}\int_{\Omega}\min\Big\{\frac{|u(x)-u(y)|^2}{\e^2},\frac{1}{\e}\Big\}\rho\Big(\frac{|x-y|}{\e}\Big)\,{\rm d}x\,{\rm d}y,
\end{equation}
for $\rho\colon\R\to [0,+\infty)$ any smooth kernel with compact support (see \cite[Section~7]{Gobbino} and \cite[Chapter~5]{BraidesApproximation}). This result was later adapted by Chambolle \cite{ChambolleDiscretGobbino}, who proved a discrete version of De Giorgi's conjecture.

Gobbino and Mora in \cite{GobbinoMora} further discussed a variant of 
\eqref{intro de giorgi} in which more general convex-concave integrands $\phi_\e\colon [0,+\infty)\to[0,+\infty)$ are considered and slightly different scaling conditions are taken into account.  
The choice of the integrands $\{\phi_\e\}_\e$  determines the expression of a free discontinuity functional appearing in the limit, which is of the form 
\begin{equation}
\label{Gob-Mora fun}
    \int_\Omega f_{\rm bulk}(|\nabla u|)\,{\rm d}x+\int_{J_u}f_{\rm surf}(|[u]|)\,{\rm d}\Hd,
\end{equation}
for suitable integrands $f_{\rm bulk}$ and $f_{\rm surf}$, with $f_{\rm surf}$ that may also be superlinear near $0$.
Moreover, the authors prove that, for specific choices of $f_{\rm bulk}$ and $f_{\rm surf}$, the functional \eqref{Gob-Mora fun} can be approximated via non-local energies associated to appropriately chosen integrands $\{\phi_\e\}_\e$, for which they provide an explicit expression. 

A further variant of  Gobbino's result was recently considered in \cite{AlmiDavoliKT} by Almi, Davoli, Kubin, and Tasso, who studied  a vectorial analogue of \eqref{intro de giorgi} and showed its $\Gamma$-convergence to the Griffith functional, whose domain is a subspace of ${ \rm GSBD}(\Omega)$ (see \cite[Definition~4.2]{DalJems}) and whose expression can be formally obtained by replacing  in \eqref{intro Mumford Shah} the approximate gradient $\nabla u$ with the approximate symmetric gradient $\mathcal{E} u$.

The aim of this paper is to prove $\Gamma$-compactness and integral representation results for a general class of non-local functionals modelled after \eqref{intro de giorgi 2}. 
We study functionals defined for $u\in L^1_{\rm loc}(\Omega;\Rm)$ by
\begin{equation}\label{intro funzionali}
F_\e(u):=\frac{1}{\e^d}\int_{\Omega}\int_{\Omega}f_\e\Big(x,\frac{y-x}{\e},\frac{u(y)-u(x)}{\e}\Big)\,{\rm d}x\,{\rm d}y,
\end{equation}
where the integrands
$f_\e\colon\Rd\times \Rd\times \Rm\to [0,+\infty)$ satisfy suitable assumptions and growth conditions making the resulting energies comparable with \eqref{intro de giorgi 2} (see Sections \ref{section : functionals} and \ref{sec: Notation}). 
 To some extent, our energies may be regarded as a continuum counterpart  of those considered in \cite{BachBraidesCicalese} (see also the related work \cite{AlicandroBraidesCicaleseSolci}), where the authors detect sufficient conditions implying the $\Gamma$-convergence of certain discrete energies to free discontinuity functionals as in  \eqref{intro integral representation} below.

The first main result we obtain is  Theorem \ref{thm:compactness}, which states that, given a  sequence $\{\e_n\}_{n}$ converging to $0^+$ as $n\to+\infty$, the sequence $\{F_{\e_n}\}_{n}$ given by \eqref{intro funzionali} admits a subsequence that $\Gamma$-converges in the $L^1_{\rm loc}$-topology to a functional $F$, which for $u\in {\rm GSBV}^2(\Omega;\Rm)\cap L^1_{\rm loc}(\Omega;\Rm)$ takes the form
\begin{equation}\label{intro integral representation}
   F(u)= \int_\Omega f_{\rm bulk}(x,\nabla u)\,{\rm d}x + \int_{J_u}f_{\rm surf}(x,[u],\nu_u)\,{\rm d}\Hd,
\end{equation}
where the integrands $f_{\rm bulk}$ and $f_{\rm surf}$ are determined by auxiliary minimisation problems on small cubes.

In Section \ref{Section: Compactness} we use the localisation method for $\Gamma$-convergence (see \cite[Section~3.3]{braides2006handbook} or \cite[Chapter~18]{DalBook}) to prove $\Gamma$-compactness. The starting point of this technique is introducing  the set functions $U\mapsto F_\e(u,U)$, defined on the class of open subsets of $\Omega$ 
by a suitable localisation of the functionals \eqref{intro funzionali}.  The second step requires to produce a non-local version of the so called {\it fundamental estimate} (see \cite[Chapter~18]{DalBook}), whose proof, in our case, crucially relies on an argument, borrowed from \cite{AABPT} and developed in Section \ref{section: auxiliary}, that allows us to consider only finite-range interactions in \eqref{intro funzionali}. Finally, exploiting the fundamental estimate we prove that, if $U$ is sufficiently regular, up to subsequences $\{F_{\e_n}(\cdot, U)\}_{n}$  $\Gamma$-converges in $L^1_{\rm loc}(\Omega;\Rm)$ to 
a certain  $F(\cdot,U)$ and that this limit functional can be extended to a Borel measure in the set variable. We point out that if $U$ is not regular, the $\Gamma$-convergence may fail also for the reference functionals \eqref{intro de giorgi 2} (see \cite[Remark~7.1]{Gobbino}).

We will then exploit in Section \ref{sec: integral representation} the integral representation on ${\rm SBV}^2$ of Bouchitté, Fonseca, Leoni, and Mascarenhas \cite{BFLM2002}.
However,  we cannot directly apply the results of \cite{BFLM2002} to the limit functional $F$. Indeed, those results require in particular the existence of a constant $C>0$ such that
\begin{equation*}
    C\Bigl(\int_{U}|\nabla u|^2\,{\rm d}x+\int_{J_u\cap U}\!\!\!\!\!(1+|[u]|)\,{\rm d}\Hd\Bigr)\leq F(u,U)
\end{equation*}
for every $U\subseteq \Omega$ open and $u\in {\rm SBV^2}(\Omega;\Rm)$,  while our growth conditions only yield the weaker estimate
\begin{equation*}
     C\Bigl(\int_{U}|\nabla u|^2\,{\rm d}x
+\Hd(J_u\cap U)\Bigr)\leq F(u,U).
\end{equation*}
To deal with this problem, we resort to well-known truncation techniques employed, for instance, in  \cite{CagnettiPoincare} and \cite{DalToa}. We consider the perturbations $F^\delta(u):=F(u)+\delta\|[u]\|_{L^1(J_u;\Rm)}$ for $\delta>0,$ which satisfy the required lower bound, and  therefore admit an integral representation, and then let $\delta\to 0^+$ recovering equality  \eqref{intro integral representation} for all $u\in{\rm SBV}^2(\Omega;\Rm)$.  Finally, arguing by density, we obtain the desired integral representation \eqref{intro integral representation} for all functions $u$ in 
${\rm GSBV}^2(\Omega;\Rm)\cap L^1_{\rm loc}(\Omega;\Rm)$. 

We mention that a general integral representation result for functionals on ${\rm GSBV}^2(\Omega;\Rm)$ has recently been established by Crismale, Friedrich, and Solombrino in \cite{CriFriSolo}, but also in this instance, requiring growth conditions not suitable to a direct application in our setting.

In Section \ref{sec:convergence of infima} we analyse the asymptotic behaviour of the minimum values of Dirichlet problems for $\{F_\e\}_\e$ and compare them with the minimum values of the corresponding problems for the $\Gamma$-limit $F$. Combining this analysis with the compactness and integral representation results described above, we prove our second main result in Theorem \ref{theorem: minima}, which provides necessary and sufficient conditions for $\Gamma$-convergence of sequences of functionals of the form \eqref{intro funzionali}. These conditions are expressed as the existence, independently of the sequence $\{\e_n\}_n$, of limits of infima of minimisation problems for  $F_{\e_n}(\cdot,Q)$ on small cubes $Q$ (see \eqref{sufficient bulk} and \eqref{sufficient surf}). 

This kind of condition characterising  $\Gamma$-convergence is particularly useful when dealing with the homogenisation of free discontinuity functionals (see, for instance, the works \cite{BDV,CagnettiPoincare,CDMSZ2} in the local setting). In our non-local framework, a simple example may be given by 
\begin{equation*}
 \frac{1}{\e^{d}}\int_{\Omega}\int_{\Omega}a\Bigl(\frac{x}{\e}\Bigr)\min\Big\{\frac{|u(x)-u(y)|^2}{\e^2},\frac{1}{\e}\Big\}\rho\Big(\frac{|x-y|}{\e}\Big)\,{\rm d}x\,{\rm d}y,
\end{equation*}
where $a\colon \Rd\to (0,+\infty)$ is $1$-periodic and bounded.
 It may be interesting to study this problem, which in the Sobolev setting has already been addressed in \cite{AABPT}. 

 Our work can be compared with a previous paper by Cortesani \cite{CortesaniARMA}, who dealt with energies modelled after \eqref{braides-dalmaso}.  The arguments presented therein are flexible enough to infer that a ``generalised fundamental estimate" also holds for non-local energies as \eqref{intro de giorgi 2},    upon assuming that the interaction kernel $\rho$ is bounded, compactly supported, and lower semicontinuous. Then, the localisation method of $\Gamma$-convergence is employed in the $L^1$-topology. In this work, we pursue an approach more similar to that presented in \cite{AABPT}, which allows us to encompass a wider class of admissible kernels and more general growth conditions, performing the asymptotic analysis with respect to the $L^1_{\rm loc}$-convergence. 

Finally, we remark that, in recent years, there has been a growing interest in non-local  approximations of local functionals. Starting from the seminal work of Bourgain, Brezis, and Mironescu \cite{AnotherLook} on the approximation of the Dirichlet energy, several significant results of similar fashion have been obtained in the setting of functions of bounded variation and deformation. In particular, shortly after the publication of \cite{AnotherLook}, Dávila has addressed in \cite{Davila} the problem of approximating the total variation of the distributional derivative $Du$ of a function $u\in {\rm BV}(\Omega)$, and a similar problem has been considered by Arroyo-Rabasa and Bonicatto \cite{RabasaBonicatto} for the symmetric part of the derivative of vector fields of bounded deformation. A remarkable example of non-local approximation that fits in the above-mentioned framework is the one obtained using Gagliardo-seminorms, which are introduced via fractional kernels. We point out that such class of kernels is not taken into account in our analysis. In this vein, in a more related recent work \cite{Jolly}, the authors investigate the $\Gamma$-convergence of functionals defined on partitions through Riesz-type fractional variations.

\section{Notation and preliminary results}\label{sec: Notation}
In this section we fix the notation and introduce some preliminary results. 

\smallskip

\noindent {\bf Notation.} For fixed positive integers $d$ and $m$ we let $|\cdot|$ denote the Euclidean norm in $\R^k$, $k\in\{d,m, m\times d\}$. Given $x\in\Rd$ and $r>0$, we let $B_r(x):=\{y\in\Rd: |y-x|<r\}$ denote the open ball of center $x$ and radius $r$, the dependence on the center $x$ being omitted if $x=0$. The canonical basis of $\Rd$ is denoted by $\{e_1,...,e_d\}$ and the $i$-th component of a point $x\in\Rd$ with respect to this basis is denoted by $x_i$. The unit sphere in $\Rd$ is denoted by $\Sd:=\{x\in\Rd:|x|=1\}$ and we let $SO(m)$ denote the group of rotations of $\Rm$. For $E\subset \Rd, \e>0$, and $\xi\in\Rd$, we also introduce the notation
\begin{equation*}
    E_\e(\xi):=\{x\in E: x+\e\xi\in E\}.
\end{equation*}

We let $\Omega\subset \Rd$ denote a bounded open set with Lipschitz boundary. The collection of all open subsets of $\Omega$ is denoted by $\A(\Omega)$, while we let $\mathcal{A}_{\rm reg}(\Omega)$ denote the collection of all open subsets of $\Omega$ with Lipschitz boundary.  We let $\Ld$ and $\Hd$ denote the $d$-dimensional Lebesgue measure and the $(d-1)$-dimensional Hausdorff measure in $\Rd$, respectively.

\smallskip

\noindent {\bf Jump set.} Given a measurable function $u:\Omega\to \Rm$, a measurable set $E\subset \Omega$, and a point $x\in \Rd$ such that 
\begin{equation}\label{positive density}
    \limsup_{r\to 0^+}\frac{\Ld(B_r(x)\cap E)}{r^d}>0,
\end{equation}
we say that $a\in\Rm$ is the approximate limit in $E$ of $u$ at $x$, in symbols
\begin{equation*}
       \text{ap}\lim_{\substack{y\to x\\y\in E}}u(y)=a,
\end{equation*}
if for every $\e>0$ we have
\begin{equation*}
    \limsup_{r\to 0^+}\frac{\Ld(\{y\in E\cap B_r(x):|u(y)-a|>\e\})}{r^d}=0.
\end{equation*}
It follows from \eqref{positive density} that, if the approximate limit exists, it is unique. 

Given a function $u\colon\Omega\to \Rm$, the {\it jump set} $J_u$ of $u$ is a Borel set made up of all points $x\in\Omega$ for which there exists a triple $(u^+(x),u^-(x),\nu_u(x))\in \Rm\times \Rm\times\Sd$ such that, setting 
$H^\pm(x):=\{y\in\Omega:\pm y\cdot\nu_u(x)\geq 0\},$
we have 
\begin{equation*}
\text{ap}\!\!\!\!\!\lim_{\substack{y\to x\\y\in H^{\pm}(x)}}u(y)=u^\pm (x).
\end{equation*}
The triple  $(u^+(x),u^-(x),\nu_u(x))\in \Rm\times \Rm\times\Sd$  is uniquely defined up to changing the roles of $u^+(x)$ and $u^-(x)$ and  the sign of $\nu_u(x)$. Having fixed  $\nu_u\colon J_u\to \Sd$, we set $[u](x):=u^+(x)-u^-(x)$ for every $x\in J_u$.
\smallskip

\noindent {\bf Functions of (generalised) bounded variation.}
We here recall the main properties of the space ${\rm BV}(\Omega;\Rm)$ of {\it functions of bounded variation}. The interested reader may find a proof of all the results stated below in \cite{AFP,AmbrosioExistence}.  A function $u\colon\Omega\to \Rm$ is said to be a function of bounded variation if $u\in L^1(\Omega;\Rm)$ and its distributional derivative $Du$ is a bounded  Radon measure with values in $\Rmd$. 

The measure $Du$ admits a decomposition as the sum of three mutually singular measures
\begin{equation}\notag 
    Du=\nabla u\Ld+D^cu+([u]\otimes \nu_u)\Hd\mres J_u,
\end{equation}
where 
\begin{itemize}
    \item[(i)] $\nabla u \in L^1(\Omega;\R^{m\times d})$ denotes the {\it approximate gradient} of $u$,
    \item[(ii)] $D^cu$, called the {\it Cantor part} of $Du$,  is a measure singular with respect to $\Ld$ that vanishes on every Borel set $B\subseteq \Omega$  which is $\sigma$-finite with respect to $\Hd$,
    \item[(iii)]  $[u]\otimes \nu_u$ denotes the tensor product between $[u]$ and $\nu_u$.
\end{itemize}

The space of {\it special functions of bounded variation }${\rm SBV}(\Omega;\Rm)$, is the subspace of functions $u\in {\rm BV}(\Omega;\Rm)$ such that $D^cu=0$. We will often work with its subspace 
 $${\rm SBV}^2(\Omega;\Rm):=\Big\{u\in {\rm SBV}(\Omega;\Rm):\int_\Omega|\nabla u|^2\,\dx x+\Hd(J_u)<+\infty\Big\}.$$

We now recall the definition of ${\rm GSBV}(\Omega;\Rm)$-functions introduced by De Giorgi and Ambrosio in \cite{DeGiorgiAmbrosio} and that constitute the natural domain where to set the study of  minimisation problems related to the Mumford-Shah functional.

${\rm GSBV}(\Omega;\Rm)$ is the space of all $\Ld$-measurable functions $u\colon \Omega\to\Rm$ such that $\Phi\circ u\in {\rm SBV}_{\rm loc}(\Omega;\Rm)$ for every $C^1$ function $\Phi\colon\Rm\to \Rm$ whose gradient has compact support, where we have set \begin{equation*}
    {\rm SBV}_{\rm loc}(\Omega;\Rm):=\{u\in {\rm SBV}(\Omega';\R^m) \text{ for all } \Omega' \text{ open}, \Omega' \subset\subset \Omega\}.
\end{equation*} When $m=1$, the space ${\rm GSBV}(\Omega;\R)$ can also be characterised by means of ${\it rough}$ vertical truncations. More precisely, $u\in {\rm GSBV}(\Omega;\R)$ if and only if the function $u^M:=(u\lor -M)\land M$ belongs to ${\rm SBV_{loc}}(\Omega)$, where we have set $s\lor t:=\max\{s,t\}$, and $s\land t:=\min\{s,t\}$ for every $s,t\in\R$. 

We recall that every function $u\in {\rm GSBV}(\Omega;\Rm)$ admits an approximate gradient $\nabla u$, i.e., an $\mathcal{L}^d$-measurable function $\nabla u:\Omega\to\Rmd$ such that 
\begin{equation*}
       \text{ap}\lim_{\substack{y\to x\\y\in \Omega}}\frac{u(y)-u(x)- \nabla u(x)(y-x)}{|y-x|}=0
\end{equation*}
for $\Ld$-a.e. $x\in\Omega$.

Finally, we introduce the subspace 
\begin{gather}
\notag 
    {\rm GSBV}^2(\Omega;\Rm):=\Big\{u\in {\rm GSBV}(\Omega;\Rm):\int_\Omega|\nabla u|^2\,\dx x+\Hd(J_u)<+\infty\Big\},
\end{gather}
and also set 
\begin{equation*}
    {\rm SBV}^2_{\rm loc}(\Omega;\Rm):=\{u\in {\rm SBV}_{\rm loc}(\Omega;\Rm) : \int_\Omega|\nabla u|^2\,\dx x+\Hd(J_u)<+\infty\}.
\end{equation*}
For the sake of notation, in all the previous functional spaces we omit the target space $\Rm$ when $m=1$.

In the following lemma, whose proof can be found in \cite[Proposition~2.3]{DMFrToa2005}, we observe that, even in the vector-valued case $m>1$, the space ${\rm GSBV^2}(\Omega;\Rm)$ can be characterised by means of {\it rough} vertical truncations.
\begin{Lemma}
    ${\rm GSBV}^2(\Omega;\Rm)$ is a vector space. Moreover, if $u\in {\rm GSBV^2}(\Omega;\Rm)$ then $u_i\in {\rm GSBV^2}(\Omega)$ for all $i\in\{1,...,m\}$. In particular, for all $M>0$, if $u^M$ denotes the function whose $i$-th component is given by $u^M_i:=(u_i\lor -M)\land M$, then $u^M\in{\rm SBV}^2(\Omega;\Rm)$.
\end{Lemma}

Justified by the previous lemma, for every $u\colon \Rd\to \Rm$ and $M>0$, $u^M$ will always denote the truncated function whose $i$-th component is given by $u^M_i:=(u_i\lor -M)\land M$.
\begin{Remark}\label{remark truncations}Let $u\in {\rm GSBV}^2(\Omega;\Rm)$. 
One can check that the following conditions hold:
\begin{itemize}
    \item[(i)] $\nabla u^M=\nabla u$ for $\Ld$-a.e.\ $x\in\Omega$ such that $|u(x)|\leq M$ so that $\nabla u^M$ converges $\Ld$-a.e.\ in $\Omega$ (and in $L^1(\Omega;\Rm)$) to $\nabla u$ as $M\to+\infty$, and $|\nabla u^M|$ converges monotonically to $|\nabla u|$;
    \item[(ii)] for every $M>0$ the jump set $J_{u^M}$ is contained in $J_u$,   up to a $\Hd$-negligible set, and  $\Hd(J_u\setminus \bigcup_{M>0} J_{u^M})=0$.
    Moreover, for $\Hd$-a.e.\ $x\in J_{u^M}$, we have $(u^M(x))^\pm=(u^{\pm} (x))^M$;
    \item [(iii)] the map $\zeta\mapsto\zeta^M$ is $1$-Lipschitz, $(\zeta^{M+1})^M=\zeta^M$ for all $\zeta\in\Rm$ and $M>0$; hence, $|[u^M]|=|[(u^{M+1})^M]|\leq |[u^{M+1}]| $ for all $M>0$ and $\Hd$-a.e.\ in $J_{u^M}$.
\end{itemize}
\end{Remark}

\medskip 

\noindent {\bf $\mathbf{\Gamma}$-convergence.} 
 We recall the definition of $\Gamma$-convergence in $L^1_{\rm loc}$. For a thorough introduction to the general theory of $\Gamma$-convergence we refer the reader to \cite{BraidesGamma,DalBook}. 
 
 Given a positive sequence $\{\e_n\}_n$ tending to $0$ and functionals $\mathcal{F}_{\e_n}\colon L^1_{\rm loc}(\Omega;\Rm)\to[-\infty, +\infty]$ for $n\in\N$ and $\mathcal{F}_0\colon L^1_{\rm loc}(\Omega;\Rm) \to[-\infty, +\infty]$, we say that $\mathcal{F}_0$ is the $\Gamma$-limit with respect to the $L^1_{\rm loc}(\Omega;\Rm)$-convergence of the sequence $\{\mathcal{F}_{\e_n}\}_{n}$ if the following conditions hold:
\begin{itemize}
    \item[(i)] for every $u\in L^1_{\rm loc}(\Omega;\Rm)$ and $\{u_n\}_n$ such that $u_n\to u$ in $L^1_{\rm loc}(\Omega;\Rm)$ as $n\to+\infty$, it holds that
    \begin{equation*}
        \liminf_{n\to+\infty} \mathcal{F}_{\e_n}(u_n) \geq \mathcal{F}_0(u);
    \end{equation*}
    \item[(ii)] for every $u\in L^1_{\rm loc}(\Omega;\Rm)$ there exists $\{v_n\}_n$ such that $v_n\to u$ in $L^1_{\rm loc}(\Omega;\Rm)$ as $n\to+\infty$ and
    \begin{equation*}
        \limsup_{n\to+\infty} \mathcal{F}_{\e_n}(v_n) \leq \mathcal{F}_0(u).
    \end{equation*}
\end{itemize}
In such case, we write 
\begin{equation*}
    \mathcal{F}_0=\underset{n\to+\infty}{\Gamma\text{-}\lim}\, \mathcal{F}_{\e_n}.
\end{equation*}

We recall that the functionals $\Gamma$-$\liminf$ and $\Gamma$-$\limsup$ are respectively defined as follows:
\begin{equation}\label{def Gammalimsup}
\begin{gathered}
    \Gamma\text{-}\liminf_{n\to+\infty} \mathcal{F}_{\e_n} (u) := \inf\{\liminf_{n\to+\infty} \mathcal{F}_{\e_n}(u_n): u_n\to u \text{ in } L^1_{\rm loc}(\Omega;\Rm)\}; \\
        \Gamma\text{-}\limsup_{n\to+\infty} \mathcal{F}_{\e_n} (u) := \inf\{\limsup_{n\to+\infty} \mathcal{F}_{\e_n}(u_n): u_n\to u \text{ in } L^1_{\rm loc}(\Omega;\Rm)\}.
\end{gathered}
\end{equation}
Both functionals in \eqref{def Gammalimsup} are lower semicontinuous with respect to the convergence in $L^1_{\rm loc}(\Omega;\Rm)$, they coincide if and only if $\Gamma$-$\lim_{n\to+\infty} \mathcal{F}_{\e_n}$ exists, and, in this case, the $\Gamma$-limit coincides with their common value. 

Given a family of functionals $\mathcal{F}_{\e}\colon L^1_{\rm loc}(\Omega;\Rm)\to[-\infty, +\infty]$ for $\e>0$ and $\mathcal{F}_0\colon L^1_{\rm loc}(\Omega;\Rm) \to[-\infty, +\infty]$, we say that $\mathcal{F}_0$ is the $\Gamma$-limit with respect to the $L^1_{\rm loc}(\Omega;\Rm)$-convergence of the family $\{\mathcal{F}_{\e}\}_{\e}$ if, for any positive sequence $\{\e_n\}_n$ tending to $0$, it holds
\begin{equation*}
    \mathcal{F}_0=\underset{n\to+\infty}{\Gamma\text{-}\lim}\, \mathcal{F}_{\e_n};
\end{equation*}
and we write 
\begin{equation*}
    \mathcal{F}_0=\Gamma\text{-}\lim_{\e\to0^+} \mathcal{F}_{\e}.
\end{equation*} 

We point out that, throughout the rest of the work, $\Gamma$-convergence shall always be meant with respect to the convergence in $L^1_{\rm loc}(\Omega;\Rm)$.

\medskip

\noindent  {\bf Reference functionals.}
We now introduce some families of non-local functionals that constitute upper and lower bounds for the general family of energies that we mean to study. 

 For $i\in\{1,2\}$ we let $\rho_i:\R^d\to [0,+\infty]$ be a radially symmetric kernel; in particular, we suppose that 
\begin{equation}\label{finite momenta}
    \int_{\Rd} (1+|\xi|+|\xi|^2)\rho_i(\xi)\,\dx\xi<+\infty,
\end{equation}
\begin{equation}\label{non zero in zero}
    \rho_1(\xi) \geq c_0 \qquad \text{for } \Ld \text{-a.e.\ } \xi\in B_{r_0}
\end{equation}
for some positive constants $c_0,r_0$. We let $\{\psi_{\e}\}_\e$ be a family of non-negative radially symmetric kernels such that
\begin{equation}\label{psi momento}
    \int_{\Rd}|\xi|\psi_{\e}(\xi)\,\dx\xi=1 \quad \text{ for every } \e>0,
\end{equation}
for every $\delta>0$ there exists $r_\delta>0$ such that
\begin{equation}\label{psi infinito}
   \limsup_{\e\to0^+} \int_{\Rd\setminus B_{r_\delta}}|\xi|\psi_{\e}(\xi)\,\dx\xi <\delta, 
\end{equation}
and
\begin{equation}\label{psi integrale}
    \limsup_{\e\to0^+}\int_{\Rd}\psi_{\e}(\xi)\,\dx\xi<+\infty.
\end{equation}

For $\e>0$, $U\in\A(\Omega)$, and $u\in L^1_{\text{loc}}(\Omega;\Rm)$, we define the functionals     
\begin{equation*}
    G_{i,\e}(u, U):=\int_{U\times U}\frac{1
    }{\e^d}\rho_{i}\Bigl(\frac{y-x}{\e}\Bigr)\min\Bigl\{\frac{|u(y)-u(x)|^2}{\e^2}, \frac{1}{\e}\Bigr\}\,\dx x\, \dx y \qquad \text{for } i\in\{1,2\},
\end{equation*} 
and
\begin{equation*}
    P_\e(u,U) := \int_{U\times U}\frac{1}{\e^d}\psi_{\e}\Bigl(\frac{y-x}{\e}\Bigr)\frac{|u(y)-u(x)|}{\e}\,\dx x\, \dx y.
\end{equation*}
Recalling that $U_\e(\xi):=\{x\in U: x+\e\xi\in U\}$, we perform the change of variables $\xi:=(y-x)/\e$ to obtain 
\begin{equation*}
    G_{i,\e}(u, U)  = \int_{\R^d}\rho_i(\xi)\int_{U_{\e}(\xi)}g_\e\Bigl(\frac{u(x+\e\xi)-u(x)}{\e}\Bigr)\,\dx x\,\dx \xi \qquad \text{for } i\in\{1,2\},
\end{equation*}
where we set
\begin{equation}\label{def geps}
 g_\e(z):=\min\Big\{|z|^2, \frac{1}{\e}\Big\} \quad \text{for all } z\in\R^m,
\end{equation}
and
\begin{equation*}
P_\e(u,U)=\int_{\R^d}\psi_{\e}(\xi)\int_{U_{\e}(\xi)}\frac{|u(x+\e\xi)-u(x)|}{\e}\,\dx x\,\dx \xi.
\end{equation*}

Conditions \eqref{finite momenta}-\eqref{psi infinito} play distinct roles in our analysis. On the one hand, inequalities \eqref{finite momenta} and \eqref{psi momento} guarantee that the $\Gamma$-limits of $\{G_{i,\e}\}_\e, i\in\{1,2\}$, and $\{P_\e\}_\e$ are finite. On the other, \eqref{non zero in zero} is useful in order to obtain suitable coerciveness conditions in ${\rm GSBV}^2$. Assumption \eqref{psi infinito} corresponds to the assumption (H2) introduced in \cite{AABPT} and ensures that short-range interactions are favoured when $\e\to0^+$; this condition is in fact stronger than the analogous hypothesis appearing in \cite{Davila, Ponce}, which reads as
\begin{equation*}
    \lim_{\e\to 0^+} \int_{\Rd\setminus B_{\frac{r}{\e}}}|\xi|\psi_{\e}(\xi)\,\dx\xi = 0 \quad \text{ for every } r>0.
\end{equation*}

We mention some useful results concerning  the above functionals. For $i\in\{1,2\}$ and $U\in\A(\Omega)$, we introduce 
the Mumford-Shah functional
\begin{equation*}
{\rm MS}_i(u,U):=\begin{cases}   \displaystyle 
      \lambda_i\int_U|\nabla u|^2\,\dx x+\mu_i\Hd(J_u \cap U) &\text{ if }u\in L^1_{\text{\rm loc}}(\Omega;\Rm) \cap{\rm GSBV}^2(U;\Rm),\\
      +\infty &\text{ if } u\in  L^1_{\text{loc}}(\Omega;\Rm) \setminus {\rm GSBV}^2(U;\Rm),
    \end{cases}
\end{equation*}
where $\lambda_i$ and $\mu_i$ are two positive constants given by 
\begin{equation}
\label{param MS}
\lambda_i:=\frac{1}{d}\int_{\R^d}|\xi|^{2}\rho_i(\xi)\,{\rm d}\xi,\qquad\quad \mu_i:=2\Hd(\mathbb{S}^{d-1})\int_{\R^d}|\xi|\rho_i(\xi)\, {\rm d}\xi,
\end{equation}
and we introduce (a multiple of) the total variation 

\begin{equation*}
{\rm TV}_\kappa(u,U):=\kappa|Du| 
(U)   \qquad \text{if }u\in  L^1_{\rm loc}(\Omega;\Rm)\cap{\rm BV}(U;\Rm),
\end{equation*} 
 where the constant $\kappa$ is given by
\begin{equation}
\label{def k}
    \kappa:= \Bigl(\frac{1}{\Hd(\Sd)}\int_{\Sd}|e_1\cdot \sigma| \, \dx \Hd(\sigma)\Bigr).
\end{equation}
 
Finally, we introduce the reference functionals
\begin{align*}
    H_\e(u,U) :=G_{2,\e}(u,U)+P_\e(u,U) \qquad &\text{if }u\in  L^1_{\rm loc}(\Omega;\Rm),\\
    H(u,U) :={\rm MS}_2(u,U)+{\rm TV}_\kappa(u,U) \qquad & \text{if }u\in  L^1_{\rm loc}(\Omega;\Rm)\cap{\rm BV}(U;\Rm).
\end{align*}

 We summarize in the next theorem the needed auxiliary results concerning the scalar case $m=1$, which is the setting of the works of Ponce \cite{Ponce}, Dávila \cite{Davila}, and  Gobbino \cite{Gobbino}. 
\begin{Theorem}\label{thm: auxiliary gamma}  Let $U\in\A_{\rm reg}(\Omega)$ and assume that $m=1$. 
Then the following two conditions hold: 
\begin{itemize}
    \item[(i)] for all $u\in L^1_{\rm loc}(\Omega)\cap {\rm BV}(U)$ and for all $\{u_\e\}_\e$ converging to $u$ in $L^1_{\rm loc}(\Omega)$, we have
    \begin{equation*}
        \liminf_{\e\to0^+} P_\e(u_\e,U) \geq  {\rm TV}_\kappa(u,U);
    \end{equation*}
    \item[(ii)] for all $u\in L^1_{\rm loc}(\Omega)\cap {\rm BV}(U)$, we have 
    \begin{equation*}
        \lim_{\e\to0^+} P_\e(u, U) = {\rm TV}_\kappa(u,U).   \end{equation*}
\end{itemize}
If we assume, in addition, that for $i\in\{1,2\}$ we have 
\begin{equation}\label{all momenta}
    \int_{\Rd}|\xi|^k\rho_i(\xi)\,{\rm d}\xi <+\infty
\end{equation}
 for every $k\in\N$, then:
\begin{itemize}
        \item[(iii)] for all $u\in L^1_{\rm loc}(\Omega)$ and for all $\{u_\e\}_\e$ converging to $u$ in $L^1_{\rm loc}(\Omega)$, we have
    \begin{equation*}
        \liminf_{\e\to0^+} G_{i,\e}(u_\e,U) \geq {\rm MS}_i(u,U);
    \end{equation*}
    \item[(iv)] for all $u\in L_{\text{\rm loc}}^1(\Omega)$, we have
    \begin{equation*}
        \lim_{\e\to0^+} G_{i,\e}(u,U) = {\rm MS}_i(u,U);
    \end{equation*}
\end{itemize}
In particular, assuming \eqref{all momenta}, for all $u\in  L^1_{\rm loc}(\Omega)\cap {\rm BV}(U)$ it holds that
\begin{equation*}
\Gamma\text{-}\lim_{\e\to0^+} H_\e(u,U)= H(u,U).
\end{equation*}
\end{Theorem}
\begin{proof}
Part $(i)$ follows by \cite[Theorem~8]{Ponce}, and part $(ii)$ follows by \cite[Theorem~1]{Davila}.
    
 The proofs of $(iii)$ and $(iv)$ are the content of \cite[Theorem~4.4~and~Section~7]{Gobbino}; finally, the computation of the $\Gamma$-limit of $\{H_\e(\cdot, U)\}_\e$ is obtained combining $(i)$-$(iv)$. 
\end{proof}

\begin{Remark}\label{rmk: vector bounds}
    If $m>1$, we obtain immediately from Theorem \ref{thm: auxiliary gamma} that there exist positive constants $c_1$ and $c_2$, depending only on $d$ and $m$, such that for all  $U\in\A_{\rm reg}(\Omega)$ we have
    \begin{gather} \label{lower bound geps}
        c_1{\rm MS}_1(u,U)\leq \Gamma\text{-}\liminf_{\e\to 0^+}G_{1,\e}(u,U) \text{ for all $u\in L^1_{\rm loc}(\Omega;\R^m)$},\\ \label{upper bound heps}
          \Gamma\text{-}\limsup_{\e\to 0^+}H_\e(u,U)\leq c_2 H(u,U) \text{ for all }u\in L^1_{\rm loc}(\Omega;\Rm)\cap{\rm BV}(U;\Rm).
    \end{gather}
\end{Remark}

\begin{Remark}
 We note that $(iii)$ and $(iv)$ of Theorem \ref{thm: auxiliary gamma}, and therefore also inequalities \eqref{lower bound geps}, \eqref{upper bound heps}, hold by simply assuming \eqref{finite momenta} in place of \eqref{all momenta}, hypothesis introduced in \cite[Section~7]{Gobbino}. To see this, we resort to the main result of Section \ref{section: auxiliary}, that allows us to compute the $\Gamma$-limit of a general family of functionals, including $\{G_{i,\e}\}_\e, i\in\{1,2\},$ by first limiting the analysis to kernels with compact support.
 
 For $i\in\{1,2\}$, let $T>0$, and consider functionals 
\begin{equation}\label{Gobbino troncati}
      G^T_{i,\e}(u, U) := \int_{B_T}\rho_i(\xi)\int_{U_{\e}(\xi)}g_\e\Bigl(\frac{u(x+\e\xi)-u(x)}{\e}\Bigr)\,\dx x\,\dx \xi
 \end{equation}
 for all $(u,U)\in L^1_{\text{loc}}(\Omega;\Rm) \times\A_{\rm reg}(\Omega)$. 
 Since $\rho_i$ satisfies \eqref{finite momenta}, letting $\chi_{B_T}$ denote the characteristic function of $B_T$, we infer $\rho_i\chi_{B_T}$ fulfils \eqref{all momenta}, and then, using Theorem \ref{thm: auxiliary gamma}, we obtain
 \begin{equation*}
     \Gamma\text{-}\lim_{\e\to0^+}  G^T_{i,\e}(u, U) = \begin{cases}   \displaystyle 
      \lambda_i^T\int_{U}\!\!|\nabla u|^2\,\dx x+\mu_i^T\Hd(J_u \cap U) &\text{if }u\in L^1_{\text{\rm loc}}(\Omega;\Rm)\cap {\rm GSBV}^2(U;\Rm),\\
      +\infty &\text{if } u\in  L^1_{\text{loc}}(\Omega;\Rm) \setminus {\rm GSBV}^2(U;\Rm),
    \end{cases}
 \end{equation*}
 with
 \begin{equation*}
     \lambda_i^T:=\frac{1}{d}\int_{B_T}|\xi|^{2}\rho_i(\xi)\,{\rm d}\xi,\qquad\quad \mu^T_i:=2\Hd(\mathbb{S}^{d-1})\int_{B_T}|\xi|\rho_i(\xi)\, {\rm d}\xi.
 \end{equation*}
 Applying Proposition \ref{proposition: truncated functionals}, we conclude
\begin{equation*}
      \Gamma\text{-}\lim_{\e\to0^+}  G_{i,\e}(u, U) =\lim_{T\to+\infty}\Gamma\text{-}\lim_{\e\to0^+}  G^T_{i,\e}(u, U) = {\rm MS}_i(u,U) 
\end{equation*}
 for every $u\in L^1_{\rm loc}(\Omega;\Rm)$, which is the claim.
\end{Remark}

\section{Setting of the Problem and Main results} 
\label{section : functionals}
In this section we introduce the functionals that are going to be our main object of investigation and state the main results of the paper. We recall that $\rho_1, \rho_2$ are radially symmetric non-negative kernels satisfying \eqref{finite momenta} and \eqref{non zero in zero}, that the family of radially symmetric non-negative kernels $\{\psi_\e\}_\e$ satisfies \eqref{psi momento}-\eqref{psi integrale} and that, for all $\e>0,$ the functions $g_\e$ are those introduced in \eqref{def geps}.

\medskip 

\noindent {\bf Setting of the problem.} We fix a family of non-negative functions $\{\eta_\e\}_\e$ such that for every $\delta>0$ there exists $r_\delta>0$ such that
\begin{equation}\label{eta infinito}
     \limsup_{\e\to0^+} \int_{\Rd\setminus B_{r_\delta}}\eta_{\e}(\xi)\,\dx\xi <\delta, 
\end{equation}
and
\begin{equation}\label{Lambda}
\Lambda:=\limsup_{\e\to 0^+}\|\eta_\e\|_{L^1(\R^d)}<+\infty.
\end{equation}
For every $\e>0$ we consider an integrand $f_\e:\R^d\times \R^d\times \Rm\to[0,+\infty)$ satisfying the following properties:
    \begin{itemize}
        \item[(i)]({\it measurability}) $f_\e$ is Borel measurable;
        \item[(ii)] ({\it growth conditions}) for almost every $x,\xi\in\Rd$ and $z\in\Rm$ it holds
        \begin{equation} \label{comparison densities}
    \rho_1(\xi)g_\e(z) \leq f_\e(x,\xi, z) \leq \rho_2(\xi)g_\e(z)+ \psi_\e(\xi)|z|+\eta_\e(\xi);
\end{equation}
\item[(iii)]({\it monotonicity}) for almost every $x,\xi\in\Rd$ and for every $z_1,z_2\in\Rm$ with $|z_1|\leq |z_2|$ it holds
\begin{equation}\label{monotonicity}
    f_\e(x,\xi,z_1)\leq  f_\e(x,\xi,z_2).
    \end{equation}
    \end{itemize}
We study functionals
\begin{equation}\label{def functionals}
    F_\e(u, U):=\int_{\R^d}\int_{U_{\e}(\xi)}f_\e\Big(x,\xi,\frac{u(x+\e\xi)-u(x)}{\e}\Big)\,\dx x\,\dx \xi, 
\end{equation} 
where $\e>0$ and $(u,U)\in L^1_{\rm loc}(\Omega;\Rm)\times \A(\Omega)$.

\begin{Remark}
    Thanks to the upper and lower bounds in  \eqref{comparison densities}, it follows immediately that  
\begin{equation}\label{comparison functionals}
    G_{1,\e}(u, U) \leq  F_\e(u, U) \leq  H_\e(u, U)+\|\eta_\e\|_{L^1(\Rd)}|U|.
\end{equation}
 We also note that the presence of the term $P_\e$ in the upper bound allows us to treat limit energies that, contrary to the Mumford-Shah functional, admit linear growth with respect to the size of the jump $|[u]|$.
\end{Remark}

\begin{Remark}
Hypothesis ({{iii}}) above was already considered in the discrete setting in \cite{BachBraidesCicalese}.
It will be repeatedly used throughout the work to allow us to restrict our attention to functions in $ L^\infty$.
More in general, given $(u,U)\in L^1_{\rm loc}(\Omega;\Rm)\times \mathcal{A}(\Omega)$ and a $1$-Lipschitz function $\Phi\colon\Rm\to\Rm$, it follows immediately from \eqref{monotonicity} that 
\begin{equation}
\label{contraction}
    F_\e(\Phi\circ u,U)\leq F_\e(u,U).
\end{equation}
As a consequence, we obtain that
\begin{align*}
    F_\e(u+a,U) & =F_\e(u,U)\quad \text{ for every }a\in \Rm, \\ F_\e(Ru,U) & =F_\e(u,U)\quad \text{ for every }R\in SO(m),\\
    F_\e(u^M,U) & \leq F_\e(u,U)\quad \text{ for every } M>0.
\end{align*}
\end{Remark}

We now present the main results of the paper. To this aim, we first introduce some further notation. 
Given a functional $\mathcal{F}\colon {\rm SBV}^2(\Omega;\Rm)\times \mathcal{A}(\Omega)\to [0,+\infty)$, a function $w\in {\rm SBV}^2(\Omega;\Rm)$, and $U\in\mathcal{A}_{\rm reg}(\Omega)$ we set 
\begin{equation}\label{def minimisation aux}
    \m^\mathcal{F}(w,U):=\inf\{\mathcal{F}(u,U):\,u\in {\rm SBV}^2(U;\Rm) \text{ with }u=w \text{ in a neighbourhood of }\partial U\}.
\end{equation}
Given $L\in \Rmd$, with a slight abuse of notation, we also let $L$ denote the corresponding linear map $L:\Rd\to \R^m$. Given $x\in\Rd$, $\zeta\in\Rm$, and $\nu\in\Sd$, we introduce the function  $u_{x,\zeta,\nu}\colon\Rd\to\Rm$ defined for every $y\in\Rd$ by 
\begin{equation*} 
    u_{x,\zeta,\nu}(y):=\begin{cases}
     \zeta &\text{ if } (y-x) \cdot\nu\geq 0,\\
     0 &\text{ otherwise}.
    \end{cases}
\end{equation*}
We let the symbol $Q^\nu(x,r)$ denote an open cube of side length $r$ and centre $x$ with two faces orthogonal to $\nu$, and we assume that $Q^\nu(x,r)=Q^{-\nu}(x,r)$. We also set $Q(x,r):=x+(-\frac{r}{2},\frac{r}{2})^d$.

\begin{Theorem}\label{thm:compactness}
  Let $\{F_\e\}_\e$ be as in \eqref{def functionals}, and let $\{\e_n\}_n$ be a positive sequence converging to $0$ as $n\to+\infty$. Then, there exists a subsequence $\{\e_{n_k}\}_k$  and a functional $F\colon L^1_{\rm loc}(\Omega;\Rm)\times \A(\Omega)\to [0,+\infty]$ such that
  \begin{equation}\notag 
\underset{n\to+\infty}{\Gamma\text{-}\lim}\, F_{\e_n}(u,U)=F(u,U)
  \end{equation}
  for every $(u,U)\in L^1_{\rm loc}(\Omega;\Rm)\times \A_{\rm reg}(\Omega)$. In particular,
   \begin{equation}\notag 
     F(u,U)=  \begin{cases}\displaystyle\int_U \!\!f_{{\rm bulk}}(x,\nabla u){\rm d}x+\!\!\int_{J_u\cap U} \!\!\!\!\!f_{{\rm surf}}(x, [u],\nu_u){\rm d}\Hd \!\!&\text{if }u\in L^1_{\text{\rm loc}}(\Omega;\Rm) {\cap}\, {\rm GSBV}^2(U;\Rm),\\
     +\infty \!\!&\text{if }u\in L^1_{\text{\rm loc}}(\Omega;\Rm) {\setminus}  {\rm GSBV}^2(U;\Rm). 
     \end{cases}
    \end{equation}
    for every $U\in\A(\Omega)$ and $u\in {\rm GSBV}^2(U;\Rm)\cap L^1_{\rm loc}(\Omega;\Rm)$,
    where 
    \begin{gather}
          \label{Gammabulk}  f_{{\rm  bulk}}(x,L):= \limsup_{r\to 0^+}\frac{\m^F( L,Q(x,r))}{r^d} \quad 
          \text{for all $x\in\Omega$ and   $L\in\R^{m\times d}$},\\
   \label{Gammasurf}
            f_{{\rm surf}}(x,\zeta,\nu):=\limsup_{r\to 0^+}\frac{\m^F(u_{x,\zeta,\nu},Q^\nu(x,r))}{r^{d-1}}\quad \text{ for all $x\in\Omega$, $\zeta\in\R^m$, and $\nu\in\Sd$.}
        \end{gather}

        \noindent Letting $\lambda_1, \lambda_2, \mu_1, \mu_2$ as in \eqref{param MS}, $\kappa$ as in \eqref{def k}, $c_1,c_2$ as in Remark \ref{rmk: vector bounds}, and $\Lambda$ as in \eqref{Lambda}, we have 
        \begin{align*}
          c_1\lambda_1|L|^2 \leq  f_{\rm bulk}(x,L)\leq c_2(\lambda_2|L|^2+\kappa|L|)+\Lambda\quad & \text{for all $x\in\Omega$ and $L\in\Rmd$},\\
          c_1\mu_1 \leq  f_{\rm surf}(x,\zeta,\nu)\leq c_2(\mu_2+\kappa|\zeta|)\quad & \text{for all $x\in\Omega$, $\zeta\in\Rm$, and $\nu\in\Sd$}.
        \end{align*}
        Moreover, for all $x\in\Omega$ and $\nu\in\Sd$ we have \begin{align*}
  f_{\rm bulk}(x,L_1)\leq  f_{\rm bulk}(x,L_2) \quad & \text{  if  $|L_1|\leq |L_2| $},\\
 f_{\rm surf}(x,\zeta_1,\nu)\leq  f_{\rm surf}(x,\zeta_2,\nu)\quad & \text{ if $|\zeta_1|\leq |\zeta_2|$,}
\end{align*}
 so that  
\begin{align*}f_{\rm bulk}(x,L_1)=  f_{\rm bulk}(x,L_2) \quad & \text{ if $|L_1|=|L_2|$},\\
f_{\rm surf}(x,\zeta_1,\nu)  =f_{\rm surf}(x,\zeta_2,\nu)\quad & \text{ if $|\zeta_1|=|\zeta_2|$},\\
    f_{\rm surf}(x,\zeta,\nu)=f_{\rm surf}(x,\zeta,-\nu) \quad & \text{ for all $\zeta\in\Rm$}.
\end{align*}
\end{Theorem}
\begin{proof}
    It is sufficient to combine Propositions \ref{prop:compactness} and \ref{prop: representation}.
\end{proof}

We now introduce the minimisation problems for the non-local functionals we are interested in. 
Given $U\in\A(\Omega)$, $w\in L^1_{\rm loc}(\Omega;\Rm)$, and $s>0$, we set 
\begin{equation*}
    \mathcal{D}_{s,w}(U):=\big\{u\in L^1_{\rm loc}(\Omega;\Rm): \, u=w \text{ for $\Ld$-a.e.\ $x\in U$ with dist}(x,\Rd\setminus U)<s\big\},
\end{equation*}
and consider the minimisation problems
\begin{equation*}
   \m_s^{F_{\e_n}}(w,U):=\inf\big\{F_{\e_n}(u,U): u\in \mathcal{D}_{\e_n s,w}(U)
   \big\}.
\end{equation*}
We observe that for every $0<s_1<s_2$ we clearly have the inclusion $\mathcal{D}_{s_2,w}(U)\subset \mathcal{D}_{s_1,w}(U)$, which implies 
\begin{equation*}
    \m_{s_1}^{F_{\e_n}}(w,U)\leq \m_{s_2}^{F_{\e_n}}(w,U).
\end{equation*}

For every $x\in\Omega$, $L\in\Rmd$, $z\in\Rm$, and $\nu\in\Sd$ we define 
\begin{align*}
   f_{\rm bulk}'(x,L) & := \limsup_{r\to 0^+}\,\sup_{s>0}\,\liminf_{n\to+\infty}\frac{\m_s^{F_{\e_n}}( L, Q(x,r))}{r^d},\\
    f_{\rm bulk}''(x,L) & := \limsup_{r\to 0^+}\,\sup_{s>0}\,\limsup_{n\to+\infty}\frac{\m_s^{F_{\e_n}}( L, Q(x,r))}{r^d},\\
    f_{\rm surf}'(x,\zeta,\nu) & :=\limsup_{r\to 0^+}\,\sup_{s>0}\,\liminf_{n\to+\infty}\frac{\m_s^{F_{\e_n}}(u_{x,\zeta,\nu}, Q^\nu(x,r))}{r^{d-1}},\\
    f_{\rm surf}''(x,\zeta,\nu) & :=\limsup_{r\to 0^+}\,\sup_{s>0}\,\limsup_{n\to+\infty}\frac{\m_s^{F_{\e_n}}(u_{x,\zeta,\nu}, Q^\nu(x,r))}{r^{d-1}}.
\end{align*}

\begin{Theorem}\label{theorem: minima}
      Let $\{F_\e\}_\e$ be as in \eqref{def functionals}, and let $\{\e_n\}_n$ be a positive sequence converging to $0$ as $n\to+\infty$. If there exists a functional $F\colon L^1_{\rm loc}(\Omega;\Rm)\times \A(\Omega)\to [0,+\infty]$ such that for every $U\in\A_{\rm reg}(\Omega)$ the sequence $\{F_{\e_n}(\cdot,U)\}_{n}$ $\Gamma$-converges to $F(\cdot,U)$ as $n\to+\infty$, 
   then the functions $f_{\rm bulk}$ and $f_{\rm surf}$ defined in \eqref{Gammabulk} and \eqref{Gammasurf}, respectively, satisfy
  \begin{align*}
    f_{\rm bulk}(x,L) & =f'_{\rm bulk}(x,L)=f''_{\rm bulk}(x,L),\\
   f_{\rm surf}(x,\zeta,\nu)& =f'_{\rm surf}(x,\zeta,\nu)=f''_{\rm surf}(x,\zeta,\nu),
\end{align*}
   for all $x\in\Omega$, $L\in\Rmd$, $\zeta\in\Rm$, and $\nu\in\Sd$.
   
 Conversely, if there exist functions $\hat{f}_{\rm bulk}\colon\Rd\times \R^{m\times d}\to [0,+\infty)$ and $\hat{f}_{\rm surf}\colon\Rd\times \Rm\times \Sd\to[0,+\infty)$ such that 
\begin{equation*} 
    \hat{f}_{\rm bulk}(x,L)=f'_{\rm bulk}(x,L)=f''_{\rm bulk}(x,L),
\end{equation*}
for $\Ld$-a.e.\ $x\in\Omega$ and for all $L\in\Rd$ and
\begin{equation*}
    \hat{f}_{\rm surf}(x,\zeta,\nu)=f'_{\rm surf}(x,\zeta,\nu)=f''_{\rm surf}(x,\zeta,\nu)
\end{equation*}
 for $\Hd$-a.e.\ $x\in\Omega$, for all $\zeta\in\Rm$ and $\nu\in\Sd$,
then for every $U\in\A_{\rm reg}(\Omega)$ the sequence $\{F_{\e_n}(\cdot,U)\}_n$ $\Gamma$-converges as $n\to+\infty$ to the functional $F(\cdot,U)$ defined  by 
\begin{equation}\notag 
   \hspace{-0.15 cm} F(u,U):=\begin{cases}\displaystyle\int_U \!\!\hat{f}_{{\rm bulk}}(x,\nabla u){\rm d}x+\!\!\int_{J_u\cap U} \!\!\!\!\!\hat{f}_{{\rm surf}}(x, [u],\nu_u){\rm d}\Hd \!\!&\text{if }u\in L^1_{\text{\rm loc}}(\Omega;\Rm) {\cap}\, {\rm GSBV}^2(U;\Rm),\\
     +\infty \!\!&\text{if }u\in L^1_{\text{\rm loc}}(\Omega;\Rm) {\setminus}  {\rm GSBV}^2(U;\Rm). 
     \end{cases}
\end{equation}
\end{Theorem}
\begin{proof}
    It is enough to combine Proposition \ref{prop: necessary} and \ref{prop: sufficiency}.
\end{proof}

The proofs of the two main theorems is subdivided as follows: Section \ref{Section: Compactness} is devoted to the proof of the compactness property, Section \ref{sec: integral representation}
is devoted to the proof of the integral representation, and Section \ref{sec:convergence of infima} to the representation of the energy densities as limit of minimum problems.

\section{Reduction to finite-range interactions}\label{section: auxiliary}

In this section we show that the computation of the $\Gamma$-limit of functionals $\{F_\e\}_\e$ given by \eqref{def functionals} can be achieved by first computing the $\Gamma$-limits of functionals $\{F^T_\e\}_\e$, that are obtained considering interactions between pairs of points whose distance is at most $\e T$, and then letting $T$ to $+\infty$. 

For any $T>0$, we define the functionals $F_\eps^T:L^1_{\rm loc}(\Omega;\Rm)\times \A(\Omega) \rightarrow [0,+\infty]$ as
\begin{equation}\label{def truncated functionals}
F^T_\eps(u,U):= \int_{B_T}\int_{U_{\e}(\xi)}f_\e\Big(x,\xi,\frac{u(x+\e\xi)-u(x)}{\e}\Big)\,\dx x\,\dx \xi,
\end{equation}
and, given a positive sequence $\{\e_n\}_n$ converging to $0$, we let
\begin{equation}\label{def Gamma liminf and Gamma limsup T}
F'^{,T}(u,U):= \Gamma\text{-}\liminf_{n \to +\infty} F_{\eps_n} ^T (u,U), \quad F''^{,T}(u,U):= \Gamma\text{-}\limsup_{n \to +\infty} F_{\eps_n}^T (u,U)
\end{equation}
and 
\begin{gather}\label{def Gamma liminf and Gamma limsup}
F'(u,U):= \Gamma\text{-}\liminf_{n \to +\infty} F_{\eps_n} (u,U), \quad F''(u,U):= \Gamma\text{-}\limsup_{n \to +\infty} F_{\eps_n}  (u,U)
\end{gather}
for all $(u,U)\in L^1_{\rm loc}(\Omega;\Rm)\times \A(\Omega)$, where $\Gamma$-liminf and $\Gamma$-limsup are defined as in \eqref{def Gammalimsup}.

It is simple to observe that, due to \eqref{contraction}, all the above functionals are continuous with respect to ``vertical truncations".

\begin{Lemma}\label{lemma:vertical truncations}
    Let $T>0$, $\Phi\colon\Rm\to\Rm$ be a $1$-Lipschitz function, and let  $F'^{,T}$ and $F''^{,T}$ be the functionals defined by \eqref{def Gamma liminf and Gamma limsup T}. Then \begin{gather}\label{monotonic liminf}
    F'^{,T}(\Phi\circ u,U)\leq F'^{,T}(u,U) \quad \text{ and }\quad  F''^{,T}(\Phi\circ u,U)\leq F''^{,T}(u,U)
    \end{gather} for all $(u,U)\in L^1_{\rm loc}(\Omega;\Rm)\times \A(\Omega)$. Moreover, 
    \begin{equation}\label{statement truncation}  F'^{,T}(u,U)=\lim_{M\to+\infty}F'^{,T}(u^M,U) \quad \text{ and } \quad  F''^{,T}(u,U)=\lim_{M\to+\infty}F''^{,T}(u^M,U).
    \end{equation}
In addition, the same conclusions hold for the functionals $F',F''$ defined by \eqref{def Gamma liminf and Gamma limsup}.  
\end{Lemma}
\begin{proof}
    We only prove \eqref{monotonic liminf} and \eqref{statement truncation} for $F''^{,T}$, the other cases being analogous. We consider a sequence of functions $\{u_n\}_n\subset L^1_{\rm loc}(\Omega;\Rm)$ converging to $u$ in $L^1_{\rm loc}(\Omega;\Rm)$ as $n\to+\infty$ and such that 
    \begin{equation*}
        F ''^{,T}(u,U)=\lim_{n\to+\infty}F_{\e_n}^{T}(u_n,U).    \end{equation*}
    Thanks to \eqref{contraction}, it follows immediately that $F^{T}_{\e_n}(\Phi \circ u_n,U)\leq F^{T}_{\e_n}(u_n,U)$ for every $n\in \N$, and therefore, we get
    \begin{equation*}        \limsup_{n\to+\infty}F^{T}_{\e_n}(\Phi\circ u_n,U)\leq \limsup_{n\to+\infty}F^{T}_{\e_n}(u_n,U)=F''^{,T}(u,U).
        \end{equation*}
        Since $\Phi\circ u_n\to \Phi \circ u$ in $ L^1_{\rm loc}(\Omega;\Rm)$ as $n\to+\infty$,
        \eqref{monotonic liminf} follows by the definition \eqref{def Gamma liminf and Gamma limsup T} of $F''^{,T}$.
        \noindent To conclude, we note that, as  $F''^{,T}(\cdot,U)$ is lower semicontinuous with respect to the  $L^1_{\rm loc}(\Omega;\Rm)$ convergence, we also have
    \begin{equation}\notag
        F''^{,T}(u,U)\leq \liminf_{M\to+\infty} F''^{,T}(u^M,U),
    \end{equation}
    which, together with \eqref{monotonic liminf}, implies \eqref{statement truncation}. 
\end{proof}

The following proposition constitutes the main result of this section.
\begin{Proposition}
    \label{proposition: truncated functionals}
Let $(u,U)\in L^1_{\rm loc}(\Omega;\Rm)\times \A_{\rm reg}(\Omega)$. It holds  
$$
F'(u,U)= \lim_{T \to +\infty} F'^{,T}(u,U), \quad \text{and} \quad F''(u,U)= \lim_{T \to +\infty} F''^{,T}(u,U).
$$
In particular, if $\{T_j\}_j$ is a positive sequence such that $T_j \to +\infty$ as $j\to+\infty$ and 
\begin{equation*}
    \underset{n\to+\infty}{\Gamma\text{-}\lim}\,\,F_{\eps_n}^{T_j}(u,U) =  F^{T_j}(u,U)
\end{equation*}
for all $(u,U)\in L^1_{\rm loc}(\Omega;\Rm)\times \A_{\rm reg}(\Omega)$ and $j \in \mathbb{N}$, then we have
$$
\underset{n\to+\infty}{\Gamma\text{-}\lim}\,\, F_{\eps_n}(u,U)= \lim_{j \rightarrow +\infty} F^{T_j}(u,U)
$$
for all $(u,U)\in L^1_{\rm loc}(\Omega;\Rm)\times \A_{\rm reg}(\Omega)$.
\end{Proposition}

The proof of Proposition \ref{proposition:  truncated functionals} requires several preliminary results, which are  obtained by adapting some of the arguments presented in \cite[Chapter~4]{AABPT}. We pursue the same approach in a slightly more general and flexible framework. In the rest of the section, we let $C$ denote a positive constant whose values may change from line to line, for which the dependence on the relevant quantities shall be emphasized.

For every $\e>0$, we let $a_\e:\Rm\to[0,+\infty)$ be a radially increasing function and assume there exist $c>0$ and $p\in[1,+\infty)$ such that
\begin{equation}\label{a growth}
    |a_\e(z)|\leq c|z|^p \quad \text{ and for every } z\in\Rm,
\end{equation}
and that, given $k$ any positive integer and $\{z_1,z_2,...,z_k\}\subset \Rm$, it holds
\begin{equation}\label{a subadd}
a_\eps \Bigl( \sum_{i=1}^k z_i \Bigr) \leq k ^{p-1} \sum_{i=1}^k  a_\eps (z_i).     
\end{equation}
Then we set
\begin{equation}\label{characteristicfunctionals A}
    A^r_\e(u,U):=\int_{B_{r}}\int_{U_{\e}(\xi)}a_\e\Big(\frac{u(x+\e\xi)-u(x)}{\e}\Big)\,\dx x\,\dx \xi,
\end{equation}
for $r>0$ and $(u,U)\in L^1_{\rm loc}(\Rd;\Rm)\times \A(\Rd)$.

\begin{Theorem}\label{extension theorem}
Let $U\subset \Rd$ be a bounded open set with Lipschitz boundary and let $r>0$. There exist a bounded open set $\widetilde{U} \supset\supset U$, a linear continuous map 
$$
E: L^p(U; \Rm) \rightarrow L^p(\widetilde{U};\Rm),
$$
and three positive constants $C_1=C_1(r,U),$ $ r_1=r_1(r,U)$, and  $\eps'=\eps'(r,U)$ such that
$$
Eu=u \quad \text{$\Ld$-a.e.\ in $U$}
$$
and
$$
\| Eu \|^p_{L^p(\widetilde{U};\Rm)}+A_\eps^{r_1}(Eu,\widetilde{U}) \leq C_1 ( \| u\|^p_{L^p(U;\Rm)}+ A_\eps^{r}(u,U))
$$
for all $u \in L^p(U; \Rm)$ and $\eps < \eps'$.
\end{Theorem}
\begin{proof}
Thanks to the Lipschitz regularity of $\partial U$, there exists $\{U_i\}_{i=1}^n$ an open cover of $\partial U$ with the property that for all $i\in\{1,...,n\}$ there exists a bi-Lipschitz continuous map $\Psi_i:Q\to U_i$ such that
\begin{equation*}
    \Psi_i(Q^+)=  U_i \cap U, \quad  \Psi_i(Q^-)= U_i \setminus U, \quad  \Psi_i(Q^0)= U_i \cap \partial U,
\end{equation*}
where $Q:=(-1,1)^d$, $Q^{\pm}:=\{x\in Q : \pm x_1>0\}$, and $Q^0=\{x\in Q:x_1=0\}$. Let then $U_0$ be an open set such that $U \setminus \bigcup_{i=1}^n U_i\subset U_0 \subset U$ and set $\widetilde{U}:=U_0\cup\bigcup_{i=1}^nU_i$.

For $i\in\{1,...,n\}$ we define the map $R_i:  U_i \setminus U\to U_i \cap U$ as 
\begin{equation*}
    R_i(x):= \Psi_i(-y_1, y_2,...,y_d),
\end{equation*}
where $(y_1,...,y_d)=\Psi_i^{-1}(x)\in Q^-$, and given $u\in L^p(U; \Rm)$, we let
\begin{equation*}
    u_i(x):=\begin{cases}
     u(x) & \text{ if } x\in  U_i \cap U, \\
     u(R_i(x)) & \text{ if } x\in U_i \setminus U.
    \end{cases}
\end{equation*}

Consider now $\{\varphi_i\}_{i=0}^n$ a partition of unity on $U$ subordinated to the open cover $\{U_i\}_{i=0}^n$, and let
\begin{equation*} \qquad \widetilde{u}_0(x):=\varphi_0(x)u(x), \quad \widetilde{u}_i(x):=\varphi_i(x)u_i(x) \text{ for all } i\in\{1,...,n\},
\end{equation*}
and finally set $Eu:= \sum_{i=0}^n \widetilde{u_i}$.

By a change of variables, we have that 
\begin{equation}\label{extensionbound}
   \int_{U_i} |u_i|^p\,\dx x \leq C\int_{U\cap U_i} |u|^p\,\dx x
\end{equation}
for all $i\in\{1,\dots, n\}$, where $C$ is a positive constant depending only on $U$. Therefore, we have
\begin{equation}\label{extension1}
    \|Eu\|_{L^p(\widetilde{U};\Rm)}\leq C \|u\|_{L^p(U;\Rm)}
\end{equation} 
for all $u\in L^p(U;\Rm)$, and then 
$$
E: L^p(U; \Rm) \rightarrow L^p(\widetilde{U};\Rm)
$$
is a linear continuous map such that $Eu=u$ a.e. on $U$. 

We claim that there exists a positive constant $C=C(U)$ such that 
\begin{equation}\label{extensionclaim}
    A_\e^{r_1}(u_i, U_i) \leq C  A_\e^{r}(u,  U_i \cap U)
\end{equation}
for all $i\in \{1,...,n\}$, $\e>0$, and $u\in L^p(U;\Rm)$, where 
\begin{equation*}
     r_1:=\frac{r}{1+S^2}, \qquad S:= \max_{i\in\{1,..,n\}} \|D\Psi_i^{-1}\|_{L^\infty(U_i; \R^{d\times d})}.
\end{equation*}
Recalling \eqref{characteristicfunctionals A}, by a change of variables we have
\begin{align} \notag
    A^{r_1}_\e(u_i, U_i) &= \frac{1}{\e^d}\int_{U_i}\int_{U_i\cap B_{\e r_1}(x)} a_\e \Bigl(\frac{u_i(y)-u_i(x)}{\e}\Bigr) \, \dx y\, \dx x \\ \notag
    & \leq A_\e^r(u,U_i \cap U) \\ \label{extensionclaim1}
    & \quad + \frac{1}{\e^d}\int_{U_i\cap U}\int_{(U_i\setminus U) \cap B_{\e r_1}(x)} a_\e \Bigl(\frac{u_i(y)-u_i(x)}{\e}\Bigr) \, \dx y\, \dx x \\ \label{extensionclaim2}
    & \quad + \frac{1}{\e^d}\int_{U_i\setminus U}\int_{U_i\cap U \cap B_{\e r_1}(x)} a_\e \Bigl(\frac{u_i(y)-u_i(x)}{\e}\Bigr) \, \dx y\, \dx x \\ \label{extensionclaim3}
    & \quad + \frac{1}{\e^d}\int_{U_i\setminus U}\int_{(U_i\setminus U) \cap B_{\e r_1}(x)} a_\e \Bigl(\frac{u_i(y)-u_i(x)}{\e}\Bigr) \, \dx y\, \dx x.
\end{align}
 We have that the Lipschitz constant of $R_i$ is less than $S^2$, and then, with our choice of $r_1$, it holds that $R_i((U_i\setminus U) \cap B_{\e r_1}(x))\subseteq U_i \cap U \cap B_{\e r}(x)$ for all $i\in\{1,...,n\}$ and $x\in  U_i$. Therefore, using the change of variables $y'=R_i(y)$, we obtain
\begin{align*}
    & \int_{U_i\cap U}\int_{(U_i\setminus U) \cap B_{\e r_1}(x)} a_\e \Bigl(\frac{u_i(y)-u_i(x)}{\e}\Bigr) \, \dx y\, \dx x \\
    & \leq S^{2d} \int_{U_i\cap U}\int_{U_i\cap U \cap B_{\e r}(x)} a_\e \Bigl(\frac{u(y')-u(x)}{\e}\Bigr) \, \dx y'\, \dx x \\
    & \leq  S^{2d} A_\e^r(u,  U_i \cap U),
\end{align*}
which provides an upper bound for \eqref{extensionclaim1}. With the same argument, one proves that the term in \eqref{extensionclaim2} is bounded above by $S^{2d} A_\e^r(u, U_i \cap U)$ as well, and the term in \eqref{extensionclaim3} is bounded above by $S^{4d} A_\e^r(u, U_i \cap U)$. The combination of these estimates yields \eqref{extensionclaim}.

Finally, for every $i\in\{0,...,n\}$ we set $V_i:=\text{supp }\varphi_i$ and 
$$
\eps'=\eps'(r,U):= \frac{1}{r_1} \min_{i\in\{0,\dots,n\}} \text{dist}(V_i, \Rd\setminus U_i) >0.
$$
If $\eps < \eps'$ and $i\in\{0,\dots,n\}$, summing and subtracting $\phi_i(x+\eps \xi)u_i(x)$ and using \eqref{a subadd}, we get
\begin{align*}
A_\eps^{r_1}(\widetilde {u}_i, \widetilde{U}) & = \int_{B_{r_1}}\!\!\int_{(U_i)_{\e}(\xi)}\!\!\!\!\!a_\e\Big(\frac{\widetilde{u}_i(x+\e\xi)-\phi_i(x+\eps \xi)u_i(x) +\phi_i(x+\eps \xi)u_i(x) -\widetilde{u}_i(x)}{\e}\Big)\,\dx x \,\dx \xi \\
& \leq 2^{p-1} \int_{B_{r_1}} \int_{(U_i)_\eps(\xi)} a_\eps \Bigl( \frac{(u_i(x+\eps \xi)-u_i(x))\phi_i(x+\eps \xi)}{\eps} \Bigr)\, \dx x \, \dx \xi \\
& \quad + 2^{p-1} \int_{B_{r_1}} \int_{(U_i)_\eps(\xi)} a_\eps \Bigl( \frac{(\phi_i(x+\eps \xi)-\phi_i(x))u_i(x)}{\eps} \Bigr) \,\dx  x \, \dx \xi \\
& \leq  2^{p-1} \int_{B_{r_1}} \int_{(U_i)_\eps(\xi)} a_\eps \Bigl( \frac{u_i(x+\eps \xi)-u_i(x)}{\eps} \Bigr)\, \dx x \, \dx \xi \\
& \quad + 2^{p-1} \int_{B_{r_1}} \int_{(U_i)_\eps(\xi)} \Bigl|\frac{\phi_i(x+\eps \xi)-\phi_i(x)}{\e}\Bigr|^p |u_i(x)|^p,
\end{align*}
where in the last inequality we have used that $a_\e$ is radially increasing and \eqref{a growth}. We infer
\begin{equation*}
    A_\eps^{r_1}(\widetilde {u}_i, \widetilde{U})  \leq C(A_\eps^{r_1} (u_i,U_i ) + \|u_i\|^p_{L^p(U_i;\Rm)}),
\end{equation*}
for all $i\in\{0,...,n\}$, where $C=C(r_1,\nabla \phi_0,..,\nabla \phi_n)=C(r,U)$.

If $i=0$, we have
\begin{equation*}
    A_\eps^{r_1}(\widetilde{u}_0, \widetilde{U}) \leq C(A_\eps^r (u,U_0 ) + \|u \|^p_{L^p(U_0; \Rm)}).
\end{equation*}

If $i\in\{1,...,n\}$, we apply \eqref{extensionclaim} and \eqref{extensionbound} to conclude
\begin{align*}
    A_\eps^{r_1}(\widetilde{u}_i, \widetilde{U}) \leq C( A_\eps^r (u, U_i\cap U ) +\|u \|^p_{L^p(U_i \cap U; \Rm)}).
\end{align*}
Finally, summing over $i$ we obtain
$$
A_\eps^{r_1}(Eu, \widetilde{U}) \leq \sum_{i=0}^n A_\eps^{r_1} (\widetilde{u}_i, \widetilde{U}) \leq C ( A_\eps^{r}(u,U) + \| u\|^p_{L^p(U;\Rm)})
$$
for every $\e<\e'$. This estimate, together with \eqref{extension1}, concludes the proof.
\end{proof}

\begin{Remark}\label{rmk : extension} It is easily seen that
\begin{equation*}
    \text{ dist}( \text{supp } Eu, \Rd\setminus\widetilde{U}) > \e'r_1;
\end{equation*}
and therefore,
\begin{equation*}
    A_\eps^{r_1}(Eu,\widetilde{U}) = A_\eps^{r_1}(Eu,\R^d)
\end{equation*}
    for every $\e<\e'$.
As a consequence, we observe that Theorem \ref{extension theorem} remains valid (with the same constants) replacing $\widetilde{U}$ with a larger bounded set.
\end{Remark}

The following result shows that, for $\e$ small enough, it is possible to control the inner integral in \eqref{characteristicfunctionals A} by means of the total energy $A_\e^r$.

\begin{Lemma} \label{lemma: short range interactions}
Let $W$ be a bounded open set. For every $r>0$ there exists a positive constant $C_2=C_2(r)$ such that, for every open set $V \subset W$, $\xi \in \R^d$, and $\eps >0$, with
\begin{equation}\label{condition eps r}
\eps r < \emph{dist}(V+(0,\eps)\xi, \R^d \setminus W),
\end{equation}
it holds
\begin{equation}\label{estimate lemma sr}
\int_V a_\eps \Bigl( \frac{v(x+\eps \xi)-v(x)}{\eps} \Bigr) {\rm d}x \leq C_2 (|\xi|^p +1) A_\eps^r(v,V_{\eps,\xi})
\end{equation}
for every $v\in L^p(W;\Rm)$, 
where $V_{\eps,\xi}:=V+(0,\eps)\xi +B_{\eps r}$ and $(0,\e)\xi:=\{s\xi : s\in (0,\e)\}$.
\end{Lemma}
\begin{proof}
By condition (\ref{condition eps r}) it holds that $V_{\eps,\xi} \subset W $, so that all the terms in  inequality \eqref{estimate lemma sr} are well defined. 

Set $\widetilde{\e}:=\frac{\eps r}{\sqrt{d+3}}$ and let $R_\xi$ be a rotation in $\Rd$ such that $R_\xi e_1= \frac{\xi}{|\xi|}$. We introduce the lattice $\mathcal{L}_\eps:= \{R_\xi i\,:\, i \in \widetilde{\e}\,\mathbb{Z}^d \}$ and, for any $j \in \mathcal{L}_\eps$, we define $Q_\eps^j:= j+ R_\xi (-\widetilde{\e}/2, \widetilde{\e}/2)^d$ and introduce  the indices
\begin{equation*}
    \mathcal{I}_{\e}:=\{j\in \mathcal{L}_\e : Q_\eps^j \cap V \neq \emptyset\}.
\end{equation*}
Let $k:=\lceil \eps |\xi|/\widetilde{\e}\, \rceil$ and, for fixed $j_0 \in \mathcal{I}_\eps$ and $0\leq \ell \leq k-1$, define $j_\ell:= j_0 + \ell \widetilde{\e}\frac{\xi}{|\xi|}$. We let $x_\ell$ denote any point in $Q_\eps^{j_\ell}$ for all $0\leq \ell \leq k-1$, and set $x_k:= x_0+\eps \xi$ and $Q_\eps^{j_k}:=Q_\eps^{j_0}+\eps \xi$.

Given points $x_0,x_1,...,x_k=x_0+\e\xi$ in accordance with the above notation, we use \eqref{a subadd} with $z_\ell=(v(x_\ell)-v(x_{\ell-1}))/\e$ for $\ell\in\{1,...,k\}$, and we get
$$
a_\eps \Bigl( \frac{v(x_0+\eps\xi)-v(x_0)}{\eps}\Bigr) \leq k^{p-1} \sum_{\ell=1}^k a_\eps \Bigl( \frac{v(x_\ell)-v(x_{\ell-1})}{\eps}\Bigr).
$$
Integrating this inequality with respect to the variables $(x_0,\dots,x_{k-1})$, we obtain 
\begin{align*}
 \int_{Q_\eps^{j_0}} & a_\eps \Bigl( \frac{v(x_0+\eps\xi)-v(x_0)}{\eps}\Bigr)\,\dx x_0 \\ 
 & \leq \frac{k^{p-1}}{(\widetilde{\e})^d} \sum_{\ell=1}^{k} \int_{Q_\eps^{j_{\ell-1}}}\int_{Q_\eps^{j_{\ell}}}a_\eps \Bigl( \frac{v(x_\ell)-v(x_{\ell-1})}{\eps} \Bigr)\,\dx x_\ell \, \dx x_{\ell-1} \\
&\leq \frac{k^{p-1}}{(\widetilde{\e})^d} \sum_{\ell=1}^{k} \int_{Q_\eps^{j_{\ell-1}}}\int_{B_{\eps r}(x_{\ell-1})}a_\eps \Bigl( \frac{v(y)-v(x_{\ell-1})}{\eps}  \Bigr)\,\dx y \,\dx x_{\ell-1},
\end{align*}
where in the last line of the above inequality we used that for every $\ell\in\{1,...,k\}$ it holds
$Q_\eps^{j_\ell} \subset B_{\eps r}(x_{\ell-1}),$
regardless of the choice of the point $x_{\ell-1}\in Q^{j_{\ell-1}}_\e$. Then, by the change of variables $\xi':=(y-x_{\ell-1})/\e$ and Fubini's Theorem, it follows
\begin{align*}
 \int_{Q_\eps^{j_0}} & a_\eps \Bigl( \frac{v(x_0+\eps\xi)-v(x_0)}{\eps}\Bigr)\,\dx x_0 \\ &\leq \left(\frac{\eps}{\widetilde{\e}}\right)^d k^{p-1} \sum_{\ell=1}^k \int_{B_r} \int_{Q_\eps^{j_{\ell-1}}} a_\eps \Bigl(\frac{v(x_{\ell-1} + \eps\xi')- v(x_{\ell-1})}{\eps} \Bigr)\, \dx x_{\ell-1}\, \dx \xi'  \\
 & =\Bigl(\frac{\sqrt{d+3}}{r}\Bigr)^d k^{p-1} \int_{B_r} \int_{S_\eps(j_0)} a_\eps \Bigl(\frac{v(x+\eps \xi')- v(x)}{\eps} \Bigr)\, \dx x\, \dx \xi',
\end{align*}
where $S_\eps(j_0):= \cup_{\ell=0}^{k-1} Q_\eps^{j_\ell}$ and we note that $S_\eps(j_0)\subset V_{\e,\xi}$ for every $j_0\in\mathcal{I_\e}$. 

Since the sets $\{S_\eps(j_0):j_0 \in \mathcal{I}_\eps \}$ overlap at most $2k-1$ times, summing over the indices $j_0 \in \mathcal{I}_\eps$ and recalling that $k=\lceil \eps |\xi|/\widetilde{\e}\,\rceil$, we get 
$$
 \int_V a_\eps \Bigl(\frac{v(x+\eps \xi)-v(x)}{\eps} \Bigr)\, \dx x \leq C (|\xi|^p +1) A_\eps^r(v, V_{\eps,\xi})
$$
and the proof is concluded.
\end{proof}

Combining Lemmas \ref{extension theorem} and \ref{lemma: short range interactions} we obtain the following result.

\begin{Corollary}\label{corollary : short range}
Let $U\in\A_{\rm reg}(\Omega)$ and let $r>0$. There exist two positive constants $C_3=C_3(r,U)$ and $\eps''=\eps''(r,U)$ such that
$$
\int_{U_\e(\xi)} a_\eps \Bigl(\frac{u(x+\eps\xi)-u(x)}{\eps} \Bigr) {\rm d} x \leq C_3(|\xi|^p+1) ( \| u\|^p_{L^p(U;\Rm)}+A_\eps^r(u,U))
$$
for every $\xi \in \R^d$, $u \in L^p(U;\Rm)$, and $\eps < \eps''$.
\end{Corollary}
\begin{proof}
Let $\widetilde{U},r_1$, and $\e'$ be as in the statement of Theorem \ref{extension theorem} and, in light of Remark \ref{rmk : extension}, suppose that $\text{co}(U)\subset\subset \widetilde{U}$, where we let co$(U)$ denote the closed convex hull of $U$, and let 
\begin{equation*}
    \e''=\e''(r,U):= \min\Bigl\{\e',\frac{\text{dist}(\text{co}(U), \R^d\setminus \widetilde{U})}{r_1}\Bigr\}>0.
\end{equation*}

Since we have that
\begin{equation*}
    U_\e(\xi)+(0,\e)\xi \subseteq \text{co}(U)
\end{equation*}
for every $\xi\in\Rd$ and $\e>0$, it holds that
\begin{equation*}
    \e r_1< \text{dist}(U_\e(\xi)+(0,\e)\xi, \Rd \setminus \widetilde{U})
\end{equation*}
for every $\xi\in\R^d$ and $\e<\e''$. We get that the set $(U_\e(\xi))_{\eps,\xi}=U_\e(\xi)+(0,\e)\xi+B_{\e r_1}$ is contained in $\widetilde{U}$; and then, applying Lemma \ref{lemma: short range interactions} with 
\begin{equation*}
r=r_1, \quad W=\widetilde{U}, \quad   V=U_\e(\xi), \quad v=Eu,  
\end{equation*} 
and Theorem \ref{extension theorem}, we have   
\begin{align*}
    \int_{U_\e(\xi)} a_\eps \Bigl(\frac{u(x+\eps\xi)-u(x)}{\eps} \Bigr)\, \dx x & \leq C_2(|\xi|^p+1) A_\e^{r_1}(Eu, (U_\e(\xi))_{\eps,\xi}) \\
    & \leq C_2(|\xi|^p+1) A_\e^{r_1}(Eu, \widetilde{U}) \\
    & \leq C_1C_2 (|\xi|^p+1)( \| u\|^p_{L^p(U;\Rm)}+ A_\eps^{r}(u,U)),
\end{align*}
which yields the thesis.   
\end{proof}

We introduce further reference functionals. For every $\e,r>0$ and $(u,U)\in L^1_{\rm loc}(\Rd;\Rm)\times \A(\Rd)$ we set 
\begin{align}\notag 
G^r_\e(u,U)&:= \int_{B_{r}}\int_{U_{\e}(\xi)}\min\Bigl\{\Bigl|\frac{u(x+\e\xi)-u(x)}{\e}\Big|^2, \frac{1}{\e}\Bigr\}\,\dx x\,\dx \xi \\& \label{characteristicfunctionals G} \hphantom{:}
=\int_{B_{r}}\int_{U_{\e}(\xi)}g_\e\Bigl(\frac{u(x+\e\xi)-u(x)}{\e}\Bigr)\dx x\,\dx \xi,
\end{align}
and
\begin{equation*}
    P^r_\e(u,U):=\int_{B_{r}}\int_{U_{\e}(\xi)}\Bigl|\frac{u(x+\e\xi)-u(x)}{\e}\Bigr|\,\dx x\,\dx \xi,
\end{equation*}
and finally 
\begin{equation*}
    H_\e^r(u,U):= G_\e^r(u,U)+P_\e^r(u,U).
\end{equation*}
According to the following elementary lemma, $g_\e$ fulfils \eqref{a growth} and \eqref{a subadd} with $p=2$ and, therefore, all the previous results apply to functionals $\{G^r_\e\}_\e$. The same conclusion clearly holds for $\{P_\e^r\}_\e$ with $a_\e(z)=|z|$.

\begin{Lemma}\label{lemma: stima elementare} Let $k$ be a positive integer and let $\{z_1,z_2,...,z_k\}\subset \Rm$. Then for every $\e>0$ we have
\begin{equation*}
g_\eps \Bigl( \sum_{i=1}^k z_i \Bigr) \leq k  \sum_{i=1}^k  g_\eps (z_i).     
\end{equation*}
\end{Lemma}
\begin{proof}
If there exists an index $i^*\in\{1,...,k\}$ such that $g_\eps(z_{i^*})=\frac{1}{\eps}$, then 
$$
k\sum_{i=1}^k  g_\eps (z_i) \geq \frac{1}{\eps} \geq \min\Bigl\{\Bigl| \sum_{i=1}^k z_i \Bigr|^2 , \frac{1}{\eps}  \Bigr\}  = g_\eps \Bigl( \sum_{i=1}^k z_i \Bigr),
$$
otherwise, 
\begin{equation*}
k\sum_{i=1}^k  g_\eps (z_i)  = k \sum_{i=1}^k |z_i|^2 
\geq \Bigl|\sum_{i=1}^k z_i \Bigr|^2 
 \geq \min\Bigl\{ \Bigl| \sum_{i=1}^k z_i \Bigr|^2 , \frac{1}{\eps}  \Bigr\} = g_\eps \Bigl( \sum_{i=1}^k z_i \Bigr),    
\end{equation*}
which concludes the proof.
\end{proof}

The next proposition shows that for sequences of functions $\{u_\e\}_\e$ that are essentially bounded, a uniform bound on $G_\e^r(u_\e,\Omega)$ implies a uniform bound on $H_\e^r(u_\e,\Omega)$.  
\begin{Proposition}\label{prop: interpolation} Let $u\in L^\infty(\Omega;\Rm),\,U\in\A(\Omega),$ and $\e,r>0$. We have
\begin{equation*}
    H_\e^r(u, U) \leq (1+2\|u\|_{L^\infty(\Omega;\Rm)}) G_\e^r(u,U)+(|U||B_r|G_\e^r(u, U))^{1/2}.
\end{equation*}
\end{Proposition}
\begin{proof}
 We claim that
\begin{align}\notag
     & \int_{B_r}\int_{U_\e(\xi)} \frac{|u(x+\e\xi)-u(x)|}{\e}\,\dx x \,\dx \xi \\ \notag
& \leq|U|^{1/2}|B_r|^{1/2} \Bigl( \int_{B_r}\int_{U_\e(\xi)} \min\Bigl\{\Bigl|\frac{u(x+\e\xi)-u(x)}{\e}\Bigr|^2, \frac{1}{\e}\Bigr\}\,\dx x \, \dx \xi\Bigr)^{1/2} \\ \label{interpolation}
    & \quad +2\|u\|_{L^\infty(\Omega;\Rm)}\int_{B_r} \int_{U_\e(\xi)} \min\Bigl\{\Bigl|\frac{u(x+\e\xi)-u(x)}{\e}\Bigr|^2, \frac{1}{\e}\Bigr\}\,\dx x \, \dx \xi.
\end{align}

To prove this, we let 
\begin{equation*}
    U(\e,\xi):=\{x\in U_\e(\xi): |u(x+\e\xi)-u(x)|\leq \e^{1/2}\}
\end{equation*}
for all $\e>0,\xi\in\Rd$. By H\"older's inequality, we infer that 
\begin{equation}\label{interpolation1}
    \int_{U(\e,\xi)} \frac{|u(x+\e\xi)-u(x)|}{\e}\,\dx x \leq |U|^{1/2}  \Bigl(\int_{U(\e,\xi)} \Bigl|\frac{u(x+\e\xi)-u(x)}{\e}\Bigr|^2\,\dx x\Bigr)^{1/2}
\end{equation}
and, by the boundedness of $u$,
\begin{equation}\label{interpolation2}
    \int_{U_\e(\xi)\setminus U(\e,\xi)} \frac{|u(x+\e\xi)-u(x)|}{\e}\,\dx x \leq 2\|u\|_{L^\infty(\Omega;\Rm)} \int_{U_\e(\xi)\setminus U(\e,\xi)} \frac{1}{\e}\,\dx x.
\end{equation}
Summing \eqref{interpolation1} and \eqref{interpolation2} and integrating with respect to $\xi$, we apply Jensen's inequality to get
\begin{align*}
    \int_{B_r} \int_{U_\e(\xi)} \frac{|u(x+\e\xi)-u(x)|}{\e}\,\dx x \,\dx \xi 
    & \leq |U|^{1/2} \int_{B_r} \Bigl(\int_{U(\e,\xi)} \Bigl|\frac{u(x+\e\xi)-u(x)}{\e}\Bigr|^2\dx x\Bigr)^{1/2} \,\dx \xi\\
    & \quad +2\|u\|_{L^\infty(\Omega;\Rm)}\int_{B_r} \int_{U_\e(\xi)\setminus U(\e,\xi)} \frac{1}{\e}\,\dx x\,\dx \xi \\
    & \leq |U|^{1/2}|B_r|^{1/2}\Bigl( \int_{B_r}\int_{U(\e,\xi)} \Bigl|\frac{u(x+\e\xi)-u(x)}{\e}\Bigr|^2\,\dx x \, \dx \xi\Bigr)^{1/2} \\
    & \quad +2\|u\|_{L^\infty(\Omega;\Rm)} \int_{B_r} \int_{U_\e(\xi)\setminus U(\e,\xi)} \frac{1}{\e}\,\dx x\, \dx \xi.
\end{align*}
This proves the claim.

By \eqref{interpolation}, we conclude
\begin{align*}
    H_\e^r(u, U) & = G_\e^r(u,U)+\int_{B_r} \int_{U_\e(\xi)} \frac{|u(x+\e\xi)-u(x)|}{\e}\,\dx x \,\dx \xi \\
    & \leq G_\e^r(u,U)+ |U|^{1/2}|B_r|^{1/2} (G_\e^r(u, U))^{1/2} + 2\|u\|_{L^\infty(\Omega;\Rm)}G_\e^r(u, U),
\end{align*}
which is the thesis.
\end{proof}

The last preliminary result will also be useful in Section \ref{sec:convergence of infima}.

\begin{Lemma}\label{lemma F<FT}
    Let $U\in \A_{\rm reg}(\Omega)$, let $r_0>0$ be as in \eqref{non zero in zero}, and let $\e''=\e''(r_0,U)$ be as in Corollary \ref{corollary : short range}.
    Consider a sequence $\{u_n\}_{n}\subset L^\infty(\Omega;\Rm)$ such that 
    \begin{equation}\label{boundedness S}\sup_{n} \big(\|u_n\|_{L^\infty(\Omega;\Rm)} + F^T_{\e_n}(u_n,U)\big)< +\infty,
    \end{equation}
    for some $T>r_0$.
    Then, $S\coloneqq \sup_n H^{r_0}_{\e_n}(u_n,U)< + \infty$. Moreover, there exists a positive constant $C^*=C^*(S,\,\sup_n\|u_n\|_{L^\infty(\Omega;\Rm)})$ such that for every $\delta>0$ there exist $T^* = T^*(\delta)>r_0$ and $\e^*=\e^*(\delta)$ such that
\begin{equation}\label{mainclaimlocalisation}
F_{\eps_n}(u_n,U) \leq F^T_{\eps_n}(u_n,U)+ C^*\delta
\end{equation}
for every $T>T^*$ and for every $\e_n<\e^*$.
\end{Lemma}
\begin{proof}
    In order to prove the first part of the statement, 
    we observe that, by  \eqref{non zero in zero}, \eqref{comparison densities}, and the fact that $T>r_0$, we have
\begin{equation*}
    c_0 G_{\eps_n}^{r_0}(u_n, U)\leq F^T_{\eps_n}(u_n,U);  
\end{equation*}
therefore, taking into account \eqref{boundedness S}, we deduce $\sup_n G_{\eps_n}^{r_0}(u_n, U)<+\infty$, and consequently we infer $S:=\sup_n H_{\eps_n}^{r_0}(u_n, U)<+\infty$ in light  of Proposition \ref{prop: interpolation}.
    
 To prove \eqref{mainclaimlocalisation} we write 
\begin{equation}\notag
F_{\eps_n}(u_n,U) = F_{\eps_n}^T(u_n,U) + \int_{\R^d \setminus B_T} \int_{U_{\eps_n}(\xi)} f_{\e_n} \Big(x,\xi, \frac{u_n(x+\eps_n\xi)-u_n(x)}{\eps_n} \Big)\,\dx x\, \dx\xi
\end{equation}
and, by \eqref{comparison densities}, we obtain  
\begin{align*}\notag
F_{\eps_n}(u_n,U) & \leq  F_{\eps_n}^T(u_n,U) + \int_{\R^d \setminus B_T} \rho_2(\xi) \int_{U_{\eps_n}(\xi)} g_{\eps_n}\Big(\frac{u_n(x+\eps_n\xi)-u_n(x)}{\eps_n} \Big) \dx x\, \dx \xi \\ &  \quad +\int_{\R^d \setminus B_T} \psi_{\e_n}(\xi) \int_{U_{\eps_n}(\xi)} \Bigl|\frac{u_n(x+\eps_n\xi)-u_n(x)}{\eps_n} \Bigr|\, \dx x\, \dx \xi +|\Omega|\int_{\R^d\setminus B_T}\eta_{\eps_n}(\xi)\,\dx\xi.
\end{align*}
Applying Corollary \ref{corollary : short range} with $r=r_0, a_\e(z)=g_\e(z)$, and with $r=r_0, a_\e(z)=|z|$, we get that for every $\e_n<\e''$ it holds
\begin{align*}\notag
F_{\eps_n}(u_n,U) & \leq F_{\eps_n}^T(u_n,U) + C_3(G_{\eps_n}^{r_0}(u_n,U) + \| u_n\|^2_{L^2(U;\Rm)})\int_{\R^d \setminus B_T} \rho_2(\xi)
(|\xi|^2+1)\,\dx \xi  \\ \notag
& \quad  + C_3(P_{\e_n}^{r_0}(u_n,U)+ \| u_n\|_{L^1(U;\Rm)})\int_{\R^d \setminus B_T} \psi_{\e_n}(\xi)(|\xi|+1)\,\dx \xi + |\Omega|\int_{\R^d\setminus B_T}\eta_{\eps_n}(\xi)\,\dx\xi \\ \notag
& \leq F_{\eps_n}^T(u_n,U) + C(G_{\eps_n}^{r_0}(u_n,U) + \|u_n\|^2_{L^\infty(\Omega;\Rm)})\int_{\R^d \setminus B_T} \rho_2(\xi)
|\xi|^2\,\dx \xi \\ 
& \quad + C(P_{\e_n}^{r_0}(u_n,U)+ \| u_n\|_{L^\infty(\Omega;\Rm)})\int_{\R^d \setminus B_T} \psi_{\e_n}(\xi)|\xi|\,\dx \xi + |\Omega|\int_{\R^d\setminus B_T}\eta_{\eps_n}(\xi)\,\dx\xi \\
& \leq  F_{\eps_n}^T(u_n,U) + C(H_{\eps_n}^{r_0}(u_n,U) + \|u_n\|^2_{L^\infty(\Omega;\Rm)}+\| u_n\|_{L^\infty(\Omega;\Rm)}+|\Omega|) \\
& \quad \times \int_{\R^d\setminus B_T}\rho_2(\xi)
|\xi|^2 + \psi_{\e_n}(\xi)|\xi| + \eta_{\eps_n}(\xi)\,\dx\xi.
\end{align*}

By \eqref{finite momenta}, \eqref{psi infinito}, and \eqref{eta infinito} there exists $T^*=T^*(\delta)$ such that
\begin{equation*}
\limsup_{n\to+\infty} \int_{\R^d\setminus B_{T^*}} \rho_2(\xi)|\xi|^2+\psi_{\e_n}(\xi)|\xi| + \eta_{\eps_n}(\xi)\,\dx\xi <\delta 
\end{equation*}
and therefore, there exists $\e^*(T^*)=\e^*(\delta)$ such that
\begin{equation*}
\int_{\R^d\setminus B_{T^*}} \rho_2(\xi)|\xi|^2+\psi_{\e_n}(\xi)|\xi| + \eta_{\eps_n}(\xi)\,\dx\xi <2\delta 
\end{equation*}
for every $\e_n<\e^*$; hence, we obtain \begin{equation*}
F_{\eps_n}(u_n,U)  \leq F_{\eps_n}^T(u_n,U) + C \bigl(S + |\Omega|+\sup_{n}(\|u_n\|^2_{L^\infty(\Omega; \Rm)}+ \|u_n\|_{L^\infty(\Omega; \Rm)})\bigr) \delta
\end{equation*}
for every $T>T^*$ and $\e_n<\e^*$, concluding the proof.
\end{proof}

    Before proceeding with the proof of Proposition \ref{proposition: truncated functionals}, we make an observation that will come to use multiple times.
\begin{Remark}
    \label{L infty recovery}
   
    Consider $u \in L^\infty(\Omega;\Rm)$ and let $\{u_n\}_n$ be such that
    \begin{equation*}
    u_n\to u \text{ in } L^1_{\rm loc}(\Omega;\Rm) \quad \text{ and } \quad  F_{\e_n}(u_n, U) \to F'(u,U) \quad \text{ as } n\to+\infty.
    \end{equation*}
     In light of \eqref{contraction}, by a truncation argument we can always assume that for every $n\in\N$ we have $\|u_n\|_{L^\infty(\Omega;\Rm)}\leq \sqrt{m}\|u\|_{L^\infty(\Omega;\Rm)}$. 
    Note that an analogous property holds if we replace the $\Gamma$-$\liminf$ with the $\Gamma$-$\limsup$; that is, considering $F''$ in place of $F'$. Moreover, the same conclusions hold for $F'^{,T}$ and $F''^{,T}$ if we consider $\{F^T_\e\}_\e$ in place of $\{F_\e\}_\e$.
\end{Remark}

We are finally ready to prove Proposition \ref{proposition: truncated functionals}.

\begin{proof}[Proof of Proposition \ref{proposition: truncated functionals}.]

We only prove the statement for $F'$, the other case being analogous.

Since $F_\eps^T(u,U) \leq F_\eps(u,U)$ for every $u \in L^1_{\rm loc}(\Omega;\Rm)$ and $\e,T>0$, one inequality is trivial; therefore, we only prove the opposite inequality assuming first that $u\in L^\infty(\Omega;\Rm)$. 

Given $T>r_0$, let $\{u_n\}_n\subset L^1_{\rm loc}(\Omega;\Rm)$ be such that
\begin{equation}\label{truncationrecovery}
    u_n\to u \text{ in } L^1_{\rm loc}(\Omega; \Rm) \quad \text{ and } \quad  F_{\e_n}^T(u_n, U) \to F'^{,T}(u,U) \quad \text{ as } n\to+\infty.
\end{equation}
It is not restrictive to suppose that $F'^{,T}(u,U)<+\infty$ and, by Remark \ref{L infty recovery}, we may also assume 
\begin{equation}\label{uniformboundrecovery}
    \sup_n\|u_n\|_{L^\infty(\Omega; \Rm)}\leq \sqrt{m}\|u\|_{L^\infty(\Omega; \Rm)}.
    \end{equation}
    Thanks to \eqref{truncationrecovery} and \eqref{uniformboundrecovery},  inequality \eqref{boundedness S} holds. Hence, by Lemma \ref{lemma F<FT}, for any $\delta>0$ it holds
\begin{equation*}
F_{\eps_n}(u_n,U) \leq F^T_{\eps_n}(u_n,U)+ C^*\delta
\end{equation*}
for every $T>T^*$ and $n$ large enough. We now pass to the liminf as $n\to+\infty$ and obtain 
\begin{equation*}
    F'(u,U)  \leq  F'^{,T}(u,U) + C^*\delta,
\end{equation*}
for every $\delta>0,\, T>T^*$;
 the conclusion follows letting first $T\to+\infty$ and then $\delta\to 0^+$.

If $u\in L^1_{\rm loc}(\Omega;\Rm)$, let $\{u^M\}_{M}$ be the sequence obtained by truncation. By the previous step we have that 
\begin{equation*}
    F'(u^M,U)= \lim_{T \to +\infty} F'^{,T}(u^M,U) = \sup_T\,  F'^{,T}(u^M,U), 
\end{equation*}
hence, using Lemma \ref{lemma:vertical truncations}, we obtain
\begin{align*}
   F'(u,U) & = \sup_M\, F'(u^M,U)  = \sup_M\, \sup_T\,  F'^{,T}(u^M,U) \\
   & =  \sup_T\, \sup_M\,  F'^{,T}(u^M,U) 
    =  \sup_T\, F'^{,T}(u,U) 
    = \lim_{T\to+\infty}\, F'^{,T}(u,U),
\end{align*}
which concludes the proof.
\end{proof}

\section{Compactness}\label{Section: Compactness} 
For this section, let $\{\e_n\}_n$ be a positive sequence converging to $0$ and recall the definitions of $F', F'', F'^{,T}, F''^{,T}$ given in Section \ref{section: auxiliary}. In next two sections we shall prove Theorem \ref{thm:compactness}. For the reader's convenience, 
we divide the proof in two different propositions, proving separately the compactness and the integral representation of the obtained $\Gamma$-limit. In this section, we address the compactness part, which is the content of Proposition \ref{prop:compactness}.

In this section, we let $C$ denote a positive constant whose values may change from line to line.

\noindent We begin by proving that the functionals $F'$ and $F''$ introduced in \eqref{def Gamma liminf and Gamma limsup} satisfy suitable upper and lower bounds.

\begin{Proposition}\label{prop: referencebound} Let $U\in \mathcal{A}_{\rm reg}(\Omega)$ and let $\{F_\e\}_\e$ be defined as in \eqref{def functionals}. There exist positive constants $c_1, c_2, \Lambda$ independent of $U$ such that 
\begin{gather} \label{referenceboundlower}
        c_1{\rm MS}_1(u,U)\leq F'(u,U) \text{ for all $u\in L^1_{\rm loc}(\Omega;\R^m)$},\\    \label{referenceboundhigher}
          F''(u,U)\leq c_2 H(u,U)+\Lambda|U| \text{ for all }u\in L^1_{\rm loc}(\Omega;\Rm)\cap{\rm BV}(U;\Rm).
    \end{gather}
    In particular, if $F'(u,U)<+\infty,$ then $u\in {\rm GSBV}^2(U; \Rm)$. 
\end{Proposition}
\begin{proof}
    The proof immediately follows by \eqref{comparison functionals}, \eqref{lower bound geps}, and \eqref{upper bound heps}.
\end{proof}

In the proof of Proposition \ref{prop:compactness} below, we shall use the localisation method for $\Gamma$-convergence, which relies on De Giorgi-Letta criterion for measures, which requires one to prove superadditivity, subadditivity, and inner regularity. To this aim, 
we first prove that $F''$ satisfies a version of the {\it nested subadditivity} property. 

\begin{Lemma}\label{fundamental estimate truncated}
Let $u\in L^1_{\rm loc}(\Omega; \Rm)$, let $U',U, V', V \in \mathcal{A}(\Omega)$ such that $U' \subset\subset U$, $V' \subset\subset V$, and $U'\cup V'\in \A_{ \rm reg}(\Omega)$. Let $F''$ be the functional defined by \eqref{def Gamma liminf and Gamma limsup}: it holds
\begin{equation*}
    F''(u, U' \cup V') \le F'' (u, U) + F'' (u, V).
\end{equation*}
\end{Lemma}
\begin{proof}
Assume first that $u\in L^\infty(\Omega; \Rm)$ and, without loss of generality, also assume that $F''(u,U)$ and $F''(u,V)$ are both finite.
Consider two sequences  $u_n, v_n \to u$ in $L^1_{\rm loc}(\Omega; \Rm)$ such that
\begin{align*}
   \lim_{n \to+\infty} F_{\eps_n} (u_n, U)= F''(u, U)\  \quad \text{ and } \quad 
   \lim_{n\to +\infty} F_{\eps_n} (v_n, V)=  F''(u, V),
\end{align*}
and, in light of Remark \ref{L infty recovery}, assume that 
\begin{equation*}
\sup_n\|u_n\|_{L^{\infty}(\Omega; \Rm)}\leq \sqrt{m}\|u\|_{L^\infty(\Omega; \Rm)}, \quad \sup_n\|v_n\|_{L^{\infty}(\Omega; \Rm)}\leq \sqrt{m}\|u\|_{L^\infty(\Omega; \Rm)}.   
\end{equation*}
Fix $0< R < \min\{\text{dist}(U', \R^d\setminus U),\text{dist}(V', \R^d\setminus V)\}$, let $N>2$ be a positive integer, and let
\begin{align*}
    U_i\colon = \Big\{x \in \R^d \colon \text{dist}(x, U') < \frac{iR}{N}\Big\}, \quad i\in\{1,...,N\}.
\end{align*}
For all $i\in\{1,...,N-1\}$ we note that $U' \subset\subset U_i \subset\subset U_{i+1} \subset\subset U$, and we consider a cut-off function $\phi_i \in C^\infty_c (\R^d; [0,1])$ such that
\begin{equation}\label{properties of cutoffs}
\begin{cases}
     \phi_i= 1 &  \text{in }  \overline{U}_i,\\
      \phi_i= 0 & \text{in } \R^d \setminus U_{i+1},\\
       |\nabla \phi_i| \le \frac{2N}{R} & \text{in } \Rd.
\end{cases}
\end{equation}
Set $w_n^i:= \phi_iu_n + (1-\phi_i)v_n$ and note that, for any $\xi\in \Rd, n\in\N,$ and $x\in\Omega_{\e_n}(\xi)$, we have 
\begin{align}\notag
    \frac{w_n^i(x+\eps_n \xi)- w_n^i (x)}{\eps_n} & = \phi_i(x) \frac{u_n(x+\eps_n \xi)- u_n (x)}{\eps_n} + (1-\phi_i(x))\frac{v_n(x+\eps_n \xi)- v_n (x)}{\eps_n} \\\label{decomposition of w}
    & \quad + \frac{\phi_i(x+\eps_n \xi)- \phi_i (x)}{\eps_n} \left(u_n(x+\eps_n \xi) - v_n (x+\eps_n \xi) \right)
\end{align}
and, in particular,
\begin{equation}\label{specifying the values of weps}
    \dfrac{w_n^i(x+\eps_n \xi)- w_n^i (x)}{\eps_n} =
     \begin{cases}\displaystyle
      \frac{u_n(x+\eps_n \xi)- u_n (x)}{\eps_n} & \text{ if } x \in (U_i)_{\e_n}(\xi), \\[6pt] \displaystyle
      \frac{v_n(x+\eps_n\xi)- v_n (x)}{\eps_n} & \text{ if } x \in (\Omega \setminus \overline{U}_{i+1})_{\eps_n}(\xi).
    \end{cases}
 \end{equation}
We observe that, using \eqref{properties of cutoffs} and \eqref{decomposition of w}, we have
\begin{align}\notag  
 \Big|\dfrac{w_n^i(x+\eps_n \xi)- w_n^i (x)}{\eps_n}\Big| &\le \Big|\dfrac{u_n(x+\eps_n \xi)- u_n (x)}{\eps_n}\Big|
 + \Big|\dfrac{v_n(x+\eps_n \xi)-v_n(x)}{\eps_n}\Big| \\ \ \label{triangolare}
 &\quad + \Big(\frac{2N}{R}\Big)|\xi||u_n(x+\eps_n \xi) - v_n (x+\eps_n \xi)|;
 \end{align}
and similarly, recalling the definition of $g_\e$ and exploiting Lemma \ref{lemma: stima elementare}, we get
 \begin{align}\notag  
 g_{\eps_n} \Big(\dfrac{w_n^i(x+\eps_n \xi)- w_n^i (x)}{\eps_n}\Big) &\le 3 g_{\eps_n}\Big(\dfrac{u_n(x+\eps_n \xi)- u_n (x)}{\eps_n}\Big) 
 + 3g_{\eps_n}\Big(\dfrac{v_n(x+\eps_n \xi)-v_n(x)}{\eps_n}\Big) \\ \ \label{application of k convexity}
 &\quad + 3\Big(\frac{2N}{R}\Big)^2|\xi|^2|u_n(x+\eps_n \xi) - v_n (x+\eps_n \xi) |^2,
 \end{align}
for every $\xi\in\Rd$, $n\in\N$, and $x\in \Omega_{\e_n}(\xi)$.

For $i\in\{1,...,N-1\}$ we set  
\begin{equation}\notag
S_{n, \xi}^i :=(U' \cup V')_{\eps_n} (\xi) \setminus \bigl[( {U_i})_{\eps_n}(\xi) \cup (\Omega \setminus \overline{U}_{i+1})_{\eps_n}(\xi) \bigr]\end{equation} 
and we observe that if $\xi\in B_T$ for some $T>0$ and $\e_n T< \frac{R}{N}$, then 
\begin{equation}
\label{inclusione gusci}
S_{n, \xi}^i \subseteq [(U_{i+1}\setminus \overline{U}_{i})+(-\e_n,\e_n)\xi]\cap (V')_{\e_n}(\xi),
\end{equation}
and, in particular,
\begin{equation}\label{inclusione strip}
    \bigcup_{i=1}^{N-4} S_{n, \xi}^i \subseteq (U_{N-2}\setminus \overline{U'})\cap V'.
\end{equation}
Let $T>0$ be fixed. By \eqref{specifying the values of weps} and \eqref{comparison densities}, we have
 \begin{align*}
  F_{\eps_n}^T (w_n^i, U'\cup V') & = \int_{B_T} \int_{(U_i\cap( U'\cup V'))_{\eps_n}(\xi)} f_{\e_n} \Bigl(x, \xi, \dfrac{u_n(x+\eps_n\xi)- u_n (x)}{\eps_n}\Bigr) \dx x\, \dx \xi \\
 & \quad + \int_{B_T} \int_{((\Omega\setminus \overline{U}_{i+1})\cap V')_{\e_n}(\xi)} f_{\e_n} \Big(x, \xi, \dfrac{v_n(x+\eps_n \xi)- v_n (x)}{\eps_n} \Big) \dx x\, \dx\xi  \\
 & \quad + \int_{B_T} \int_{S_{n, \xi}^i} f_{\e_n} \Big(x, \xi, \dfrac{w_n^i(x+\eps_n \xi)- w_n^i (x)}{\eps_n} \Big) \dx x\, \dx \xi \\
 & \leq F_{\e_n} (u_n, U) + F_{\e_n} (v_n, V) \\
 & \quad + \int_{B_T} \rho_2(\xi) \int_{S_{n, \xi}^i}  g_{\e_n} \Bigl(\dfrac{w_n^i(x+\eps_n \xi)- w_n^i (x)}{\eps_n}\Bigr)\dx x\, \dx \xi \\
 & \quad + \int_{B_T} \psi_{\e_n}(\xi) \int_{S_{n, \xi}^i}  \Bigl|\dfrac{w_n^i(x+\eps_n \xi)- w_n^i (x)}{\eps_n}\Bigr|\dx x\, \dx \xi \\
 & \quad +\int_{B_T}\eta_{\e_n}(\xi)|S^i_{n,\xi}|\dx \xi,
\end{align*}
and then, using \eqref{application of k convexity} and \eqref{triangolare}, we get 
 \begin{align}\notag
  F_{\e_n}^T (w_n^i, U'\cup V') \leq & F_{\e_n}(u_n, U) + F_{\e_n} (v_n, V) \\\notag
 & + 3\int_{B_T}\hskip-0.25 cm \rho _2 (\xi)\! \int_{S_{n, \xi}^i}\hskip-0.25 cm g_{\e_n}\Bigl(\dfrac{u_n(x+\eps_n\xi)- u_n (x)}{\eps_n}\Bigr)  \notag
 +g_{\e_n}\Bigl(\dfrac{v_n(x+\eps_n \xi)- v_n (x)}{\eps_n}\Bigr)  \dx x\, \dx\xi \\ \notag
 &+ \frac{12 N^2}{R^2}\int_{B_T} \rho_2 (\xi)|\xi|^2 \int_{S_{n, \xi}^i} \left|u_n(x+\eps_n \xi) - v_n (x+\eps_n \xi) \right|^2 \dx x\, \dx\xi \\ \notag 
& + \int_{B_T}\hskip-0.25 cm \psi_{\e_n} (\xi) \int_{S_{n, \xi}^i}\hskip-0.25 cm \Bigl|\dfrac{u_n(x+\eps_n\xi)- u_n (x)}{\eps_n}\Bigr| + \Bigl|\dfrac{v_n(x+\eps_n \xi)- v_n (x)}{\eps_n}\Bigr| \,\dx x\, \dx\xi \\ \notag 
 & +\frac{2 N}{R}\int_{B_T} \psi_{\e_n}  (\xi)|\xi| \int_{S_{n, \xi}^i} \left|u_n(x+\eps_n \xi) - v_n (x+\eps_n \xi) \right| \,\dx x\, \dx\xi \\ \label{prima disuguaglianza subadd}
 & +\int_{B_T}\eta_{\e_n}(\xi)|S^i_{n,\xi}|\,\dx \xi.
 \end{align}
We estimate the summands in \eqref{prima disuguaglianza subadd}. 

As for the integrals in the third and fifth line of \eqref{prima disuguaglianza subadd}, we use \eqref{finite momenta} and \eqref{inclusione gusci} to get that, for $n$ large enough,
\begin{align}\notag
   &\frac{12 N^2}{R^2}\int_{B_T} \rho_2(\xi)|\xi|^2 \int_{S_{n, \xi}^i} \left|u_n(x+\eps_n \xi) - v_n (x+\eps_n \xi) \right|^2 \dx x\, \dx\xi \\ \notag
   & \leq \frac{12 N^2}{R^2}\Bigl(\int_{\R^d} \rho_2  (\xi)|\xi|^2\, \dx \xi\Bigr)\|u_n - v_n\|_{L^2(V';\Rm)}^2 \\ \notag
   & \leq \frac{24\sqrt{m} N^2}{R^2}\Bigl( \int_{\R^d}\rho_2(\xi)|\xi|^2\, \dx\xi\Bigr)  \|u\|_{L^\infty(\Omega;\Rm)}\|u_n-v_n\|_{L^1( V';\Rm)}\\\label{piccolezza quadrato subadditivita}
   &\leq C N^2\|u_n-v_n\|_{L^1(V';\Rm)},
\end{align}
and similarly, also using \eqref{psi momento},
\begin{align}  \notag
    & \frac{2 N}{R}\int_{B_T} \psi_{\e_n}(\xi)|\xi| \int_{S_{n, \xi}^i} |u_n(x+\eps_n \xi) - v_n(x+\eps_n \xi) | \dx x\, \dx\xi \\ \label{piccolezza subadditivita} 
    &\leq CN \|u_n-v_n\|_{L^1(V';\Rm)}.
\end{align}

In order to estimate the integrals in the second, fourth, and sixth line of \eqref{prima disuguaglianza subadd}, we note that, by \eqref{inclusione gusci},  for $ i\in\{1,...,N-4\}$ the sets $S^i_{n,\xi}$ intersect at most pairwise  when $n$ is large enough.
As a consequence, setting $E:=U_{N-2}\cap V'$ and using \eqref{inclusione strip}, we have
\begin{align}\notag
  &  \sum_{i=1}^{N-4}\Bigl[ \int_{B_T} \rho_2  (\xi) \int_{S_{n, \xi}^i}  g_{\e_n}\Bigl(\dfrac{u_n(x+\eps_n \xi)- u_n (x)}{\eps_n}\Bigr)  + g_{\e_n}\Bigl(\dfrac{v_n(x+\eps_n \xi)- v_n (x)}{\eps_n}\Bigr)  \dx x\, \dx\xi \\ \notag
& \quad + \int_{B_T}\hskip-0.25 cm \psi_{\e_n} (\xi) \int_{S_{n, \xi}^i}\hskip-0.25 cm \Bigl|\dfrac{u_n(x+\eps_n\xi)- u_n (x)}{\eps_n}\Bigr| + \Bigl|\dfrac{v_n(x+\eps_n \xi)- v_n (x)}{\eps_n}\Bigr| \,\dx x\, \dx\xi +\int_{B_T}\eta_{\e_n}(\xi)|S^i_{n,\xi}|\dx \xi\Bigr]\\ \notag
    &\le 2\int_{B_T} \rho_2  (\xi) \int_{E} g_{\e_n}\Bigl(\dfrac{u_n(x+\eps_n \xi)- u_n (x)}{\eps_n}\Bigr) + g_{\e_n}\Bigl(\dfrac{v_n(x+\eps_n \xi)- v_n (x)}{\eps_n}\Bigr) \,\dx x \,\dx 
    \xi \\  \label{conto su strip}
    & \quad + 2\int_{B_T} \psi_{\e_n} (\xi) \int_{E} \Bigl|\dfrac{u_n(x+\eps_n\xi)- u_n (x)}{\eps_n}\Bigr| + \Bigl|\dfrac{v_n(x+\eps_n \xi)- v_n (x)}{\eps_n}\Bigr| \,\dx x\, \dx\xi +4\Lambda|\Omega|,  
\end{align}
with $\Lambda$ as in \eqref{Lambda}.

Let now $r_0$ be as in \eqref{non zero in zero}. Assuming $\eps_n (r_0+T) < \text{dist} (V', \R^d \setminus \Omega)$ and $\xi\in B_T$, we have 
\begin{align*}
    \text{dist}(E+(0,\eps_n)\xi, \R^d \setminus \Omega) \ge \text {dist}(V'+B_{\eps_n T}, \R^d \setminus \Omega)
    \ge \text {dist}(V', \R^d \setminus \Omega) - \eps_n T > \eps_n r_0;
\end{align*}
hence, we are in position to apply Lemma \ref{lemma: short range interactions} which yields
\begin{align}\notag
    & \hspace{-0.5 cm}\int_{E} g_{\e_n}\Big(\dfrac{u_n(x+\eps_n \xi)- u_n (x)}{\eps_n}\Big)  + g_{\e_n}\Big(\dfrac{v_n(x+\eps_n\xi)- v_n (x)}{\eps_n}\Big)\,  \dx x \\ \notag 
    & \le C_2 (\left|\xi\right|^2+1)(G_{\e_n}^{r_0}(u_n, E_{{\e_n},\xi}) 
     + G_{\e_n}^{r_0}(v_n, E_{{\e_n},\xi})) \\ \label{altro conto strip} & \leq  C_2(\left|\xi\right|^2+1)(G_{\e_n}^{r_0}(u_n, U) + G_{\e_n}^{r_0}(v_n, V)),
\end{align}
where have used that
\begin{equation*}
    E_{{\e_n},\xi}:=E+(0,\eps_n)\xi +B_{\eps_n r_0} \subset U \cap V
\end{equation*}
upon further assuming that $n$ is large enough so that $\eps_n(r_0+T)< 2\frac{R}{N}$. 
 
By \eqref{non zero in zero} and \eqref{comparison functionals}, we have that $c_0G_{\e_n}^{r_0}\leq G_{1,\e_n}\leq F_{\e_n}$; therefore, recalling that
\begin{equation*}
    \sup_n F_{\e_n}(u_n, U)<+\infty, \qquad \sup_n F_{\e_n}(v_n, V)<+\infty,
\end{equation*}
we deduce from \eqref{altro conto strip} that
\begin{equation}\label{bound sulla strip}
    \int_{E} g_{\e_n}\Big(\dfrac{u_n(x+\eps_n \xi)- u_n (x)}{\eps_n}\Big)  + g_{\e_n}\Big(\dfrac{v_n(x+\eps_n\xi)- v_n (x)}{\eps_n}\Big)\,  \dx x  \leq  C (\left|\xi\right|^2+1),
\end{equation}
where the constant $C$ is independent of $n$ and $\xi$. 

With a similar argument based on Lemma \ref{lemma: short range interactions}, we get
\begin{align}
    \notag \int_{E} \Bigl|\dfrac{u_n(x+\eps_n\xi)- u_n (x)}{\eps_n}\Bigr| + \Bigl|\dfrac{v_n(x+\eps_n \xi)- v_n (x)}{\eps_n}\Bigr| \,\dx x & \leq  C_2(\left|\xi\right|+1)(P_{\e_n}^{r_0}(u_n, U) + P_{\e_n}^{r_0}(v_n, V)) \\ \notag
    & \leq  C_2(\left|\xi\right|+1)(H_{\e_n}^{r_0}(u_n, U) + H_{\e_n}^{r_0}(v_n, V)) \\ \label{altro conto strip bis}
    & \leq C(|\xi|+1),
\end{align}
where we have used Proposition \ref{prop: interpolation} in the last inequality.

Putting together \eqref{conto su strip}, \eqref{bound sulla strip}, and \eqref{altro conto strip bis} we deduce that, upon taking $\e_n$ small enough,
\begin{align*}
  \sum_{i=1}^{N-4}\Bigl[ &\int_{B_T} \rho_2  (\xi) \int_{S_{n, \xi}^i}  g_{\e_n}\Bigl(\dfrac{u_n(x+\eps_n \xi)- u_n (x)}{\eps_n}\Bigr)  + g_{\e_n}\Bigl(\dfrac{v_n(x+\eps_n \xi)- v_n (x)}{\eps_n}\Bigr)  \dx x\, \dx\xi \\ \notag
& \quad + \int_{B_T}\hskip-0.25 cm \psi_{\e_n} (\xi) \int_{S_{n, \xi}^i}\hskip-0.25 cm \Bigl|\dfrac{u_n(x+\eps_n\xi)- u_n (x)}{\eps_n}\Bigr| + \Bigl|\dfrac{v_n(x+\eps_n \xi)- v_n (x)}{\eps_n}\Bigr| \,\dx x\, \dx\xi \\ \notag
    & \quad +\int_{B_T}\eta_{\e_n}(\xi)|S^i_{n,\xi}|\dx \xi\Bigr]\\ \notag
& \leq C \int_{B_T} \rho_2(\xi) (|\xi|^2+1) + \psi_{\e_n}(\xi)(|\xi|+1) \,\dx\xi + 4\Lambda|\Omega| ,
\end{align*}
where the constant $C$ is independent of $n$. Recalling \eqref{finite momenta}, \eqref{psi momento}, and \eqref{psi integrale}, we deduce that there exists $i^* \in \{1,...,N-4\}$ such that 
\begin{align}
\notag &\int_{B_T} \rho_2  (\xi) \int_{S_{n, \xi}^{i^*}}  g_{\e_n}\Bigl(\dfrac{u_n(x+\eps_n \xi)- u_n (x)}{\eps_n}\Bigr)  + g_{\e_n}\Bigl(\dfrac{v_n(x+\eps_n \xi)- v_n (x)}{\eps_n}\Bigr)  \dx x\, \dx\xi \\ \notag
& \quad + \int_{B_T} \psi_{\e_n} (\xi) \int_{S_{n, \xi}^{i^*}} \Bigl|\dfrac{u_n(x+\eps_n\xi)- u_n (x)}{\eps_n}\Bigr| + \Bigl|\dfrac{v_n(x+\eps_n \xi)- v_n (x)}{\eps_n}\Bigr| \,\dx x\, \dx\xi \\ \label{ultima stima strip}
    & \quad +\int_{B_T}\eta_{\e_n}(\xi)|S^{i^*}_{n,\xi}|\,\dx \xi
\leq \frac{C}{N-4}
\end{align}
for all $n$ large enough.

Finally, substituting \eqref{piccolezza quadrato subadditivita}, \eqref{piccolezza subadditivita}, and \eqref{ultima stima strip} in \eqref{prima disuguaglianza subadd}, we get
\begin{equation*}
    \notag F_{\e_n}^T (w_n^{i^*}, U'\cup V') \leq F_{\e_n}(u_n, U) + F_{\e_n}(v_n, V) +\frac{C}{N-4}+ C(N^2+N)\|u_n - v_n\|_{L^1(V';\Rm)},
\end{equation*}
then, letting $n \to +\infty$ and $N\to +\infty$ in this order, recalling that $u_n,v_n, w_n^{i^*}\to u$ in $L^1_{\rm loc}(\Omega; \Rm)$, we achieve 
\begin{equation*}
    F''^{,T}(u,U'\cup V') \leq F''(u,U)+ F''(u,V)
\end{equation*}
and the thesis follows by Proposition \ref{proposition: truncated functionals} since $U'\cup V'\in \A_{\rm reg}(\Omega)$.

If $u\in L^1_{\rm loc}(\Omega;\Rm)$, by the previous case we have
\begin{equation*}
      F''(u^M, U' \cup V') \le F'' (u^M, U) + F'' (u^M, V)
\end{equation*}
for all $M>0$; hence, the conclusion follows by Lemma \ref{lemma:vertical truncations} letting $M\to+\infty$.
\end{proof}

We now investigate the inner regularity of the functionals $F'$ and $F''$.
\begin{Lemma}\label{lemma inner regulairty}
  Let $u\in L^1_{\rm loc}(\Omega;\Rm)$. Then $F'(u,\cdot)$ and $F''(u,\cdot)$ are increasing set functions and are inner regular on $A_{\rm reg}(\Omega)$; that is, $F'(u, \emptyset)=F''(u,\emptyset)=0$, given $U, V\in \A(\Omega)$ such that $U \subseteq V$, it holds
    \begin{equation*}
        F'(u,U)\leq F'(u, V), \qquad F''(u,U)\leq F''(u,V),
    \end{equation*}
    and, for every $U\in\A_{\rm reg}(\Omega)$, it holds
    \begin{gather}
    \label{F'innerregular}
        F'(u,U)=\sup\{F'(u,U') : U'\subset \subset U \text{ with }U'\in \mathcal{A}(\Omega)\},\\\label{F''innerregular}
        F''(u,U)=\sup\{F''(u,U') : U'\subset \subset U \text{ with }U'\in \mathcal{A}(\Omega)\}.
    \end{gather}
\end{Lemma}
\begin{proof}

The fact that both $F'(u,\cdot)$ and $F''(u,\cdot)$ are increasing set-functions follows observing that, for every $n\in\N$, we have $F_{\e_n}(u, \emptyset)=0$ and $F_{\e_n}(u, U)\leq F_{\e_n}(u,V)$ if $U\subseteq V$. 

Let $U, U'\in\A_{\rm reg}(\Omega)$ with $U' \subset\subset  U,$ and assume $u\notin {\rm GSBV}^2(U;\Rm)$. Using \eqref{referenceboundlower}, we have
\begin{align*}
    F'(u,U) \geq F'(u, U') \geq c_1{\rm MS}_1(u,U'),
\end{align*}
and by the arbitrariness of $U'$, we infer
\begin{align*}
     F'(u,U) & \geq \sup\{F'(u,U') : U'\subset \subset U \text{ with }U'\in \mathcal{A}(\Omega)\} \\
     & \geq \sup\{F'(u,U') : U'\subset \subset U \text{ with }U'\in \mathcal{A}_{\rm reg}(\Omega)\}  \\
     & \geq  c_1{\rm MS}_1(u,U)=+\infty
\end{align*}
from which \eqref{F'innerregular} readily follows. The same proof yields also \eqref{F''innerregular}.

Now we prove \eqref{F''innerregular} in the case $u\in {\rm GSBV}^2(U;\Rm) $, the proof of \eqref{F'innerregular} being analogous. We first suppose that $u\in L^\infty(\Omega; \Rm)$ and note that, since $F''(u, \cdot)$ is an increasing set-function, it is enough to prove that \begin{equation*}F''(u, U)\leq \sup\{F''(u,U') : U'\subset \subset U \text{ with }U'\in \mathcal{A}(\Omega)\}.\end{equation*}
    Let $\delta>0$. As $u\in {\rm SBV}^2(U; \Rm)$, we find an open set $U_\delta\subset \subset U$ such that $U\setminus \overline{U}_\delta\in{\A}_{\rm reg}(\Omega)$ and 
    \begin{equation}\label{AmenoAdelta}
      c_2H(u, U\setminus\overline{U}_{\delta})+\Lambda| U\setminus\overline{U}_{\delta}|<\delta,
    \end{equation}
  where $c_2$ and $\Lambda$ are the constants of Proposition \ref{prop: referencebound}.
    
    Let us consider an open set $U'$ such that $U_\delta\subset \subset U'\subset \subset U$ and let $R:=\text{dist}(U_\delta, \Omega\setminus U')$. We consider sequences $\{u_n\}_n,\{v_n\}_n\subset L^1_{\rm loc}(\Omega;\Rm)$ converging to $u$ in $L^1_{\rm loc}(\Omega;\Rm)$ and such that \begin{equation}\label{inner recovery}
      \lim_{n\to +\infty}F_{\e_n}(u_n,U')=F''(u,U'), \qquad \lim_{n\to+\infty }F_{\e_n}(v_n,U\setminus \overline{U}_{\delta})=F''(u,U\setminus \overline{U}_{\delta}), 
    \end{equation} 
and which, thanks to Remark \ref{L infty recovery}, we may also assume to be satisfying   
\begin{equation*}
        \sup_n \|u_n\|_{L^\infty(\Omega;\Rm)}\leq \sqrt{m}\|u\|_{L^\infty(\Omega;\Rm)}, \qquad  \sup_n \|v_n\|_{L^\infty(\Omega;\Rm)}\leq {\sqrt{m}}\|u\|_{L^\infty(\Omega;\Rm)}.
\end{equation*}
Letting $N$ be a positive integer, for $i\in\{1,...,N\}$, we set $U_i:=\{x\in U : \text{dist}(x, U_\delta)< \frac{iR}{N}\}$ and we introduce a cut-off function $\varphi_i \in C^\infty_c (\R^d; [0,1])$ such that
\begin{equation*}
\begin{cases}
     \phi_i= 1 &  \text{in }  \overline{U}_i,\\
      \phi_i= 0 & \text{in } \R^d \setminus U_{i+1},\\
       |\nabla \phi_i| \le \frac{2N}{R} & \text{in } \Rd.
\end{cases}
\end{equation*}
We set $w^i_n:=\phi_iu_n + (1-\phi_i)v_n$ and 
    \begin{equation*}
        S^i_{n,\xi}:=U_{\e_n}(\xi)\setminus [(U_i)_{\e_n}(\xi)\cup(U\setminus \overline{U}_{i+1})_{\e_n}(\xi)].
    \end{equation*} 
    Let us fix $T>0$. We have
    \begin{align*} \notag
    F_{\e_n}^T(w_n^{i}, U) & = F_{\e_n}^T(u_n, U_{i}) +  F_{\e_n}^T(v_n, U\setminus\overline{U}_{{i}+1}) 
      \\ \notag
      & \quad + \int_{B_T}\int_{S^{i}_{n,\xi}} f_{\e_n}\Bigl(x,\xi, \dfrac{w^i_n(x+\eps_n \xi)- w_n^i(x)}{\eps_n}\Bigr)\, \dx x\, \dx \xi \\ \notag
      & \leq  F_{\e_n}(u_n, U') +  F_{\e_n}(v_n, U\setminus\overline{U}_{\delta}) 
      \\ 
      & \quad + \int_{B_T}\int_{S^{i}_{n,\xi}} f_{\e_n}\Bigl(x,\xi, \dfrac{w^i_n(x+\eps_n \xi)- w^i_n (x)}{\eps_n}\Bigr)\, \dx x\, \dx \xi.
\end{align*}
Arguing exactly as in the part of the proof of Lemma \ref{fundamental estimate truncated} that led to \eqref{piccolezza quadrato subadditivita}, \eqref{piccolezza subadditivita}, and \eqref{ultima stima strip}, we infer the existence of $i^*\in\{1,..., N-4\}$ such that 
\begin{align*} \notag
     \int_{B_T}\int_{S^{i^*}_{n,\xi}} f_{\e_n}\Bigl(x,\xi, \dfrac{w^{i^*}_n(x+\eps_n \xi)- w^{i^*}_n(x)}{\eps_n}\Bigr)\, \dx x\, \dx \xi 
 \leq \frac{C}{N-4}+ C(N^2+N)\|u_n - v_n\|_{L^1(U'\setminus\overline{U}_\delta;\Rm)};
\end{align*}
then, letting $n\to+\infty$ and using \eqref{inner recovery}, we get
\begin{equation*}
    F''^{,T}(u,U)\leq F''(u,U')+F''(u, U\setminus\overline{U}_{\delta}) + \frac{C}{N-4}.
\end{equation*}
Having supposed $U\setminus \overline{U}_\delta\in \A_{\rm reg}(\Omega)$, we use \eqref{referenceboundhigher} to obtain
\begin{equation*}
    F''^{,T}(u,U)\leq F''(u,U')+ c_2H(u, U\setminus\overline{U}_{\delta})+\Lambda| U\setminus\overline{U}_{\delta}| + \frac{C}{N-4};
\end{equation*}
 hence, by \eqref{AmenoAdelta}, we get
\begin{equation}\notag
     F''^{,T}(u,U)\leq F''(u,U')+\delta + \frac{C}{N-4}.
\end{equation} 
Since $U\in\A_{\rm reg}(\Omega)$, we apply Proposition \ref{proposition:  truncated functionals} so that the proof follows letting $T\to+\infty$ and by the arbitrariness of $\delta$ and $N$.

If $u\in {\rm GSBV}^2(U;\Rm)\cap L_{\rm loc}^1(\Omega;\Rm)$, thanks to Lemma \ref{lemma:vertical truncations} and to \eqref{F''innerregular} applied with $u= u^M$, we obtain that 
\begin{align*}
    F''(u,U)&=\sup_M F''( u^M,U)\\&
    =\sup_{M}\,\sup\{F''(u^M,U') : U'\subset \subset U \text{ with }U'\in \mathcal{A}(\Omega)\}\\
     &=\sup\,\{\sup_{M}F''(u^M,U'):U'\subset \subset U \text{ with }U'\in \mathcal{A}(\Omega)\}\\
     &=\sup\{F''(u,U'):U'\subset \subset U \text{ with }U'\in \mathcal{A}(\Omega)\},
\end{align*}
which proves \eqref{F''innerregular} in the general case. 
\end{proof}

\begin{Remark} \label{rmk : inner}
    Resorting to Lemma \ref{lemma inner regulairty}, one checks  immediately that in Lemma \ref{fundamental estimate truncated} holds unaltered if we drop the hypothesis that $U'\cup V'\in\A_{\rm reg}(\Omega)$.
\end{Remark}

We conclude the section proving the compactness of the sequence $\{F_{\e_n}\}_n$ with respect to  $\Gamma$-convergence, and also provide some relevant properties of the limit functional, which are going to be useful to prove an integral representation.

\begin{Proposition}\label{prop:compactness}
    Let $\{F_\e\}_\e$ be as in \eqref{def functionals}, and let $\{\e_n\}_n$ be a positive sequence converging to $0$ as $n\to+\infty$. There exists a subsequence $\{\e_{n_k}\}_k$  and a functional $F\colon L^1_{\rm loc}(\Omega;\Rm)\times \A(\Omega)\to [0,+\infty]$ satisfying the following properties:
    \begin{itemize}
 \item[$(a)$]{\rm ($\Gamma$-convergence)} for every $(u,U)\in L^1_{\rm loc}(\Omega;\R^m)\times\mathcal{A}_{\rm reg}(\Omega)$ it holds
 \begin{equation*}
     F(u,U)= \underset{k\to+\infty}{\Gamma\text{-}\lim}\,F_{\e_{n_k}}(u,U);
 \end{equation*}
        
        \item[$(b)$]{\rm (lower semicontinuity)} for every $U\in\mathcal{A}(\Omega)$ the functional $F(\cdot,U)$ is $L^1_{\rm loc}(\Omega;\R^m)$-lower semicontinuous;
        
       \item[$(c)$] {\rm (locality)} for every $U\in\mathcal{A}(\Omega)$ and  $u,v\in L^1_{\rm loc}(\Omega;\R^m)$ with $u=v$
 $\Ld$-a.e.\ on $U$ we have
 $F(u,U)=F(v,U)$;
        \item[$(d)$]{\rm (growth conditions)} there exist positive constants $c_1,c_2, \Lambda$ such that for every $U\in\A(\Omega)$ we have
        \begin{gather}\label{lower bound proposition}
        c_1{\rm MS}_1(u,U)\leq F(u,U) \text{ for all $u\in L^1_{\rm loc}(\Omega;\R^m)$},\\   \label{upper bound proposition} 
          F(u,U)\leq c_2 H(u,U)+\Lambda|U| \text{ for all }u\in L^1_{\rm loc}(\Omega;\Rm)\cap{\rm BV}(U;\Rm);
    \end{gather}
       
\item[$(e)$]{\rm (measure property)} for every $u\in L^1_{\rm loc}(\Omega;\Rm)$ the set function $F(u,\cdot)$ is the restriction to $\A(\Omega)$ of a Borel measure on $\Omega$;

\item[$(f)$]{\rm (monotonicity)} for every  $1$-Lipschitz function $\Phi\colon \Rm\to\Rm$ and $(u,U)\in  L^1_{\rm loc}(\Omega;\Rm) \times \A(\Omega)$, we have 
\begin{equation}\notag 
    F(\Phi\circ u,U)\leq F(u,U).
\end{equation}
In particular, 
\begin{align*}
 F(u+a,U)& = F(u,U)\quad \text{ for every } a\in\Rm,\\
    F(Ru,U) & =F(u,U)\quad \text{ for every }R\in SO(m),\\
    F(u^M,U) & \leq F(u,U)\quad \text{ for every } M>0.
\end{align*}
    \end{itemize}
\end{Proposition}

\begin{proof}
Consider the functionals $ F', F'':  L^1_{\rm loc}(\Omega;\Rm)\times \A(\Omega)\to [0,+\infty]$ given by
\begin{equation}\label{def gamma compactness}
   F'(u,U):= \Gamma\text{-}\liminf_{n\to+\infty} F_{\e_n}(u,U), \qquad  F''(u,U):= \Gamma\text{-}\limsup_{n\to+\infty} F_{\e_n}(u,U) 
\end{equation}
and, for fixed $u\in L^1_{\rm loc}(\Omega;\Rm)$, define the inner regular envelopes 
\begin{gather*}
    F'_{-}(u,U):= \sup\{F'(u,U') : U'\subset \subset U \text{ with }U'\in \mathcal{A}(\Omega)\}, \\
     F''_{-}(u,U):= \sup\{F''(u,U') : U'\subset \subset U \text{ with }U'\in \mathcal{A}(\Omega)\}
\end{gather*}
for every $U\in \A(\Omega)$. By \cite[Theorem~16.9]{DalBook}, there exists a subsequence $\e_{n_k}\to 0^+$ as $k\to +\infty$ such that
\begin{equation}\label{def gammalimite}
    F'_{-}(u,U) = F''_{-}(u,U)=: F(u,U)
\end{equation}
for every $(u,U)\in L^1_{\rm loc}(\Omega;\Rm)\times \A(\Omega)$, where $F'_{-}, F''_{-}$ in \eqref{def gammalimite} are obtained replacing $\{\e_{n}\}_n$ in \eqref{def gamma compactness} with the subsequence $\{\e_{n_k}\}_k$.

Using Lemma \ref{lemma inner regulairty}, we obtain that 
\begin{equation*}
    F'(u,U)= F'_{-}(u,U)= F''_{-}(u,U)= F''(u,U) 
\end{equation*}
for every $(u,U)\in L^1_{\rm loc}(\Omega;\Rm)\times \A_{\rm reg}(\Omega)$, which proves property $(a)$. 

Property $(b)$ follows by the lower semicontinuity of the $\Gamma$-liminf and \eqref{def gammalimite}. Locality property $(c)$ follows by the locality of $\{F_{\e_n}\}_n$ and \cite[Proposition~16.15]{DalBook}. 

In order to prove the measure property $(e)$ we use De Giorgi-Letta criterion for measures (see, e.g., \cite[Theorem~10.2]{bdf} or \cite[Theorem~14.21]{DalBook}). The proof then amounts to show that, for every $u\in L^1_{\rm loc}(\Omega; \Rm)$, the set function $ F(u,\cdot)$
is increasing, inner regular on $\A(\Omega)$, superadditive, and subadditive. 
 
Let us fix $u\in L^1_{\rm loc}(\Omega;\Rm)$.
By Lemma \ref{lemma inner regulairty}, \eqref{def gammalimite}, and \cite[Remark~14.5]{DalBook} the set function $F(u,\cdot)$ is increasing, while it is inner regular  on $\A(\Omega)$ by definition. 

We now prove that ${F}(u, \cdot)$ is subadditive. Let $U,V\in\mathcal{A}(\Omega)$. By inner regularity of $F$,  it suffices to show that $F(u,W)\leq {F}(u,U)+{F}(u,V)$ for any $W\in \A_{\rm reg}(\Omega)$ such that $W\subset\subset U\cup V$. We choose $U',U'',V',V''\in \A_{\rm reg}(\Omega)$  with $U'\subset \subset U''\subset \subset  U$ and $V'\subset \subset V''\subset \subset  V$ and such that $W\subset U'\cup V'$. Using Lemma \ref{fundamental estimate truncated} and Remark \ref{rmk : inner}, we obtain that  
\begin{equation}\notag
  F(u,W)\leq F(u,U'\cup V')\leq   F''(u,U'\cup V')\leq F''(u,U'')+F''(u,V'')\leq {F}(u,U)+{F}(u,V).
\end{equation}

We are left with proving that ${F}(u, \cdot)$ is superadditive. Let $U,V\in\mathcal{A}(\Omega)$ with $U\cap V=\emptyset$. One can easily check that $U_{\e_{n_k}}(\xi)\cup V_{\e_{n_k}}(\xi)\subseteq (U\cup V)_{\e_{n_k}}(\xi)$ for every $k\in\N$ and $\xi\in \Rd$. Since the sets $U_{\e_{n_k}}(\xi)$ and $V_{\e_{n_k}}(\xi)$ are disjoint, for every $\{u_{k}\}_k$ converging to $u$ in $L^1_{\rm loc}(\Omega;\Rm)$ we obtain 
\begin{equation*}
     F_{\e_{n_k}}(u_{k},U)+F_{\e_{n_k}}(u_{k},V)\leq F_{\e_{n_k}}(u_{k},U\cup V),
\end{equation*}
which, letting $k\to+\infty$, implies the superadditivity, thus, property ($e$) is verified.

For $U\in\A_{\rm reg}(\Omega)$, property $(d)$ holds  by Proposition \ref{prop: referencebound} and property $(f)$ is a consequence of Lemma \ref{lemma:vertical truncations}. Both are extended to every element of $\A(\Omega)$ resorting to $(e)$.
\end{proof}

\section{Integral Representation}\label{sec: integral representation}
In this section we complete the proof of Theorem \ref{thm:compactness} by proving its integral representation part. We recall the notation set in \eqref{def minimisation aux}.  Given a functional $\mathcal{F}\colon {\rm SBV^2}(\Omega;\Rm)\times \mathcal{A}(\Omega)\to [0,+\infty]$, a function $w\in {\rm SBV}^2(\Omega;\Rm)$, and $U\in\mathcal{A}_{\rm reg}(\Omega)$, we set 
\begin{equation*}
    \m^\mathcal{F}(w,U):=\inf\{\mathcal{F}(u,U):\,u\in {\rm SBV}^2(U;\Rm) \text{ with }u=w \text{ in a neighbourhood of }\partial U\}.
\end{equation*}
Given $L\in \Rmd$, with a slight abuse of notation, we also let $L$ denote the corresponding linear map $L:\Rd\to \R^m$. Given $x\in\Rd$, $\zeta\in\Rm$, and $\nu\in\Sd$, we introduce the function  $u_{x,\zeta,\nu}\colon\Rd\to\Rm$ defined for every $y\in\Rd$ by 
\begin{equation}\notag
    u_{x,\zeta,\nu}(y):=\begin{cases}
     \zeta &\text{ if } (y-x) \cdot\nu\geq 0,\\
     0 &\text{ otherwise}.
    \end{cases}
\end{equation}
We let the symbol $Q^\nu(x,r)$ denote an open cube of side length $r$ and centre $x$ with two faces orthogonal to $\nu$, and we assume that $Q^\nu(x,r)=Q^{-\nu}(x,r)$. We also set $Q(x,r):=x+(-\frac{r}{2},\frac{r}{2})^d$.

The main result of this section is as follows. 

\begin{Proposition}\label{prop: representation} Let  $F\colon L^1_{\rm loc}(\Omega;\Rm)\times \A(\Omega)\to [0,+\infty]$ be a functional satisfying properties $($b$\,)$-$($f$\,)$ of Proposition \ref{prop:compactness}. 
Then for every $U\in\A(\Omega)$ and $u\in {\rm GSBV}^2(U;\Rm)\cap L^1_{\rm loc}(\Omega;\Rm)$ we have   \begin{equation}\notag
     F(u,U)=  \int_U f_{{\rm bulk}}(x,\nabla u)\,{\rm d}x+\int_{J_u\cap U} f_{{\rm surf}}(x, [u],\nu_u)\, {\rm d}\Hd,
    \end{equation}
    where 
    \begin{equation}
          \label{Gammabulk'}  f_{{\rm  bulk}}(x,L):= \limsup_{r\to 0^+}\frac{\m^F( L,Q(x,r))}{r^d}
          \end{equation}
          for all $x\in\Omega,$  $L\in\R^{m\times d}$, and
    \begin{equation}\label{Gammasurf'}
            f_{{\rm surf}}(x,\zeta,\nu):=\limsup_{r\to 0^+}\frac{\m^F(u_{x,\zeta,\nu},Q^\nu(x,r))}{r^{d-1}}
        \end{equation}
        for all $x\in\Omega$, $\zeta\in\R^m$, and $\nu\in\Sd$.
        
        \noindent Moreover, we have 
        \begin{align*}
          c_1\lambda_1|L|^2 \leq  f_{\rm bulk}(x,L)\leq c_2(\lambda_2|L|^2+\kappa|L|)+\Lambda\quad &  \text{for all $x\in\Omega$ and $L\in\Rmd$},\\
          c_1\mu_1 \leq  f_{\rm surf}(x,\zeta,\nu)\leq c_2(\mu_2+\kappa|\zeta|)\quad & \text{for all $x\in\Omega$, $\zeta\in\Rm$, and $\nu\in\Sd$},
        \end{align*}
        and \begin{align*}
  f_{\rm bulk}(x,L_1)\leq  f_{\rm bulk}(x,L_2) \quad & \text{  if  $|L_1|\leq |L_2| $},\\
 f_{\rm surf}(x,\zeta_1,\nu)\leq  f_{\rm surf}(x,\zeta_2,\nu)\quad & \text{ if $|\zeta_1|\leq |\zeta_2|$.}
\end{align*}
for all $x\in\Omega$ and $\nu\in\Sd$. In particular, for all $x\in\Omega$ and $\nu\in\Sd$ it holds 
\begin{align*}f_{\rm bulk}(x,L_1)=  f_{\rm bulk}(x,L_2) \quad & \text{ if $|L_1|=|L_2|$},\\
f_{\rm surf}(x,\zeta_1,\nu)  =f_{\rm surf}(x,\zeta_2,\nu)\quad & \text{ if $|\zeta_1|=|\zeta_2|$},\\
    f_{\rm surf}(x,\zeta,\nu)=f_{\rm surf}(x,\zeta,-\nu) \quad & \text{ for all $\zeta\in\Rm$}.
\end{align*}
\end{Proposition}

\begin{Remark}\label{remark:new fbulk}
Thanks to $(f)$ of Proposition \ref{prop:compactness}, the functional $F$ considered in Proposition \ref{prop: representation} is invariant under vertical translations. In particular, for every $x\in\Omega$ and  $L\in\Rmd$ we have 
\begin{equation}\label{eq:new fbulk}
    f_{\rm bulk}(x,L):=\limsup_{r\to 0^+}\frac{\m^F( L(\cdot-x),Q(x,r))}{r^d}.
\end{equation}
\end{Remark}

The proof of this proposition is based on a result of Bouchitté, Fonseca, Leoni, and Mascarenhas \cite{BFLM2002}, who gave sufficient conditions for a free discontinuity functional to be representable as the sum of  bulk and surface energies. For the reader's convenience, we here give a complete, slightly modified, statement of this theorem.
\begin{Theorem}\label{Theorem:BouchFonsLeonMas}
    Let $\mathcal{F}\colon {\rm SBV}^2(\Omega;\Rm)\times \mathcal{A}(\Omega)\to [0,+\infty)$ be a functional that satisfies the following conditions:
    \begin{itemize}
  \item [(i)] for every $U\in\A(\Omega)$ the functional $\mathcal{F}(\cdot,U)$ is $L^1_{\rm loc}(\Omega;\Rm)$-lower semicontinuous;
        \item [(ii)] for every $U\in\A(\Omega)$ and $u,v\in {\rm SBV}^2(\Omega;\Rm)$ with  $u=v$ $\Ld$-a.e. on $U\in \A(\Omega)$ we have $\mathcal{F}(u,U)=\mathcal{F}(v,U)$;

        \item [(iii)] there exists a positive constant $C$ such that for every $U\in\A(\Omega)$ and $u\in {\rm SBV}^2(\Omega;\Rm)$ we have 
        \begin{multline*}
           \frac{1}{C}\Bigl(\int_U |\nabla u|^2\, \dx x+\int_{J_u\cap U}(1+|[u]|)\,{\rm d}\Hd\Bigr)\leq \mathcal{F}(u,U) \\
           \leq C\Bigl(\int_U (1+|\nabla u|^2)\, \dx x+\int_{J_u\cap U}\!\!\!\!(1+|[u]|)\,{\rm d}\Hd\Bigr);
        \end{multline*}
        \item [(iv)] for every $u\in {\rm SBV}^2(\Omega;\Rm)$ the set function $\mathcal{F}(u,\cdot)$ is the restriction to $\mathcal{A}(\Omega)$ of a Borel measure on $\Omega$;
         \item [(v)] for every $u\in {\rm SBV}^2(\Omega;\Rm)$, $U\in\A(\Omega)$, and $a\in\Rm$ we have $\mathcal{F}(u+a,U)=\mathcal{F}(u,U)$.
    \end{itemize}
       Then for every $u\in {\rm SBV}^2(\Omega;\Rm)$ and $U\in\mathcal{A}(\Omega)$ it holds
        \begin{equation*}
            \mathcal{F}(u,U)=\int_U f_{{\rm bulk}}(x,\nabla u)\,{\rm d}x+\int_{J_u\cap U} f_{{\rm surf}}(x,[u],\nu_u)\, {\rm d}\Hd,
        \end{equation*}
        where ${f}_{{\rm bulk}}$ and ${f}_{{\rm surf}}$ are given by \eqref{Gammabulk'} and \eqref{Gammasurf'} with $F$ replaced by $\mathcal{F}$.
\end{Theorem}

With this result at our disposal, we are now ready to  prove Proposition \ref{prop: representation}.

\medskip

\begin{proof}[Proof of Proposition \ref{prop: representation}] It follows immediately from $(b), (c), (e)$, and $(f)$ that the restriction  of the functional ${F}$ to ${\rm SBV}^2(\Omega;\Rm)\times \A(\Omega)$  satisfies conditions $(i), (ii), (iv)$, and $(v)$ of Theorem \ref{Theorem:BouchFonsLeonMas}. 
 By \eqref{upper bound proposition}, ${F}$ satisfies the upper bound in $(iii)$. However, by hypothesis it only satisfies the weaker lower bound in \eqref{lower bound proposition}, and not the desired lower bound in condition ({\it iii}). To overcome this difficulty, we use a perturbative approach similar to the one employed in \cite[Section~5]{DalToa}   (see also \cite[Section~4]{CagnettiPoincare}). For every $\delta>0$ we introduce the functional ${F}^\delta$ defined by   
\begin{equation*}
    {F}^\delta(u,U):={F}(u,U)+\delta\int_{J_u\cap U}|[u]|\,\dx\Hd,
\end{equation*}
for every $u\in {\rm SBV}^2(\Omega;\Rm)$ and $U\in\A(\Omega)$, which satisfies conditions ({\it i})-({\it v}) of Theorem \ref{Theorem:BouchFonsLeonMas}.

\noindent Hence, we  deduce the integral representation
\begin{equation}\label{deltarepresentation}
    {F}^\delta(u,U) = \int_U f^\delta_{{\rm bulk}}(x,\nabla u)\,{\rm d}x+\int_{J_u\cap U} f^\delta_{{\rm surf}}(x,[u],\nu_u)\, {\rm d}\Hd,
\end{equation}
for all $(u,U)\in {\rm SBV}^2(\Omega;\Rm)\times \A(\Omega)$, where $f^\delta_{\rm bulk}$ and $f^\delta_{\rm surf}$ are defined as in \eqref{Gammabulk'} and \eqref{Gammasurf'} with ${F}$ replaced by ${F}^\delta$. 

At this point, we aim to pass to the limit as $\delta\to0^+$ in \eqref{deltarepresentation} in order to obtain an integral representation for ${F}$ on ${\rm SBV}^2(\Omega;\Rm)\times \A(\Omega)$. Clearly, ${F}^\delta(u,U)\to {F}(u,U)$ as $\delta\to 0^+$; therefore, the conclusion follows provided we are in position to apply the Dominated Convergence Theorem on the right-hand side of \eqref{deltarepresentation}.

We preliminarily observe that, using $L$ as a test function for the minimisation problems $\m^F(L,Q(x,r))$ and $\m^{{F}^\delta}( L,Q(x,r))$, from the  upper bound in \eqref{upper bound proposition}, we have   
    \begin{align}\notag
           f_{\rm bulk}(x,L)\leq c_2(\lambda_2|L|^2+\kappa|L|)+\Lambda \quad & \text{for all  } x\in \Omega \text{ and }L\in\R^{m\times d}, \\ \label{f delta}
           f^\delta_{\rm bulk}(x,L)\leq  c_2(\lambda_2|L|^2+\kappa|L|)+\Lambda \quad & \text{for all } x\in\Omega \text{ and }L\in\R^{m\times d}
    \end{align}
and, applying Theorem \ref{Theorem:BouchFonsLeonMas} with $\mathcal{F}={\rm  MS}_1$, we also obtain the lower bound 
\begin{equation*}
     c_1\lambda_1|L|^2\leq f_{\rm bulk}(x,L)\quad \text{for all  } x\in \Omega \text{ and }L\in\R^{m\times d}.
\end{equation*}
 Analogously, we have
    \begin{align}\notag
         c_1\mu_1\leq f_{\rm surf}(x,\zeta,\nu)\leq c_2(\mu_2+\kappa|\zeta|)\quad & \text{for all }  x\in\Omega,\, \zeta\in\Rm, \text{ and }\nu\in\Sd,\\   \label{g delta}
          f^\delta_{\rm surf}(x,\zeta,\nu)\leq c_2\mu_2+(c_2\kappa+\delta)|\zeta|\quad & \text{for all } x\in \Omega,\,\zeta\in\R^m, \text{ and }\nu\in\Sd.
    \end{align}

Now we prove that 
\begin{equation}\label{claimbulk}
\lim_{\delta \to 0^+}f^\delta_{\rm bulk}(x,L)= f_{\rm bulk}(x,L)\quad  \text{ for all } x\in \Omega \text{ and } L\in \Rmd,
\end{equation}
\begin{equation}\label{claimsurf}
\lim_{\delta\to 0^+}f^\delta_{\rm surf}(x,\zeta,\nu)= f_{\rm surf}(x,\zeta,\nu)\quad  \text{ for all } x\in \Omega,\, \zeta\in \Rm,\text{ and } \nu\in \Sd.
\end{equation}
To this aim, we first observe that, since $F(u,Q(x,r))\leq  F^\delta(u,Q(x,r))$ for every $\delta>0$ and $\{{F}^\delta\}_\delta$ is increasing in $\delta$, it follows immediately that
     \begin{gather*}
         f_{\rm bulk}(x,L)\leq \inf_{\delta>0}f^\delta_{\rm bulk}(x,L)=\lim_{\delta\to 0^+}f^\delta_{\rm bulk}(x,L) \quad \text{ for all $x\in\Rd$ and $L\in\R^{m\times d}$}, \\
         f_{\rm surf}(x,\zeta,\nu)\leq  \inf_{\delta>0}f^\delta_{\rm surf}(x,\zeta,\nu)= \lim_{\delta\to0^+}f^\delta_{\rm surf}(x,\zeta,\nu) \quad \text{ for all $x\in\Rd$, $\zeta\in\Rm$, and $\nu\in\Sd$}.
     \end{gather*}
To obtain the converse inequalities, we will prove that for every $\delta>0$ we have
     \begin{gather}
        \label{claim limit bulk} f^\delta_{\rm bulk}(x,L)\leq f_{\rm bulk}(x,L) \quad \text{ for all $x\in\Rd$ and   $L\in\Rmd$},
        \end{gather}
and that there exists a constant $C>0$, independent of $\delta$ and of $|\zeta|$, such that 
\begin{gather}\label{claim limit surface}f^\delta_{\rm surf}
         (x,\zeta,\nu)\leq  (1+C|\zeta|\delta)f_{\rm surf}(x,\zeta,\nu) \quad \text{ for all $x\in\Rd$, $\zeta\in\Rm$, and $\nu\in\Sd$}.
     \end{gather}
       
Let us fix $x\in\Rd$ and  $L\in\Rmd$. In view of Remark \ref{remark:new fbulk}, for  $r>0$ we consider a function $u_r\in {\rm SBV}^2(Q(x,r); \Rm)$, with $u_r= L(\cdot-x)$ in a neighbourhood of $\partial Q(x,r)$ such that 
     \begin{equation}\label{minimality of ueps}
         F(u_r,Q(x,r))\leq \m^{{F}}(L(\cdot-x),Q(x,r))+r^{d+1}.
     \end{equation}
   By a truncation argument we assume  $\|u_r\|_{L^\infty(Q(x,r); \Rm)}\leq  \sqrt{md}|L|r$. We compare ${F}^\delta(u_r,Q(x,r))$ with ${F}(u_r,Q(x,r))$. By definition of ${F}^\delta$, the lower bound \eqref{lower bound proposition}, and \eqref{minimality of ueps} we get
     \begin{align}\notag
         {F}^\delta(u_r,Q(x,r))&={F}(u_r,Q(x,r))+\delta\int_{J_{u_r}\cap Q(x,r)}|[u_r]|\,{\rm d}\Hd\\\notag &\leq {F}(u_r,Q(x,r))+2\delta{\sqrt{md}}|L|r\Hd(J_{u_r}\cap Q(x,r))\\\notag 
         &\leq \big(1+\frac{2\delta{\sqrt{md}}|L|r}{c_1\mu_1}\big){F}(u_r,Q(x,r))\\\notag 
         &\leq \big(1+\frac{2\delta{\sqrt{md}}|L|r)}{c_1\mu_1}\big)\big(\m^{{F}}( L(\cdot-x),Q(x,r))+r^{d+1}\big).
     \end{align} 
As $u_r$ is a competitor for the minimisation problem $\m^{F^\delta}(L(\cdot-x),Q(x,r))$, dividing by $r^d$ the previous inequality we get that 
     \begin{equation*}
         \frac{\m^{F^\delta}(L(\cdot-x),Q(x,r))}{r^d}\leq  \bigl(1+\frac{2\delta\sqrt{md}|L|r}{c_1\mu_1}\bigr)\bigl(\frac{\m^{{F}}(L(\cdot-x),Q(x,r))}{r^d}+r\bigr),
     \end{equation*}
    and then, letting $r\to0^+$ and using \eqref{eq:new fbulk}, we obtain \eqref{claim limit bulk}, which in turn implies \eqref{claimbulk}. 

  We now show \eqref{claim limit surface}. Let us fix $x\in\Rd$, $\zeta\in\R^m$,  and $\nu\in\Sd$. Arguing as in the proof of \eqref{claim limit bulk}, for every $r>0$ we find a function $u_r\in {\rm SBV}^2(Q^\nu(x,r); \Rm)$, with $u_r=u_{x,\zeta,\nu}$ in a neighbourhood of $\partial Q^\nu(x,r)$ and   $\|u_r\|_{L^\infty(Q^\nu(x,r);\Rm)}\leq {\sqrt{m}}|\zeta|$,  such that 
  \begin{equation*}
      {F}(u_r,Q^\nu(x,r))\leq \m^{{F}}(u_{x,\zeta,\nu},Q^\nu(x,r))+r^d.
  \end{equation*}
  Using this inequality and the lower bound \eqref{lower bound proposition}, we obtain that
  \begin{align*}
  {F}^\delta(u_r,Q^\nu(x,r))&={F}(u_r,Q^\nu(x,r))+\delta\int_{J_{u_r}\cap Q^\nu(x,r)}\!\!\!\!|[u_r]|\,{\rm d}\Hd\\
  \notag &\leq {F}(u_r,Q^\nu(x,r))+2{\sqrt{m}}\delta|\zeta|\Hd(J_{u_r}\cap Q^\nu(x,r))\\
  &\leq \big(1+\frac{2\sqrt{m}\delta|\zeta|}{c_1\mu_1}\big){F}(u_r,Q^\nu(x,r))\leq \big(1+\frac{2{\sqrt{m}}\delta|\zeta|}{c_1\mu_1}\big)\big (\m^{{F}}(u_{x,\zeta,\nu},Q^\nu(x,r))+r^d\big),
  \end{align*}
  then, dividing this last inequality by $r^{d-1}$ and using the fact that $u_r=u_{x,\zeta,\nu}$ in a neighbourhood of $\partial Q(x,r)$, we get 
  \begin{equation*}
      \frac{\m^{F^\delta}(u_{x,\zeta,\nu},Q^\nu(x,r))}{r^{d-1}}\leq \bigl(1+\frac{2\sqrt{m}\delta|\zeta|}{c_1\mu_1}\bigr)\bigl(\frac{\m^{{F}}(u_{x,\zeta,\nu},Q^\nu(x,r))}{r^{d-1}}+r\bigr).
  \end{equation*}
  Letting $r\to0^+$, we get \eqref{claim limit surface}, which leads to \eqref{claimsurf}.

  Having proved the pointwise convergence, we finally resort to the uniform bounds given by \eqref{f delta}  and \eqref{g delta} in order to apply the Dominated Convergence Theorem in \eqref{deltarepresentation} and to obtain 
    \begin{equation}\label{SBV repre}
    {F}(u,U) = \int_U f_{{\rm bulk}}(x,\nabla u)\,{\rm d}x+\int_{J_u\cap U} f_{{\rm surf}}(x,[u],\nu_u)\, {\rm d}\Hd
\end{equation}
for every $(u,U)\in {\rm SBV}^2(\Omega;\Rm)\times \A(\Omega)$. Note that, by $(c)$, this integral representation yields an identical integral representation on every $U\in\A_{\rm reg}(\Omega)$  and $u\in L^1_{\rm loc}(\Omega;\Rm)\cap {\rm SBV}^2(U;\Rm)$. Indeed, if $U\in\A_{\rm reg}(\Omega)$, any function  $u\in{\rm SBV}^2(U;\Rm)$ can be extended to a function $\widetilde{u}\in {\rm SBV}^2(\Omega;\Rm)$ and by locality we then have $F(u,U)=F(\widetilde{u},U)$.

In order to extend the integral representation to ${\rm GSBV}^2$,  we begin  showing that \begin{equation}\label{bulkmono}
    |L_1|\leq |L_2|\implies f_{\rm bulk}(x,L_1)\leq  f_{\rm bulk}(x,L_2) 
\end{equation}
for all $x\in\Rd$, and
\begin{equation}\label{surfmono}
    |\zeta_1|\leq |\zeta_2|\implies f_{\rm surf}(x,\zeta_1,\nu)\leq  f_{\rm surf}(x,\zeta_2,\nu) 
\end{equation}
for all $x\in\Rd$ and $\nu\in \Sd$. We prove only the first inequality, the proof of the latter being obtained with an analogous argument.
Let us fix $L_1,L_2\in\R^{m\times d}$ and $r>0$ and let $v_2\in {\rm SBV}^2(Q(x,r);\Rm)$ with $v_2=L_2$ in a neighbourhood of $\partial Q(x,r)$. As $|L_1|\leq |L_2|$, there exists a $1$-Lipschitz map $\Phi\colon \Rm\to \Rm$ so that $v_1:=\Phi\circ v_2$ equals $L_1$ in a neighbourhood of $\partial Q(x,r)$ and therefore, by $(f)$, we obtain
\begin{equation*}
    \m^F(L_1,Q(x,r))\leq F(v_1,Q(x,r))\leq F(v_2,Q(x,r)).
\end{equation*}
From of the arbitrariness of $v_2$ and of $r$ we deduce \eqref{bulkmono}.

We are now ready to conclude the proof of the integral representation. Let us fix 
$U\in \A(\Omega)$ and $u\in {\rm GSBV^2}(U;\Rm)\cap L^1_{\rm loc}(\Omega;\Rm)$. Since $u^M\in {\rm SBV^2}(U;\Rm)\cap L^1_{\rm loc}(\Omega;\Rm)$ for every $M>0$, by \eqref{SBV repre} we obtain that 
\begin{equation}\label{integral M}
    {F}(u^M,U) = \int_{U} f_{{\rm bulk}}(x,\nabla u^M)\,{\rm d}x+\int_{J_{u^M}\cap U} f_{{\rm surf}}(x,[u^M],\nu_{u^M})\, {\rm d}\Hd.
\end{equation}
Combining $(b)$ with $(f)$, we deduce that 
\begin{equation*}
    \lim_{M\to+\infty}F(u^M,U)=F(u,U).
\end{equation*}
 Then, recalling Remark \ref{remark truncations}, \eqref{bulkmono} and \eqref{surfmono}, and using the Monotone Convergence Theorem, we let $M\to+\infty$ in \eqref{integral M} to get 
\begin{equation*}
    {F}(u,U) = \int_{U'} f_{{\rm bulk}}(x,\nabla u)\,{\rm d}x+\int_{J_{u}\cap U} f_{{\rm surf}}(x,[u],\nu_{u})\, {\rm d}\Hd,
\end{equation*}
which is the thesis.
\end{proof}

\section{Convergence of minima and \texorpdfstring{$\Gamma$}{Γ}-convergence}\label{sec:convergence of infima}
In this section we fix a positive sequence $\{\e_n\}_{n}$ converging to $0$ as $n\to +\infty$ and
 a sequence $\{F_{\e_n}\}_{n}$ given by \eqref{def functionals}, with $\{f_{\e_n}\}_n$ satisfying  conditions (i)-(iii) of Section \ref{section : functionals}.  We show that thanks to general properties of $\Gamma$-convergence and to results obtained in previous sections, the asymptotic behaviour of minimisation problems associated to $F_{\e_n}$ on small cubes determines the energy densities in the integral representation of the $\Gamma$-limit $F$. 
 
\medskip

We first state a compactness result for sequences with equi-bounded energy, originally proved by Gobbino in \cite[Theorem~5.4]{Gobbino} (see also \cite[Theorem~5.23]{BraidesApproximation}).
\begin{Lemma}\label{lemma comp gobbino}
Let $\{u_n\}_n\subset L^\infty(\Omega;\Rm)$, let $r>0$, and assume that 
\begin{equation}\notag 
    \sup_{n} \big(\|u_n\|_{L^\infty(\Omega;\Rm)}+G^r_{1,\e_n}(u_n,\Omega)\big)<+\infty,
\end{equation}
where $G^r_{1,\e_n}$ are the functionals defined by \eqref{Gobbino troncati}.
Then there exists a subsequence $\{\e_{n_k}\}_k$   and a function $u\in {\rm SBV}^2(\Omega;\Rm)$ such that $\{u_{\e_{n_k}}\}_k$ converges to $u$ in $L^1(\Omega;\Rm)$ as $k\to+\infty$.
\end{Lemma}

We now introduce the minimisation problems for the non-local functionals we are interested in. 
Given $U\in\A(\Omega)$, $w\in L^1_{\rm loc}(\Omega;\Rm)$, and $s>0$, we set 
\begin{equation*}
    \mathcal{D}_{s,w}(U):=\big\{u\in L^1_{\rm loc}(\Omega;\Rm): \, u=w \text{ for $\Ld$-a.e.\ $x\in U$ }  \text{with dist}(x,\Rd\setminus U)<s\big\},
\end{equation*}
and consider the minimisation problems
\begin{equation*}
   \m_s^{F_{\e_n}}(w,U):=\inf\big\{F_{\e_n}(u,U): u\in \mathcal{D}_{\e_n s,w}(U)
   \big\}.
\end{equation*}
We recall that
\begin{equation*}
    \m_{s_1}^{F_{\e_n}}(w,U)\leq \m_{s_2}^{F_{\e_n}}(w,U)
\end{equation*}
for every $0<s_1<s_2$.

The following result shows that, if $\{F_{\e_n}\}_n$ $\Gamma$-converges to $F$ on a regular open set $U\subset \Omega$ and if $V\subset \subset U$, the minimum value of $\m^F(w,U)$ can be controlled from above by the asymptotic value of the minimisation problems $\m^{F_{\e_n}}_s(w,V)$ and  $H(w,U\setminus V)$. 

\begin{Lemma}\label{lemma liminf minima}
Let  $w\in {\rm SBV}^2(\Omega; \Rm)\cap L^\infty(\Omega; \Rm)$, and let  $U,V\in\A_{\rm reg}(\Omega)$ with $V\subset \subset U$. Assume that $\{F_{\e_n}(\cdot,U)\}_{n}$ $\Gamma$-converges to $F(\cdot,U)$ in $L^1_{\rm loc}(\Omega;\Rm)$. Then 
     \begin{equation}\label{claim liminf minima}
         \m^F(w,U)\leq \sup_{s>0}\, \liminf_{n\to +\infty}\m^{F_{\e_n}}_s(w,V) + c_2H(w,U\setminus V)+\Lambda|U\setminus V|.
     \end{equation}
\end{Lemma}
\begin{proof}
 We consider a sequence of functions $\{u_{n}\}_{n}\subset \mathcal{D}_{\e_n s,w}(V)$  such that 
 \begin{equation}\label{recovery on V}
     F_{\e_{n}}(u_{n}, V) \leq \m^{F_{\e_{n}}}_s(w,V) + \e_{n},
 \end{equation}
 and which, by  truncation, we  may assume to be satisfying $\|u_{n}\|_{L^\infty(\Omega;\Rm)} \leq {\sqrt{m}}\|w\|_{L^\infty(\Omega;\Rm)}$.

Let us fix $s>2r_0$, where $r_0$ is the constant introduced in \eqref{non zero in zero}.  Thanks to \eqref{non zero in zero} and \eqref{comparison functionals}, from this inequality we obtain
\begin{equation*}
    c_0G^{r_0}_{\e_n}(u_n,V)\leq F_{\e_n}(u_{n}, V) \leq \m^{F_{\e_n}}_s(w,V) + \e_n\leq H_{\e_n}(w,V)+\|\eta_{\e_n}\|_{L^1(\R^d)}|V|+\e_n,
 \end{equation*}
 where $c_0$ is the positive constant introduced in \eqref{non zero in zero} and $G_{\e_n}^{r_0}$ is as in \eqref{characteristicfunctionals G}. Recalling \eqref{upper bound heps}, this chain of inequalities implies 
\begin{equation}\label{sup bounded G}  \sup_{n} G^{r_0}_{\e_n}(u_n,V) \leq\sup_{n} \frac{1}{c_0}\Big(H_{\e_n}(w,U)+\|\eta_{\e_n}\|_{L^1(\R^d)}|U|+1\Big)<+\infty
\end{equation}
upon assuming $\e_n<1$ for every $n\in\N$, with the right hand side being independent of $s>2r_0$.

We introduce the extension $\widetilde{u}_n$  of the function $u_n$ obtained by setting $
\widetilde{u}_n:=u_n$ on $V$ and 
$\widetilde{u}_n:=w$ on $U\setminus V$.  We observe that
\begin{gather}\notag 
    G^{r_0}_{\e_n}(\widetilde{u}_n,U)\leq G^{r_0}_{\e_n}({u}_n, V)+ G^{r_0}_{\e_n}\big(w, \big((U\setminus V)+B_{\e_nr_0}\big)\cap U\big),
\end{gather}
which by \eqref{sup bounded G} implies 
\begin{equation}\label{sup bounded G bis}
    \sup_{n} G^{r_0}_{\e_n}(\widetilde{u}_n,U) \leq \big(1+\frac{1}{c_0}\big)\sup_{n} \Big(H_{\e_n}(w,U)+\|\eta_{\e_n}\|_{L^1(\R^d)}|U|+1\Big)<+\infty.
\end{equation}
Observe that by Proposition \ref{prop: interpolation}, inequality  \eqref{sup bounded G bis} also implies that  
\begin{equation*}
    S:=\sup_{n} H^{r_0}_{\e_n}(\widetilde{u}_n,U)<+\infty.
\end{equation*}

We now pass to a subsequence $\{\e_{n_k}\}_k$ such that 
\begin{equation*}
\liminf_{n\to+\infty}\m^{F_{\e_{n}}}_s(w,V)=\lim_{k\to+\infty}\m^{F_{\e_{n_k}}}_s(w,V).
\end{equation*}
Thanks to \eqref{sup bounded G bis}, we may apply Lemma \ref{lemma comp gobbino} to obtain a further subsequence, not relabelled, such that $\{\widetilde{u}_{n_k}\}_k$ converges strongly in $L^1(U;\Rm)$ to a function $u\in {\rm SBV}^2(U;\Rm)$, which also satisfies the boundary conditions $u=w$ $\Ld$-a.e.\ in $U\setminus V$.

For every $T>0$ let us consider the functional $F_{\e_n}^T$  defined by \eqref{def truncated functionals}. Thanks to Lemma \ref{lemma F<FT}, there is a constant $C^*$, depending only on $S$ and on $\|w\|_{L^\infty(\Omega;\Rm)}$, and which is therefore independent of $s$, such that for every $\delta>0$  it holds 
\begin{equation*}
    F_{\e_n}(\widetilde{u}_n,U)\leq F_{\e_n}^T(\widetilde{u}_n,U)+C^*\delta
\end{equation*}
for every $T>T^*=T^*(\delta)$ and $n$ large enough. 
Hence, given $\delta>0$ and further assuming that $ s>T^*$, recalling that $u$ is a competitor for the minimisation problem $\m^F(w,U)$,  by $\Gamma$-convergence we have
 \begin{align}\notag
     \m^F(w,U)&\leq F(u,U)\leq \liminf_{k\to+\infty}F_{\e_{n_k}}(\widetilde{u}_{n_k},U)\leq \liminf_{k\to+\infty}\Big(F_{\e_{n_k}}^T(\widetilde{u}_{n_k},U)+C^*\delta\Big)\\\notag
     & \leq \liminf_{k\to+\infty}\Big(F_{\e_{n_k}}^T({u}_{n_k},V)+ F^T_{\e_{n_k}}\big(w, \big((U\setminus V)+B_{\e_{n_k}T})\big)\cap U\big)\Big)+C^*\delta,
 \end{align}
 for every $T\in (T^*,s]$.
 Using \eqref{comparison functionals} and \eqref{upper bound heps}, from this last chain of inequalities we obtain that for every $\sigma>0$
 \begin{equation*}
      \m^F(w,U)\leq \liminf_{k\to+\infty}F_{\e_{n_k}}({u}_{n_k},V)+ 
      c_2H\big(w, \big((U\setminus V)+B_{\sigma}\big)\cap U\big)+\Lambda|((U\setminus V)+B_{\sigma})\cap U)|+C^*\delta.
 \end{equation*}
 By arbitrariness of $\sigma>0$ and by \eqref{recovery on V} we get
 \begin{align*}
     \m^F(w,U)&\leq  \liminf_{k\to+\infty}F_{\e_{n_k}}({u}_{n_k},V)+ 
      c_2H(w, U\setminus V)+\Lambda|U\setminus V|+C^*\delta\\
      &\leq  \lim_{k\to+\infty}\m^{F_{\e_{n_k}}}_s(w,V)+ 
      c_2H(w, U\setminus V)+\Lambda|U\setminus V|+C^*\delta.
 \end{align*}
 
 Taking the limit as $s\to+\infty$ and as $\delta\to 0^+$, we obtain \eqref{claim liminf minima}, concluding the proof.
\end{proof}

\begin{Remark}\label{remark: cmp supp}
Following the above proof, it is easily seen that the statement of the previous lemma can be improved upon assuming that there exists $T>0$ such that $f_{\e_n}(x,\xi,z)=0$ for a.e $x\in \Omega, \xi\in \Rd\setminus B_T, z\in \Rm$, and $n\in\N$.   In this case, we obtain that 
\begin{equation*}
         \m^F(w,U)\leq  \liminf_{n\to +\infty}\m^{F_{\e_n}}_s(w,V) + c_2H(w,U\setminus V)+\Lambda|U\setminus V|
     \end{equation*}
for all $s \geq T$.
\end{Remark}

In the following lemma we show that,  if $\{F_{\e_n}\}_n$ $\Gamma$-converges to $F$ on a regular open set $U\subset \Omega$, the minimum value of $\m^F(w,U)$ is always larger than the asymptotic value of the minimisation problems $\m^{F_{\e_n}}_s(w,U)$. 
\begin{Lemma}\label{lemma limsup minima}
       Let  $w\in {\rm SBV}^2(\Omega; \Rm)\cap L^\infty(\Omega; \Rm)$, let $s>0$, and let  $U\in\A_{\rm reg}(\Omega)$. Assume that $\{F_{\e_n}(\cdot,U)\}_n$ $\Gamma$-converges to $F(\cdot,U)$ in $L^1_{\rm loc}(\Omega;\Rm)$ as $n\to+\infty$. Then
     \begin{equation}\label{claim limsup minima}
         \limsup_{n\to +\infty}\m^{F_{\e_n}}_s(w,U)\leq \m^F(w,U).
     \end{equation}
\end{Lemma}

\begin{proof}
    Since $w \in {\rm SBV}^2(\Omega;\Rm)$, we have that $H(w,U)<+\infty$, which by \eqref{comparison functionals} clearly implies  $\m^F(w,U)<+\infty$.
    Let us fix $\sigma>0$ and $u\in {\rm SBV}^2(U;\Rm)\cap L^1_{\rm loc}(\Omega;\Rm)$ such that $F(u,U) < \m^F(w,U) + \sigma$ and $u=w$ in a neighbourhood of $\partial U$. Observe that, since $F$ satisfies property $(f)$ of Proposition \ref{prop:compactness} on $U$, it is not restrictive to assume that $u\in L^\infty(\Omega;\Rm)$.
    
  We pass to a not relabelled subsequence such that the limsup in  \eqref{claim limsup minima} is actually a limit and consider a sequence $\{u_n\}_n\subset L^1_{\rm loc}(\Omega;\Rm)$, converging to $u$ in $L^1_{\rm loc}(\Omega;\Rm)$, such that 
    \begin{equation}\label{quasi minimo limsup}\lim_{n\to+\infty}F_{\e_n}(u_n,U)=F(u,U)<\m^F(w,U)+\sigma.
    \end{equation}
    By Lemma \ref{lemma:vertical truncations} we can also assume that $\|u_n\|_{L^\infty(\Omega;\Rm)}\leq {\sqrt{m}}\|u\|_{L^\infty(\Omega;\Rm)}$ for every $n\in\N$, so that, possibly  passing to a further subsequence, we can assume that $\{u_n\}_n$ converges to $u$ in $L^1(U;\Rm)$  by Lemma \ref{lemma comp gobbino}.

    Let $V\in\A_{\rm reg}(U)$ be obtained intersecting $U$ with a neighbourhood of $\partial U$ on which $u=w$ $\Ld$-a.e. and let $U'\subset \subset U''\subset \subset U$ be open sets such that $U\setminus U' \subset V$. Having fixed $T>0$ and $N\in\N$, we may argue as in the proof of Lemma \ref{fundamental estimate truncated} to find a sequence of cutoff functions $\{\phi_n\}_n$ compactly supported in $U''$ such that $\phi_n = 1$ in $\overline{U'}$ and, setting $v_{n,N}:= \phi_n u_n + (1-\phi_n)u$, we have 
     \begin{equation}\label{first bound limsup}
        F_{\eps_n}^T (v_{n,N}, U) \le F_{\eps_n}(u_n, U) + F_{\eps_n}(w, V) + \frac{C}{N-4}+ C(N^2+N)\|u_n - w\|_{L^1(V;\Rm)},
     \end{equation}
     for a positive constant $C>0$, independent of $n$ and $N$, and for $n$ large enough. Note that $\{v_{n,N}\}_n$ converges to $u$ in $L^1_{\rm loc}(\Omega;\Rm)$ as $n\to+\infty$ and that $\|v_{n,N}\|_{L^\infty(\Omega;\Rm)}\leq {2\sqrt{m}}\|u\|_{L^\infty(\Omega;\Rm)}$ for every $n\in\N$. Moreover, we observe that, since the supports of $\{\phi_n\}_n$  are compactly contained in $U''$, $U\setminus U'\subset V$, and $u=w$ $\Ld$-a.e.\ on $V$, we have   $v_{n,N}\in \mathcal{D}_{\e_ns,w}(U)$ for every $n$ large enough.

   Then,  by \eqref{comparison functionals} and \eqref{upper bound heps}  we get
     \begin{equation}\notag
         \limsup_{n\to +\infty} F_{\e_n}(w, V) \leq \limsup_{n \to +\infty} H_{\e_n}(w,V) + \Lambda |V| \leq c_2 H(w,V) + \Lambda|V|.
     \end{equation}
      Recalling that $\|u_n-w\|_{L^1(V;\Rm)}\to 0$ as $n\to+\infty$ and \eqref{quasi minimo limsup}, we may pass to the limsup as $n\to+\infty$ in \eqref{first bound limsup}  and obtain
     \begin{equation*}\limsup_{n\to + \infty}  F_{\eps_n}^T (v_{n,N}, U) \leq \m^F(w,U)+\sigma+ c_2H(w,V)+ \Lambda |V|+\frac{C}{N-4}\end{equation*}
     for every $N\in\N$, and by arbitrariness of $\sigma$ and $V$, we get
      \begin{equation}\label{leq inter}\limsup_{n\to + \infty}  F_{\eps_n}^T (v_{n,N}, U) \leq \m^F(w,U)+\frac{C}{N-4}.
      \end{equation}
     
    To conclude, we  show that for every $N\in\N$ we have
    \begin{equation}\label{to conclude limsup}
        \limsup_{n\to +\infty}\m^{F_{\e_n}}_{s}(w,U)\leq \sup_{T>0}\,\limsup_{n\to+ \infty} F_{\eps_n}^T (v_{n,N}, U).
    \end{equation}
To see this, we begin noting that from  
 \eqref{leq inter} and Lemma  \ref{lemma F<FT} it follows that 
\begin{equation*}
S:=\sup_{n}H^{r_0}_{\e_n}(v_{n,N},U)<+\infty
\end{equation*}
 and  that there is a positive constant $C^*$, depending only on  $S$, $N$, and $\|u\|_{L^\infty(\Omega;\Rm)}$,  such that for every  $\delta>0$  
\begin{equation}\notag 
    F_{\e_n}(v_{n,N},U)\leq F_{\e_n}^T(v_{n,N},U)+C^*\delta,
\end{equation}
 for every $T>T^*(\delta)$ and $n$ large enough.
Recalling that  $v_{n,N}\in \mathcal{D}_{\e_ns,w}(U)$, this implies that 
\begin{equation}\notag
\m^{F_{\e_n}}_s(w,U)\leq F_{\e_n}^T(v_{n,N},U)+C^*\delta
\end{equation}
for every $T>T^*(\delta)$. Passing to the limsup as $n\to+\infty$ and to the limit as $T\to+\infty$ and $\delta\to 0^+$, we obtain \eqref{to conclude limsup}.
Finally, combining \eqref{leq inter} and \eqref{to conclude limsup} and passing to the limit as $N\to+\infty$, we get \eqref{claim limsup minima}, concluding the proof.
\end{proof}

We recall the definition of the functions $f'_{\rm bulk},$ $f''_{\rm bulk}$, $f'_{\rm surf}$, and $f''_{\rm surf}$ given in Section \ref{section : functionals}. For every $x\in\Omega$, $L\in\Rmd$, $z\in\Rm$, and $\nu\in\Sd$ we let  
\begin{align}\label{f'bulk}
   f_{\rm bulk}'(x,L) & := \limsup_{r\to 0^+}\,\sup_{s>0}\,\liminf_{n\to+\infty}\frac{\m_s^{F_{\e_n}}( L, Q(x,r))}{r^d},\\
    f_{\rm bulk}''(x,L) & := \limsup_{r\to 0^+}\,\sup_{s>0}\,\limsup_{n\to+\infty}\frac{\m_s^{F_{\e_n}}( L, Q(x,r))}{r^d},\\
    f_{\rm surf}'(x,\zeta,\nu) & :=\limsup_{r\to 0^+}\,\sup_{s>0}\,\liminf_{n\to+\infty}\frac{\m_s^{F_{\e_n}}(u_{x,\zeta,\nu}, Q^\nu(x,r))}{r^{d-1}},\\\label{f''surf}
    f_{\rm surf}''(x,\zeta,\nu) & :=\limsup_{r\to 0^+}\,\sup_{s>0}\,\limsup_{n\to+\infty}\frac{\m_s^{F_{\e_n}}(u_{x,\zeta,\nu}, Q^\nu(x,r))}{r^{d-1}}.
\end{align}

The following two propositions prove Theorem \ref{theorem: minima}.

\begin{Proposition}\label{prop: necessary}
   Assume that there exists a functional $F\colon L^1_{\rm loc}(\Omega;\Rm)\times \A(\Omega)\to [0,+\infty]$ such that for every $U\in\A_{\rm reg}(\Omega)$ the sequence $\{F_{\e_n}(\cdot,U)\}_{n}$ $\Gamma$-converges to $F(\cdot,U)$ with respect the topology of $L^1_{\rm loc}(\Omega;\Rm)$ as $n\to+\infty$.
   Then the functions $f_{\rm bulk}$ and $f_{\rm surf}$ defined in \eqref{Gammabulk} and \eqref{Gammasurf}, respectively, satisfy
  \begin{gather}\label{first necessary}
    f_{\rm bulk}(x,L)=f'_{\rm bulk}(x,L)=f''_{\rm bulk}(x,L),\\\label{second necessary}
   f_{\rm surf}(x,\zeta,\nu)=f'_{\rm surf}(x,\zeta,\nu)=f''_{\rm surf}(x,\zeta,\nu),
\end{gather}
   for all $x\in\Omega$, $L\in\Rmd$, $\zeta\in\Rm$, and $\nu\in\Sd$.
\end{Proposition}
\begin{proof}
Let us fix $x\in\Omega$, $L\in\Rmd$, $\zeta\in\Rm$, and $\nu\in\Sd$. We begin proving \eqref{first necessary}. Given $r>0$, we set $\sigma:=r+r^2$ and apply Lemma \ref{lemma liminf minima} with $w=L$, $U=Q(x,\sigma)$, and $V=Q(x,r)$ to get 
\begin{equation}\notag
    \frac{\m^{F}(L,Q(x,\sigma))}{\sigma^d}\leq \sup_{s>0}\,\liminf_{n\to +\infty} \frac{\m^{F_{\e_n}}_s(L,Q(x,r))}{\sigma^d}+[c_2(\lambda_2|L|^2+\kappa|L|)+\Lambda]\Big(1-\frac{r^{d}}{\sigma^{d}}\Big).
\end{equation}
 Since $\sigma/r\to 1$ as $r\to 0^+$, from the previous inequality we get $f_{\rm bulk}(x,L)\leq f'_{\rm bulk}(x,L)$. The inequality $f''_{\rm bulk}(x,L)\leq f_{\rm bulk}(x,L)$ is obtained applying Lemma \ref{lemma limsup minima}.

To prove \eqref{second necessary}, we apply again Lemma \ref{lemma liminf minima} with $w=u_{x,\zeta,\nu}$, $U=Q^\nu(x,\sigma)$, and $V=Q^\nu(x,r)$ to get 
\begin{equation}\notag
    \frac{\m^{F}(u_{x,\zeta,\nu},Q^\nu(x,\sigma))}{\sigma^{d-1}}\leq \sup_{s>0}\, \liminf_{n\to +\infty} \frac{\m^{F_{\e_n}}_s(u_{x,\zeta,\nu},Q^\nu(x,r))}{\sigma^{d-1}} + c_2(\mu_2+\kappa|\zeta|) \Big(1-\frac{r^{d-1}}{\sigma^{d-1}}\Big) + \Lambda \Big(\frac{\sigma^d-r^{d}}{\sigma^{d-1}}\Big),
\end{equation}
Since $\sigma/r\to 1$ as $r\to 0^+$, from the previous  inequality we get $f_{\rm surf}(x,\zeta,\nu)\leq f'_{\rm surf}(x,\zeta,\nu)$. The inequality $f''_{\rm surf}(x,\zeta,\nu)\leq f_{\rm surf}(x,\zeta,\nu)$ follows from an application of Lemma \ref{lemma limsup minima}.
\end{proof}

\begin{Proposition}\label{prop: sufficiency}
 Assume that there exist functions $\hat{f}_{\rm bulk}\colon\Omega\times \R^{m\times d}\to [0,+\infty)$ and $\hat{f}_{\rm surf}\colon\Omega\times \Rm\times \Sd\to[0,+\infty)$ such that 
\begin{equation} \label{sufficient bulk}
    \hat{f}_{\rm bulk}(x,L)=f'_{\rm bulk}(x,L)=f''_{\rm bulk}(x,L),
\end{equation}
for $\Ld$-a.e.\ $x\in\Omega$ and for all $L\in\Rd$ and
\begin{equation}\label{sufficient surf}
    \hat{f}_{\rm surf}(x,\zeta,\nu)=f'_{\rm surf}(x,\zeta,\nu)=f''_{\rm surf}(x,\zeta,\nu)
\end{equation}
 for $\Hd$-a.e.\ $x\in\Omega$, for all $\zeta\in\Rm$ and $\nu\in\Sd$.
Then for every $U\in\A_{\rm reg}(\Omega)$ the sequence $\{F_{\e_n}(\cdot,U)\}_n$ $\Gamma$-converges  with respect to the topology of  $L^1_{\rm loc}(\Omega;\Rm)$ to the functional $F(\cdot,U)$ defined  by 
\begin{equation}\notag 
   \hspace{-0.15 cm} F(u,U):=\begin{cases}\displaystyle\int_U \!\!\hat{f}_{{\rm bulk}}(x,\nabla u){\rm d}x+\!\!\int_{J_u\cap U} \!\!\!\!\!\hat{f}_{{\rm surf}}(x, [u],\nu_u){\rm d}\Hd \!\!&\text{if }u\in L^1_{\text{\rm loc}}(\Omega;\Rm) {\cap}\, {\rm GSBV}^2(U;\Rm),\\
     +\infty \!\!&\text{if }u\in L^1_{\text{\rm loc}}(\Omega;\Rm) {\setminus}  {\rm GSBV}^2(U;\Rm). 
     \end{cases}
\end{equation}
\end{Proposition}
\begin{proof}
       Thanks to Theorem \ref{thm:compactness}, there exists a subsequence $\{\e_{n_k}\}_n$ such that 
for every open set $U\in\A_{\rm reg}(\Omega)$ the sequence $\{F_{\e_{n_k}}(\cdot,U)\}_k$ $\Gamma$-converges  in the $L^1_{\rm loc}(\Omega;\Rm)$ topology as $k\to+\infty$  to the functional  $F(\cdot,U)$  defined for $u\in L^1_{\rm loc}(\Omega;\Rm)$ by
\begin{equation}\notag 
   \hspace{-0.15 cm} F(u,U):=\begin{cases}\displaystyle\int_U \!\!{f}_{{\rm bulk}}(x,\nabla u){\rm d}x+\!\!\int_{J_u\cap U} \!\!\!\!\!{f}_{{\rm surf}}(x, [u],\nu_u){\rm d}\Hd \!\!&\text{if }u\in L^1_{\text{\rm loc}}(\Omega;\Rm) {\cap}\, {\rm GSBV}^2(U;\Rm),\\
     +\infty \!\!&\text{if }u\in L^1_{\text{\rm loc}}(\Omega;\Rm) {\setminus}  {\rm GSBV}^2(U;\Rm). 
     \end{cases}
\end{equation}
where $f_{\rm bulk}$ and $f_{\rm surf}$ are given by \eqref{Gammabulk} and \eqref{Gammasurf}, respectively.

By Proposition \ref{prop: necessary}, we note that 
\begin{align*}
    f_{{\rm  bulk}}(x,L) & = \limsup_{r\to 0^+}\,\sup_{s>0}\,\liminf_{k\to+\infty}\frac{\m_s^{F_{\e_{n_k}}}( L, Q(x,r))}{r^d} \\
    & \geq \limsup_{r\to 0^+}\,\sup_{s>0}\,\liminf_{n\to+\infty}\frac{\m_s^{F_{\e_n}}( L, Q(x,r))}{r^d} = \hat{f}_{\rm bulk}(x,L);
\end{align*}
and similarly
\begin{align*}
    f_{{\rm  bulk}}(x,L) & = \limsup_{r\to 0^+}\,\sup_{s>0}\,\limsup_{k\to+\infty}\frac{\m_s^{F_{\e_{n_k}}}( L, Q(x,r))}{r^d} \\
    & \leq \limsup_{r\to 0^+}\,\sup_{s>0}\,\limsup_{n\to+\infty}\frac{\m_s^{F_{\e_n}}( L, Q(x,r))}{r^d} = \hat{f}_{\rm bulk}(x,L),
\end{align*}
which implies $f_{{\rm  bulk}}(x,L)=\hat{f}_{\rm bulk}(x,L)$ for $\mathcal{L}^d$-a.e.\ $x\in\Omega$ and for every $L\in \Rmd$. A similar argument proves that $f_{{\rm  bulk}}(x,\zeta,\nu)=\hat{f}_{\rm surf}(x,\zeta,\nu)$ for $\mathcal{H}^{d-1}$-a.e.\ $x\in\Omega$ and for every $\zeta\in \Rm$ and $\nu\in \Sd$.

 Finally, since the functions $\hat{f}_{\rm bulk}$ and $\hat{f}_{\rm surf}$ are  by \eqref{sufficient bulk} and \eqref{sufficient surf} independent of the chosen subsequence, by  the Urysohn property of $\Gamma$-convergence (see \cite[Proposition~8.3]{DalBook}) the original sequence $\{F_{\e_n}(\cdot,U)\}$ $\Gamma$-converges to $F(\cdot,U)$ for every $U\in\A_{\rm reg}(\Omega)$, concluding the proof.
\end{proof}

\begin{Remark}\label{rmk: truncated}
 In view of Remark \ref{remark: cmp supp}, if for some $T>0$ it holds $f_{\e_n}(x,\xi,z) = 0$ for $\Ld$-a.e.\ $x\in \Omega,\, \xi \in \R^d\setminus B_T$ and for all $z \in \Rm$ and $n\in\N$, then in definitions \eqref{f'bulk}-\eqref{f''surf} we may neglect the supremum over $s>0$ and simply consider a fixed $s\geq T$, and have Proposition \ref{prop: necessary} and Proposition \ref{prop: sufficiency} hold unaltered.
\end{Remark}

We conclude mentioning a further characterization of the energy densities in the $\Gamma$-limit that involves functionals taking into account finite-range interactions.

\begin{Proposition}
    Let $\{T_j\}_j$ be a monotonically increasing sequence such that $T_j\to+\infty$ as $j\to+\infty$, and assume that for every $j\in\N$ and for every $U\in\A_{\rm reg}(\Omega)$ the sequence $\{F^{T_j}_{\e_n}(\cdot,U)\}_{n}$ $\Gamma$-converges  with respect the topology of $L^1_{\rm loc}(\Omega;\Rm)$ to the functional $F^{T_j}(\cdot,U)$ defined by 
\begin{equation*}\notag 
   \hspace{-0.15 cm} F^{T_j}(u,U):=\begin{cases}\displaystyle\!\int_U \!f^{T_j}_{{\rm bulk}}(x,\nabla u){\rm d}x+\!\!\int_{J_u\cap U} \!\!f^{T_j}_{{\rm surf}}(x, [u],\nu_u){\rm d}\Hd \!\!&\text{if }u\in L^1_{\text{\rm loc}}(\Omega;\Rm) {\cap}\, {\rm GSBV}^2(U;\Rm),\\
     +\infty \!\!&\text{if }u\in L^1_{\text{\rm loc}}(\Omega;\Rm) {\setminus}  {\rm GSBV}^2(U;\Rm)
     \end{cases}
\end{equation*}  
as $n\to+\infty$. Then, for every $U\in\A_{\rm reg}(\Omega)$ the sequence $\{F_{\e_n}(\cdot,U)\}_{n}$ $\Gamma$-converges with respect the topology of $L^1_{\rm loc}(\Omega;\Rm)$ to 
  \begin{equation*}\notag 
   \hspace{-0.15 cm} F(u,U):=\begin{cases}\displaystyle\int_U \!\!f_{{\rm bulk}}(x,\nabla u){\rm d}x+\!\!\int_{J_u\cap U} \!\!\!\!\!f_{{\rm surf}}(x, [u],\nu_u){\rm d}\Hd \!\!&\text{if }u\in L^1_{\text{\rm loc}}(\Omega;\Rm) {\cap}\, {\rm GSBV}^2(U;\Rm),\\
     +\infty \!\!&\text{if }u\in L^1_{\text{\rm loc}}(\Omega;\Rm) {\setminus}  {\rm GSBV}^2(U;\Rm)
     \end{cases}
\end{equation*}  
 as $n\to+\infty$.
 
 Moreover, it holds that
 \begin{align*}
     f_{{\rm bulk}}(x,L) & = \lim_{j\to+\infty} \limsup_{r\to 0^+}\liminf_{n\to+\infty}\frac{\m_{T_j}^{F^{T_j}_{\e_n}}( L, Q(x,r))}{r^d} \\
     & = \lim_{j\to+\infty} \limsup_{r\to 0^+}\limsup_{n\to+\infty}\frac{\m_{T_j}^{F^{T_j}_{\e_n}}( L, Q(x,r))}{r^d}
 \end{align*} 
 and 
 \begin{align*}
     f_{\rm surf}(x,\zeta,\nu) & =\lim_{j\to+\infty}\limsup_{r\to 0^+}\liminf_{n\to+\infty}\frac{\m_{T_j}^{F^{T_j}_{\e_n}}(u_{x,\zeta,\nu}, Q^\nu(x,r))}{r^{d-1}} \\
     & = \lim_{j\to+\infty}\limsup_{r\to 0^+}\limsup_{n\to+\infty}\frac{\m_{T_j}^{F^{T_j}_{\e_n}}(u_{x,\zeta,\nu}, Q^\nu(x,r))}{r^{d-1}}.
 \end{align*}
for all $x\in\Omega$, $L\in\Rmd$, $\zeta\in\Rm$, and $\nu\in\Sd$.
\end{Proposition}
\begin{proof}
    The first part of the statement follows by Proposition \ref{proposition: truncated functionals}. Since $\{f^{T_j}_{\rm bulk}\}_j$ and $\{f^{T_j}_{\rm surf}\}_j$ are sequences of monotonically increasing functions, we have that
    \begin{equation}\label{eq: auxiliary infima 1}
       f_{\rm bulk} = \lim_{j\to+\infty} f^{T_j}_{\rm bulk}, \qquad f_{\rm surf} = \lim_{j\to+\infty} f^{T_j}_{\rm surf}.
    \end{equation}
    Moreover, applying Proposition \ref{prop: necessary} and Remark \ref{rmk: truncated}, we obtain
     \begin{equation}\label{eq: auxiliary infima 2}
     f^{T_j}_{{\rm bulk}}(x,L) = \limsup_{r\to 0^+}\liminf_{n\to+\infty}\frac{\m_{T_j}^{F^{T_j}_{\e_n}}( L, Q(x,r))}{r^d} =  \limsup_{r\to 0^+}\limsup_{n\to+\infty}\frac{\m_{T_j}^{F^{T_j}_{\e_n}}( L, Q(x,r))}{r^d}
 \end{equation} 
  and 
 \begin{equation}\label{eq: auxiliary infima 3}
     f_{\rm surf}^{T_j}(x,\zeta,\nu) =\limsup_{r\to 0^+}\liminf_{n\to+\infty}\frac{\m_{T_j}^{F^{T_j}_{\e_n}}(u_{x,\zeta,\nu}, Q^\nu(x,r))}{r^{d-1}} = \limsup_{r\to 0^+}\limsup_{n\to+\infty}\frac{\m_{T_j}^{F^{T_j}_{\e_n}}(u_{x,\zeta,\nu}, Q^\nu(x,r))}{r^{d-1}}
 \end{equation}
for all $x\in\Omega$, $L\in\Rmd$, $\zeta\in\Rm$, and $\nu\in\Sd$. Therefore, the thesis follows by \eqref{eq: auxiliary infima 1}, \eqref{eq: auxiliary infima 2}, and \eqref{eq: auxiliary infima 3}.
\end{proof}

\section*{Acknowledgements}

The authors wish to thank Andrea Braides for having proposed the problem during his course {\em \lq\lq Concentration problems for nonlocal energies\rq\rq} held at SISSA.  The authors are members of Gruppo Nazionale per l'Analisi Matematica, la Probabilità e le loro Applicazioni (GNAMPA) of Istituto Nazionale di Alta Matematica (INdAM).  Davide Donati gratefully acknowledges the support of the GNAMPA Project ``Asymptotic analysis of nonlocal variational problems'' funded by INdAM, and thanks Roberto Alicandro and Chiara Leone for useful discussions on the topic of this paper.

\bibliographystyle{siam}
\bibliography{bibliography}
\end{document}